\documentclass[a4paper,10pt]{amsart}

\usepackage[utf8x]{inputenc}
\usepackage{amssymb}
\usepackage{amsthm}
\usepackage{amsmath}
\usepackage{amsfonts}
\usepackage{bbm}
\usepackage{graphicx}
\usepackage[]{overpic}
\usepackage{import}
\usepackage[dvips]{epsfig}
\usepackage[english]{babel}
\usepackage{array}
\graphicspath{{images/}}
\usepackage[a4paper]{geometry}
\usepackage{color}
\usepackage{setspace}
\usepackage{fancybox}
\usepackage{hyperref}
\usepackage{pinlabel} 
\usepackage{xcolor}
\definecolor{green}{RGB}{0,150,0}
\usepackage{mathabx}
\usepackage{setspace}

\newcommand{\N}{\mathbb N}
\newcommand{\Z}{\mathbb Z} 

\newcommand{\R}{\mathbb R}

\newcommand{\Ac}{\mathcal{A}}
\newcommand{\Bc}{\mathcal{B}}
\newcommand{\Cc}{\mathcal{C}}

\newcommand{\cM}{\mathcal{M}}

\newcommand{\Rc}{\mathcal{R}}

\newcommand{\bs}{\boldsymbol}
\newcommand{\ep}{\varepsilon}
\newcommand{\La}{\Lambda}
\newcommand{\Aug}{\mathcal{A}ug}

\DeclareMathOperator{\id}{id}

\DeclareMathOperator{\mfm}{\mathfrak{m}}

\DeclareMathOperator{\CZ}{CZ}
\DeclareMathOperator{\RHom}{RHom}
\DeclareMathOperator{\Hom}{Hom}
\DeclareMathOperator{\CY}{CY}
\DeclareMathOperator{\mcCY}{\mathcal{C}\mathcal{Y}}
\DeclareMathOperator{\bd}{\bs{\delta}}
\DeclareMathOperator{\Cone}{Cone}
\DeclareMathOperator{\RFC}{RFC}
\DeclareMathOperator{\ba}{\textbf{b}}
\DeclareMathOperator{\D}{\Delta}

\newtheorem{lem}{Lemma}[section]
\newtheorem{teo}{Theorem}[section]

\newtheorem{prop}{Proposition}[section]
\theoremstyle{definition}
\newtheorem{defi}{Definition}[section]
\newtheorem{ex}{Example}[section]
\newtheorem{rem}{Remark}[section]
\newtheorem{nota}{Notations}[section]
\begin{document}
\title[Calabi-Yau structure of a Legendrian sphere]{Calabi-Yau structure on the Chekanov-Eliashberg algebra of a Legendrian sphere}
\author{No\'emie Legout}
\address{Department of Mathematics\\
	Uppsala University\\
	Box 480\\
	SE-751 06 Uppsala\\
	Sweden}
\email{noemie.legout@math.uu.se}
\date{}
\maketitle

\begin{abstract}
	In this paper, we prove that the Chekanov-Eliashberg algebra of an horizontally displaceable $n$-dimensional Legendrian sphere in the contactisation of a Liouville manifold is a $(n+1)$-Calabi-Yau differential graded algebra. In particular it means that there is a quasi-isomorphism of DG-bimodules between the diagonal bimodule and the inverse dualizing bimodule associated to the Chekanov-Eliashberg algebra. On some cyclic version of these bimodules, which are chain complexes computing the Hochschild homology and cohomology of the Chekanov-Eliashberg algebra, we construct $A_\infty$ operations and show that the Calabi-Yau isomorphism extends to a family of maps satisfying the $A_\infty$-functor equations.
\end{abstract}
\tableofcontents

\section{Introduction}

We consider $n$-dimensional Legendrian submanifolds in a contact manifold $(Y,\alpha)$ which is the contactisation of a Liouville manifold $(P,\beta)$. Among the numerous invariants of Legendrians up to Legendrian isotopy, many of them are derived from the famous Chekanov-Eliashberg algebra (C-E algebra) (\cite{Che,E,EES1,EES2}). This latter is a unital differential graded algebra (DGA) associated to a Legendrian $\La$ and generated by Reeb chords of $\La$, see Section \ref{CEalg} for a brief recall of the definition. In this paper we take the coefficient ring to be the field $\Z_2$.

It is well known that linearized versions of the C-E algebra satisfy a particular type of duality, which was first proved by Sabloff \cite{Sabloff} for Legendrian knots, and generalized to higher dimensions by Ekholm-Etnyre-Sabloff \cite{EESa} and to the bilinearized case by Bourgeois-Chantraine \cite{BCh}. We prove in this paper that under some assumptions, the full C-E algebra, i.e. not linearized, also satisfies a similar type of duality. This duality is expressed in terms of differential graded bimodules (DG-bimodules).
Namely, if $\Ac$ denotes the C-E algebra of a $n$-dimensional Legendrian sphere, then $\Ac$ is a DG $\Ac$-bimodule called the \textit{diagonal bimodule}, and the \textit{inverse dualizing bimodule} of $\Ac$ is the DG $\Ac$-bimodule $\Ac^!:=\RHom_{\Ac-\Ac}(\Ac,\Ac\otimes\Ac)$. We prove that there is a quasi-isomorphism
\begin{alignat}{1}
\mcCY:\Ac\xrightarrow{\simeq}\Ac^![-n-1]\label{rel:CYiso}
\end{alignat}
satisfying $\mcCY\simeq\mcCY^![-n-1]$.
Together with the fact that $\Ac$ is a \textit{homologically smooth} differential graded algebra (see Section \ref{sec:background}), this leads to the following result:
\begin{teo}\label{teointro1}
	Let $\La\subset Y$ be an horizontally displaceable Legendrian sphere. Then its C-E-algebra is a $(n+1)$-Calabi-Yau differential graded algebra.
\end{teo}
The definition of Calabi-Yau structure we use here is the one of Ginzburg \cite{Gin} (see also \cite{KS}), although we have some opposite sign convention for degrees.\\

For the following reasons, it was reasonable to expect that the C-E algebra of a displaceable Legendrian sphere admits a Calabi-Yau structure.
In his thesis, Ganatra \cite{Ganatra:thesis} showed that the wrapped Fukaya category $\mathcal{W}$ of a Weinstein manifold is a non-compact Calabi-Yau category, where the Calabi-Yau structure is induced similarly as for \eqref{rel:CYiso} by an equivalence of $A_\infty$ bimodules between the diagonal bimodule $\mathcal{W}$ and an analogue of the inverse dualizing bimodule in the $A_\infty$ setting, $\mathcal{W}^!$.
Now, observe that the wrapped Fukaya category of a Weinstein manifold $X$ (of finite type) is generated by the Lagrangian cocores obtained by attaching critical handles to Legendrian spheres in the ideal contact boundary of $X$, see \cite{CDGG:generation,GPS:generation}.
A Calabi-Yau structure on the C-E algebra of these Legendrian spheres should then be possible to construct by using the surgery isomorphism \cite{BEE, EL, Ekholm:surgery}. Recall that the surgery isomorphism gives an $A_\infty$ quasi-isomorphism between the wrapped Floer cohomology of the cocores, i.e. the endomorphism groups of the generators of the wrapped Fukaya category, and the C-E algebra of the Legendrian attaching spheres. 
Note that $X$ is taken here to be a subcritical Weinstein manifold. When $X$ is the standard ball, the C-E algebra of a Legendrian in its boundary can be computed inside a contact Darboux ball where it is displaceable, see for example \cite{D:refined}.

This suggests that a Calabi-Yau structure on the C-E algebra of a Legendrian sphere could potentially be obtained from the one defined by Ganatra via the surgery isomorphism. However we will not adopt this method in this paper.
Instead, we restrict ourselves to Legendrian spheres in the contactization of a Liouville manifold and introduce a version of the Rabinowitz Floer homology for Legendrians with coefficients in the free $\Ac$-bimodule $\Ac\otimes\Ac$ of rank $1$, where $\Ac$ denotes the C-E algebra.

Rabinowitz Floer homology was originally defined as an homology theory for contact type hypersurfaces in \cite{CF}. A relative theory, Lagrangian Rabinowitz Floer homology, has then been introduced in \cite{Merry} for exact Lagrangians in a Liouville manifold, and more recently both the non-relative and relative theories have been generalized to the case of Liouville cobordisms admitting a filling in \cite{CO}. These previous theories were defined in the Hamiltonian setting and in \cite{L2} we introduced an SFT-type version of the Lagrangian Rabinowitz Floer homology for Lagrangians in a (trivial) Liouville cobordism. There, we used augmentations of the C-E algebras of Legendrians in the negative ends of Lagrangian cobordisms in order to define a complex over $\Z_2$. In this paper, we consider only the Rabinowitz complex of cylinders over a Legendrian submanifold and we don't use augmentations so that we get a DG-bimodule with coefficients in the C-E algebras of the Legendrians.

More precisely, the Rabinowitz DG-bimodule considered here is generated by mixed chords of a $2$-copy $\La_0\cup\La_1$ of a Legendrian $\La$, where $\La_1$ is a small negative push-off of $\La_0:=\La$, and its differential is defined by a count of pseudo-holomorphic discs with boundary on $\R\times(\La_0\cup\La_1)$.  The differential is lower triangular and so the Rabinowitz bimodule is the cone of a DG-bimodule map $C_+(\La_0,\La_1)\to C_-(\La_0,\La_1)$, where $C_+$ is generated by chords from $\La_0$ to $\La_1$ (in bijective correspondence with chords of $\La$) and $C_-$ is generated by chords from $\La_1$ to $\La_0$ (in bijective correspondence with chords of $\La$ and critical points of a Morse function on $\La$). For a Legendrian sphere $\La$ and its C-E algebra $\Ac$, the Rabinowitz bimodule can be described as the cone of a slightly different DG bimodule map
\begin{alignat*}{1}
	\CY:\widehat{C}_+(\La_0,\La_1)\to\widecheck{C}_-(\La_0,\La_1)
\end{alignat*}
which corresponds to a slight modification of the action filtration. Here $\widehat{C}_+(\La_0,\La_1)$ is generated by Reeb chords from $\La_0$ to $\La_1$ as well as the maximum of a given Morse function on $\La$, and $\widecheck{C}_-(\La_0,\La_1)$ is generated by chords from $\La_1$ to $\La_0$ and the minimum of the Morse function. Then, we show that $\widehat{C}_+(\La_0,\La_1)[-n-1]$ is quasi-isomorphic to $\Ac$ and $\widecheck{C}_-(\La_0,\La_1)$ is quasi-isomorphic to $\RHom_{\Ac-\Ac}(\Ac,\Ac\otimes\Ac)$. 
The invariance of the Rabinowitz homology up to Legendrian isotopies implies that if $\La$ is horizontally displaceable, then the complex is acyclic. The shifted map $\CY[-n-1]$ provides thus a Calabi-Yau quasi-isomorphism $\Ac\to\Ac^![-n-1]$. \\

In general, we expect that the Rabinowitz complex is acyclic for any Legendrian in the contact boundary of a subcritical Weinstein manifold. Note that for example the periodic orbit version of Rabinowitz Floer homology and Rabinowitz Floer homology for fillable Legendrians both vanish there, see \cite{CO}.
We expect also that the Calabi-Yau structure we construct here coincides with that constructed by Ganatra in \cite{Ganatra:thesis}, and plan to show this in future work with Asplund. One advantage with the perspective taken here compared to that in Ganatra is from the point-of-view of computability. The pseudo-holomorphic discs that define the operations we consider can be computed using Ekholm’s theory of gradient flow trees \cite{Ekholm:trees}, while the operations defined by Hamiltonian perturbations seem to lack general techniques for computation. \\

By taking bimodule tensor products of both bimodules $\widehat{C}_+(\La_0,\La_1)$ and $\widecheck{C}_-(\La_0,\La_1)$ with the diagonal bimodule $\Ac$ one gets complexes which we denote $\widehat{C}_+^{cyc}(\La_0,\La_1)$ and $\widecheck{C}_-^{cyc}(\La_0,\La_1)$ and which compute the Hochschild homology and cohomology of $\Ac$ respectively. On these complexes we construct higher order maps, i.e. for any $(d+1)$-copy $\La_0\cup\dots\cup\La_d$ we construct maps
\begin{alignat*}{1}
	&\widehat{\mfm}_d:\widehat{C}_+^{cyc}(\La_{d-1},\La_d)\otimes\dots\otimes\widehat{C}_+^{cyc}(\La_0,\La_1)\to\widehat{C}_+^{cyc}(\La_0,\La_d)\\
	&\widecheck{\mfm}_d:\widecheck{C}_-^{cyc}(\La_{d-1},\La_d)\otimes\dots\otimes\widecheck{C}_-^{cyc}(\La_0,\La_1)\to\widecheck{C}_-^{cyc}(\La_0,\La_d)
\end{alignat*}
satisfying the $A_\infty$-equations, see Section \ref{sec:Product}. The maps $\widecheck{\mfm}_d$ are defined by a count of pseudo-holomorphic discs with boundary on $\R\times(\La_0\cup\dots\cup\La_d)$, with $d$ negative asymptotics which are inputs and one positive asymptotic which is the output. These maps are well-known, they compute the $A_\infty$-structure of the augmentation category $\Aug_-(\La)$ \cite{BCh} with a formal unit added (corresponding to the minimum of the Morse function here). One the other side, the maps $\widehat{\mfm}_d$ are defined by a count of certain $2$ levels pseudo-holomorphic buildings appearing in the boundary of the compactification of $1$-dimensional moduli spaces. In this sense it can be seen as a \textit{secondary type} product. 
Finally, we show that the map $\CY$ induced on the cyclic complexes extends to a family of maps 
\begin{alignat*}{1}
	\CY_d:\widehat{C}_+^{cyc}(\La_{d-1},\La_d)\otimes\dots\otimes\widehat{C}_+^{cyc}(\La_0,\La_1)\to\widecheck{C}_-^{cyc}(\La_0,\La_d)
\end{alignat*}
satisfying the $A_\infty$-functor equations. 

Observe that the map $\CY$ gives an isomorphism between the Hochschild homology and cohomology of $\Ac$, which after generalizing our definition of the Rabinowitz complex to Legendrians in more general contact manifolds, would recover the quasi-isomorphism between Hochschild homology and cohomology for the wrapped Fukaya category in \cite{Ganatra:thesis}. Moreover we presumably recover the relation between the different product structures as constructed by Bourgeois-Ekholm-Eliashberg in \cite{BEE:product}.\\

\textit{Acknowledgment:} I am very grateful to Georgios Dimitroglou Rizell for all the discussions we had together and his generosity in sharing ideas.
I also thank Paolo Ghiggini who gave me valuable comments on a first version of the paper. In particular he pointed out a remark which I was using in some proof but which was wrong. 
This work is supported by the Knut and Alice Wallenberg Foundation under the grants KAW 2021.0191 and KAW 2021.0300, and by the Swedish Research Council under the grant number 2020-04426.
 
\section{Background on DG-bimodules}\label{sec:background}

Let $(\Ac_0,\partial_{\Ac_0})$, $(\Ac_1,\partial_{\Ac_1})$ be unital differential graded algebras (DGAs) over $\Z_2$ (we restrict to $\Z_2$ for simplicity here but it is not strictly necessary). A \textit{DG $\Ac_1-\Ac_0$-bimodule} is a graded $\Ac_1-\Ac_0$-bimodule $\Bc$ endowed with a degree $1$ differential $\partial_\Bc$ such that $\partial_\Bc(\bs{\alpha}_1b\bs{\alpha}_0)=\partial_{\Ac_1}(\bs{\alpha}_1)b\bs{\alpha}_0+\bs{\alpha}_1\partial_\Bc(b)\bs{\alpha}_0+\bs{\alpha}_1b\partial_{\Ac_0}(\bs{\alpha}_0)$.	 
If $(\Ac_0,\partial_{\Ac_0})=(\Ac_1,\partial_{\Ac_1})=(\Ac,\partial_\Ac)$, we write simply $\Ac$-bimodule instead of $\Ac-\Ac$-bimodule.

\begin{ex}
	\begin{enumerate}
		\item A DGA $(\Ac,\partial_\Ac)$ is itself a DG $\Ac$-bimodule, called the \textit{diagonal bimodule}.
		\item The tensor product of a DGA $\Ac$ with itself over $\Z_2$ is also a DG $\Ac$-bimodule, with differential given by $\partial_{\Ac\otimes\Ac}(a_1\otimes a_0)=\partial_\Ac(a_1)\otimes a_0+a_1\otimes\partial_\Ac(a_0)$. Note that this bimodule carries two different bimodule structures, the \textit{outer} and the \textit{inner} bimodule structure:
		$$\begin{array}{cccc}
			\mu_{\Ac\otimes\Ac}^{out}:&\Ac\otimes (\Ac\otimes\Ac)\otimes\Ac&\to& \Ac\otimes\Ac\\
			&\bs{\alpha}\otimes (a\otimes a')\otimes\bs{\alpha}'&\mapsto&\bs{\alpha}a\otimes a'\bs{\alpha}'
		\end{array}$$
		$$\begin{array}{cccc}
			\mu_{\Ac\otimes\Ac}^{in}:&\Ac\otimes (\Ac\otimes\Ac)\otimes\Ac&\to& \Ac\otimes\Ac\\
			&\bs{\alpha}\otimes (a\otimes a')\otimes\bs{\alpha}'&\mapsto&a\bs{\alpha}\otimes\bs{\alpha}'a'
		\end{array}$$
		\item Given two DG $\Ac_1-\Ac_0$-bimodules $\Bc$ and $\Cc$, the set $\Hom_{\Ac_1-\Ac_0}(\Bc,\Cc)$ of bimodule maps is a DG $\Ac_1-\Ac_0$-bimodule whose differential is given by $D(\phi)=\phi\circ\partial_\Bc+\partial_\Cc\circ\phi$.
	\end{enumerate}
\end{ex}

In this paper, we denote $-\otimes-$ instead of  $-\otimes_{\Z_2}-$ for the tensor product over $\Z_2$.
A \textit{DG-morphism} of $\Ac_1-\Ac_0$-bimodules from $\Bc$ to $\Cc$ is a degree $0$ element $\phi$  in $\Hom_{\Ac_1-\Ac_0}(\Bc,\Cc)$ which commutes with the differentials of $\Bc$ and $\Cc$; in other words, it is a degree $0$ cycle in $\big(\Hom_{\Ac_1-\Ac_0}(\Bc,\Cc),D\big)$. A \textit{quasi-isomorphism} of DG-bimodules is a DG-morphism which induces an isomorphism in homology.

Following \cite{FHT}, we recall that a DG $\Ac_1-\Ac_0$-bimodule $(\Bc,\partial_\Bc)$ is \textit{free} if it is isomorphic to $\Ac_1\otimes V\otimes\Ac_0$ where $V$ is a $\Z_2$-vector space generated by cycles. In this case a \textit{free generating set} for $\Bc$ is a basis of $V$.
Given a DGA $\Ac$ the diagonal bimodule is not a free bimodule while $\Ac\otimes\Ac$ endowed with either the outer or the inner structure is free, generated by $1\otimes1$.
A DG-bimodule $\Bc$ is called \textit{semifree} if there is a filtration
\begin{alignat*}{1}
	\{0\}\subset F_0\Bc\subset F_1\Bc\subset\dots\subset\Bc
\end{alignat*}
such that $F_i\Bc$ is a DG-sub-bimodule for all $i\geq0$, $\bigcup F_i\Bc=\Bc$, and $F_0\Bc$ and $F_{i+1}\Bc/F_i\Bc$ are free bimodules. We say that a semifree DG-bimodule $\Bc$ as above is of \textit{finite rank} if there is a $k\in\N$ such that $F_k\Bc=\Bc$.
A \textit{semifree resolution} of a DG-bimodule $\Bc$ is a semifree DG bimodule $R_\Bc$ together with a quasi-isomorphism of DG-bimodules $R_\Bc\to\Bc$.\\

In the category of DG $\Ac_1-\Ac_0$-bimodules, $\RHom_{\Ac_1-\Ac_0}(-,\Cc)$ denotes the right derived functor of the functor $\Hom_{\Ac_1-\Ac_0}(-,\Cc)$. Let $R_\Bc$ be any semifree resolution of the $\Ac_1-\Ac_0$-bimodule $\Bc$, we have by definition that $\RHom_{\Ac_1-\Ac_0}(\Bc,\Cc)=\Hom_{\Ac_1-\Ac_0}(R_\Bc,\Cc)$, which is well defined up to quasi-isomorphism.
Given a DGA $\Ac$, the \textit{inverse dualizing bimodule}, denoted $\Ac^!$, is $\RHom_{\Ac-\Ac}(\Ac,\Ac\otimes\Ac)$. The elements of $\RHom_{\Ac-\Ac}(\Ac,\Ac\otimes\Ac)$ are bimodule morphisms from the diagonal bimodule $\Ac$ to $\Ac\otimes\Ac$ where $\Ac\otimes\Ac$ is endowed with the inner bimodule structure.
Then, $\RHom_{\Ac-\Ac}(\Ac,\Ac\otimes\Ac)$ is a DG $\Ac$-bimodule with a bimodule structure induced by the outer bimodule structure on $\Ac\otimes\Ac$, i.e. for $\phi\in\RHom_{\Ac-\Ac}(\Ac,\Ac\otimes\Ac)$ which can be written as $\phi=\phi_l\otimes\phi_r$ and for any $\bs{v},\bs{a}, \bs{w}\in\Ac$ we have
\begin{alignat*}{1}
	(\bs{v}\cdot\phi\cdot\bs{w})(\bs{a})=\bs{v}\phi_l(\bs{a})\otimes\phi_r(\bs{a})\bs{w}
\end{alignat*}
Note that if $\phi\colon\Bc\to\Cc$ is a DG-morphism of $\Ac$-bimodules in the derived sense, then there is an induced morphism $\phi^!\colon\Cc^!\to\Bc^!$ given by $\phi^!(c^!)=c^!\circ\phi$.
\begin{rem}
	The semifree resolution mentioned above in order to compute $\RHom_{\Ac_1-\Ac_0}(\Bc,\Cc)$ can be taken to be the bar resolution of $\Bc$; however this one is not of finite rank. In this paper we will be interested in the inverse dualizing bimodule $\RHom_{\Ac-\Ac}(\Ac,\Ac\otimes\Ac)$ of the Chekanov-Eliashberg DGA (see Section \ref{CEalg}) of a closed Legendrian submanifold. In this case, the diagonal bimodule $\Ac$ admits a finite rank semifree resolution, see Section \ref{subsec:CYiso}.
\end{rem}

The purpose of this paper is to prove that under some hypothesis the Chekanov-Eliashberg algebra of a Legendrian sphere is a Calabi-Yau DGA. Let us recall the necessary definition.
\begin{defi}\cite{KS}
	Let $(\Ac,\partial_\Ac)$ be a DGA.
	\begin{enumerate}
		\item A DG $\Ac$-bimodule is \textit{perfect} if it is quasi-isomorphic to a direct summand of a finite dimensional semifree DG-bimodule.
		\item The DGA $\Ac$ is \textit{homologically smooth} if it is perfect as a DG $\Ac$-bimodule.
	\end{enumerate}	
\end{defi}
\begin{defi}\label{defi:CY}\cite{Gin}
	A homologically smooth DGA $\Ac$ is called $d$-\textit{Calabi-Yau} if there is a quasi-isomorphism of DG $\Ac$-bimodules
\begin{alignat*}{1}
	\phi:\Ac\to\Ac^![-d]
\end{alignat*}
such that $\phi\simeq\phi^![-d]$.
\end{defi}
\begin{nota}
	The convention we use for shifts in this paper is that if $|a|$ is the degree of an element of a DG-bimodule $\Bc$, then the same element viewed in $\Bc[d]$ has degree $|a|+d$. Then, for a DG-bimodule morphism $f:\Bc\to\Cc$ we denote $f[d]:\Bc[d]\to\Cc[d]$ the shifted map.
\end{nota}

\section{Moduli spaces}\label{sec:moduli}
We will be working in the same setting as in \cite{L2} and refer to Sections 2.2-2.6 of the mentioned paper and references therein for more details about the moduli spaces of pseudo-holomorphic discs we consider here. Throughout this paper, when we consider a Legendrian submanifold of $Y$ we always assume that it is closed and non-degenerate in the sense that it admits a finite number of Reeb chords which are isolated and correspond to transverse intersection points of the Lagrangian projection on $P$.

Let $\La\subset Y$ be a Legendrian and denote $\Rc(\La)$ the set of Reeb chords of $\La$. For any chord $\gamma\in\Rc(\La)$ we denote $\CZ(\gamma)$ its Conley-Zehnder index, see \cite{EES1}.
\begin{rem}
		In case $\La$ is not connected, there are additional choices (as for example choices of paths between the various connected components) to make in order to define the Conley-Zehnder index of a Reeb chord connecting two distinct connected components, see \cite{DR}. The index of a chord depends on these additional choices but for any two chords from a connected component to another, the difference in index is independent of the choices. Moreover, we assume in this paper that the Legendrian we consider always have Maslov number $0$ and that the first Chern class of $P$ vanishes. In this case we get a well defined $\Z$-valued Conley-Zehnder index (after the potentially additional choices discussed above).
\end{rem}
Let $J_P$ be an almost complex structure on $(P,\beta)$ which is compatible with $d\beta$ and cylindrical outside of a compact set in the cylindrical end of $P$. We call such a structure \textit{admissible}.
Then, we denote $J$ the almost complex structure on $(\R\times Y,d(e^t\alpha))$ which is the cylindrical lift of $J_P$, i.e. the unique cylindrical almost complex structure $J$ on $\R\times Y$ such that the projection $\pi_P:\R\times P\times\R\to P$ is $(J,J_P)$-holomorphic.

Given Reeb chords $\gamma, \gamma_1,\dots,\gamma_d$ of $\La$ we denote
\begin{alignat*}{1}
	\widehat{\cM}_\La(\gamma;\gamma_1,\dots,\gamma_d)
\end{alignat*}
the moduli space of $J$-holomorphic discs with boundary on $\R\times\La$ having a positive asymptotic at $\gamma$ and negative asymptotics at $\gamma_1,\dots,\gamma_d$. As the boundary condition for these discs is cylindrical as well as the almost complex structure, there is an action of $\R$ by translation on these moduli spaces. We denote
\begin{alignat*}{1}
	\cM_\La(\gamma;\gamma_1,\dots,\gamma_d)=:\widehat{\cM}_\La(\gamma;\gamma_1,\dots,\gamma_d)/\R
\end{alignat*}
the quotient by this action. 
 
Let $\La_0,\dots,\La_d\subset Y$ be $d+1$ Legendrian submanifolds such that the link $\La_0\cup\dots\cup\La_d$ is non degenerate. For any $0\leq i\neq j\leq d$, we denote $\Rc(\La_i,\La_j)$ the set of Reeb chords from $\La_j$ to $\La_i$. Such chords are called \textit{mixed} while chords in $\Rc(\La_i)$ are called \textit{pure}. Let $\gamma_{0d}\in\Rc(\La_0,\La_d)$, $(\gamma_1,\dots,\gamma_d)$ be a $d$-tuple of Reeb chords such that $\gamma_i\in\Rc(\La_{i-1},\La_i)\cup\Rc(\La_i,\La_{i-1})$, and $\bd_i$ for $0\leq i\leq d$ words of Reeb chords of $\La_i$. We denote
\begin{alignat*}{1}
	\cM_{\La_{0\dots d}}(\gamma_{0d};\bd_0,\gamma_1,\bd_1,\dots,\gamma_d,\bd_d)
\end{alignat*}
the quotient by the action of $\R$ of the moduli space of pseudo-holomorphic discs satisfying the following:
\begin{itemize}
	\item the boundary of the discs lie on the ordered $(d+1)$-tuple of Lagrangians $\R\times\La_0,\dots,\R\times\La_d$ when following the boundary counter clockwise,
	\item the discs in this moduli space are positively asymptotic to the Reeb chord $\gamma_{0d}$, positively or negatively asymptotic to the Reeb chords $\gamma_i$, and negatively asymptotic to the words of Reeb chords $\bd_i$
\end{itemize} 
Similarly, for a chord $\gamma_{d0}\in\Rc(\La_d,\La_0)$, and other chords as above, we denote
\begin{alignat*}{1}
	\cM_{\La_{0\dots d}}(\gamma_{d0};\bd_0,\gamma_1,\bd_1,\dots,\gamma_d,\bd_d)
\end{alignat*}
the quotient by the action of $\R$ of the moduli space of pseudoholomorphic discs with the same boundary conditions as above, negatively asymptotic to the Reeb chord $\gamma_{d0}$, and with the same asymptotic conditions as above for the other punctures.
\begin{rem}
	In the notation we employ, the first asymptotic $\gamma$, $\gamma_{0d}$ or $\gamma_{d0}$ will always be the output of a map defined by a count of rigid discs in the corresponding moduli spaces. 
	When the Lagrangian boundary conditions for the pseudoholomorphic discs is the ordered $(d+1)$-tuple $(\R\times\La_0,\dots,\R\times\La_d)$, knowing if the first asymptotic is a positive or negative asymptotic is enough to determine if the other mixed asymptotics are positive or negative, according to their direction. More precisely, if a chord $\gamma_i$ is in $\Rc(\La_i,\La_{i-1})$ then it will always be a positive asymptotic while if it is in $\Rc(\La_{i-1},\La_i)$ it will always be a negative asymptotic.
\end{rem}
By \cite{DR}, if the almost complex structure $J$ is the cylindrical lift of an  admissible almost complex structure on $P$ which is regular (i.e. such that the pseudo-holomorphic discs $\pi_P\circ u$ are transversely cut out, for any pseudo-holomorphic disc $u$ in any of the moduli spaces described above), then $J$ is regular meaning that the moduli spaces we described above are transversely cut out. The necessary transversality results for moduli spaces of pseudo-holomorphic discs in $P$ with boundary on $\pi_P(\La)$ are carried out in \cite{EES1}. When transversality holds, these moduli spaces are thus smooth manifolds which can moreover be compactified in the sense of Gromov.

The dimension of a moduli space can be expressed in terms of the Conley-Zehnder indices of its asymptotics. For moduli spaces with only pure asymptotics, we have 
\begin{alignat*}{1}
	&\dim\cM_{\La}(\gamma;\gamma_1,\dots,\gamma_d)=(\CZ(\gamma)-1)-\sum_{i=1}^d\big(\CZ(\gamma_i)-1\big)-1
\end{alignat*}
Then, for a word of pure Reeb chords $\bd=\delta_1\dots\delta_k$ let us denote $\ell(\bd)=k$ its length and denote $\CZ(\bd):=\sum\limits_{i=1}^k\CZ(\delta_i)$. Moreover, let us denote $j^+$ the number of positive mixed Reeb chord asymptotics among $\{\gamma_1,\dots,\gamma_d\}$. We have:
\begin{alignat*}{1}
		&\dim\cM_{\La_{0\dots d}}(\gamma_{0d};\bs{\delta}_0,\gamma_1,\bs{\delta}_1,\dots,\gamma_d,\bs{\delta}_d)=(\CZ(\gamma_{0d})-1)+\sum_{\gamma_i\in\Rc(\La_{i},\La_{i-1})}\big(\CZ(\gamma_i)-1\big)\\
		&\hspace{35mm}-\sum_{\gamma_i\in\Rc(\La_{i-1},\La_{i})}\big(\CZ(\gamma_i)-1\big)-\sum\big(\CZ(\bd_i)-\ell(\bd_i)\big)+(2-n)j^+-1
\end{alignat*}
and
\begin{alignat*}{1}
		&\dim\cM_{\La_{0\dots d}}(\gamma_{d0};\bs{\delta}_0,\gamma_1,\bs{\delta}_1,\dots,\gamma_d,\bs{\delta}_d)=-\big(\CZ(\gamma_{d0})-1\big)+\sum_{\gamma_i\in\Rc(\La_{i},\La_{i-1})}\big(\CZ(\gamma_i)-1\big)\\
		&\hspace{3cm}-\sum_{\gamma_i\in\Rc(\La_{i-1},\La_{i})}\big(\CZ(\gamma_i)-1\big)-\sum\big(\CZ(\bd_i)-\ell(\bd_i)\big)+(2-n)(j^+-1)-1
\end{alignat*}
We refer to \cite[Section 4.3]{CDGG2long} for the computation of these dimensions.
In the following we add an exponent to indicate the dimension of the moduli space, i.e. $\cM^i_\La(\gamma;\gamma_1,\dots,\gamma_d)$ denotes a $i$-dimensional moduli space.
We call \textit{rigid} the pseudo-holomorphic discs in a $0$-dimensional moduli space.

We define the \textit{action} $\mathfrak{a}(\gamma_{ij})$ of a Reeb chord $\gamma_{ij}$ by
$\mathfrak{a}(\gamma_{ij})=\int_{\gamma_{ij}}\alpha$ if $i>j$, and $\mathfrak{a}(\gamma_{ij})=-\int_{\gamma_{ij}}\alpha$ if $i\leq j$.
By positivity of energy for pseudo-holomorphic discs in the moduli spaces defined above, we have  the following. If the moduli space $\cM_\La(\gamma;\gamma_1,\dots,\gamma_d)$ is not empty then the action of the asymptotics satisfy
\begin{alignat*}{1}
	-\mathfrak{a}(\gamma)+\sum\limits_{i=1}^d\mathfrak{a}(\gamma_i)\geq0
\end{alignat*}
If $\cM_{\La_{0\dots d}}(\gamma_{0d};\bd_0,\gamma_1,\bd_1,\dots,\gamma_d,\bd_d)$ is not empty then we have:
\begin{alignat*}{1}
	-\mathfrak{a}(\gamma_{0d})+\sum\limits_{\gamma_i\in\Rc(\La_{i-1},\La_i)}\mathfrak{a}(\gamma_i)+\sum\limits_{\gamma_i\in\Rc(\La_i,\La_{i-1})}\mathfrak{a}(\gamma_i)+\sum_{j=0}^d\mathfrak{a}(\bd_j)\geq0
\end{alignat*}
Finally if $\cM_{\La_{0\dots d}}(\gamma_{d0};\bd_0,\gamma_1,\bd_1,\dots,\gamma_d,\bd_d)$ is non empty then we have:
\begin{alignat*}{1}
	-\mathfrak{a}(\gamma_{d0})+\sum\limits_{\gamma_i\in\Rc(\La_{i-1},\La_i)}\mathfrak{a}(\gamma_i)+\sum\limits_{\gamma_i\in\Rc(\La_i,\La_{i-1})}\mathfrak{a}(\gamma_i)+\sum_{j=0}^d\mathfrak{a}(\bd_j)\geq0
\end{alignat*}
In particular, the definition of action we choose here implies that the maps which we will be defined later by a count of pseudo-holomorphic discs in the moduli spaces above will be action decreasing.

\section{The Chekanov-Eliashberg DGA}\label{CEalg}
In this section we briefly recall the definition of the Chekanov-Eliashberg algebra (C-E algebra) of a Legendrian originally defined in \cite{Che,E} and refer to \cite{Che,EES1,EES2} for details.

Let $\La\subset Y$ be a Legendrian and denote $C(\La)$ the $\Z_2$-module generated by Reeb chords of $\La$. The Chekanov-Eliashberg DGA of $\La$, denoted $\Ac(\La)$, is the tensor algebra of $C(\La)$ over $\Z_2$:
\begin{alignat*}{1}
	\Ac(\La)=\bigoplus_{i\geq0}C(\La)^{\otimes i}
\end{alignat*}
with $C(\La)^{\otimes0}=\Z_2$. The grading of a Reeb chord $\gamma\in\Rc(\La)$ is given by $|\gamma|_\Ac=1-\CZ(\gamma)$, and we extend it to the whole algebra $\Ac(\La)$ by $|\gamma_1\dots\gamma_d|_\Ac=|\gamma_1|_\Ac+\dots+|\gamma_d|_\Ac$.

\begin{rem}
	Note that in this paper we use a slightly unconventional grading for the C-E algebra. It is chosen in such a way that the differential will be a map of degree $1$ instead of a map of degree $-1$.
\end{rem}
The differential $\partial_\Ac$ on $\Ac$ is given on Reeb chords by
\begin{alignat*}{1}
	\partial_\Ac(\gamma)=\sum_{d\geq0}\sum_{\gamma_1,\dots,\gamma_d}\#\cM^0_\La(\gamma;\gamma_1,\dots,\gamma_d)\cdot\gamma_1\dots\gamma_d
\end{alignat*}
and extends to the whole algebra $\Ac$ by Leibniz rule, i.e. $\partial_\Ac(\gamma_1\gamma_2)=\partial_\Ac(\gamma_1)\gamma_2+\gamma_1\partial_\Ac(\gamma_2)$.
\begin{rem}
	In the original definition of the C-E algebra of a Legendrian, the differential is defined by a count of pseudo-holomorphic discs with boundary on the projection $\pi_P(\La)$, for $\pi_P:\R\times P\to P$. Dimitroglou-Rizell proved in \cite{DR} that the differential can equivalently be defined by a count of pseudo-holomorphic discs with boundary on $\R\times\La\subset\R\times Y$. We use this latter perspective in this paper.
\end{rem}
\begin{teo}\cite{Che,EES1,EES2}
	The map $\partial_\Ac$ is a degree $1$ map which satisfies $\partial_\Ac^2=0$.
\end{teo}
The homology of the complex $(\Ac(\La),\partial_\Ac)$ is called the \textit{Legendrian contact homology} of $\La$.

\section{The Rabinowitz DG-bimodule}\label{sec:RFCDG}
Let $\La_0,\La_1\subset Y$ be Legendrian submanifolds and $\Ac(\La_0)$, $\Ac(\La_1)$ denote their C-E DGAs. We denote $C(\La_0,\La_1)$ the graded $\Z_2$-module generated by chords in $\Rc(\La_0,\La_1)$, and graded with the Conley-Zehnder index. We further  denote $C_{\Ac_1-\Ac_0}(\La_0,\La_1)$ the $\Ac(\La_1)-\Ac(\La_0)$-bimodule generated by chords in $\Rc(\La_0,\La_1)$, i.e. elements of $C_{\Ac_1-\Ac_0}(\La_0,\La_1)$ are linear combination of words $\bs{v}_1\gamma_{01}\bs{v}_0$ where $\bs{v}_i$ are words of Reeb chords of $\La_i$ for $i=0,1$ and $\gamma_{01}\in\Rc(\La_0,\La_1)$. The degree of $\bs{v}_1\gamma_{01}\bs{v}_0$ is given by $\CZ(\gamma_{01})+|\bs{v}_0|_{\Ac_0}+|\bs{v}_1|_{\Ac_1}$. Analogously denote $C_{\Ac_1-\Ac_0}(\La_1,\La_0)$ the $\Ac(\La_1)-\Ac(\La_0)$-bimodule generated by chords in $\Rc(\La_1,\La_0)$. 

The Rabinowitz DG-bimodule $\RFC_{\Ac_1-\Ac_0}(\La_0,\La_1)$ is a DG $\Ac(\La_1)-\Ac(\La_0)$-bimodule which is defined as follows. The underlying graded bimodule has two different types of generators:
\begin{alignat*}{1}
	\RFC_{\Ac_1-\Ac_0}(\La_0,\La_1)=C_+(\La_0,\La_1)\oplus C_-(\La_0,\La_1)
\end{alignat*}
where
\begin{itemize}
	\item $C_+(\La_0,\La_1)$ is the $\Ac(\La_1)-\Ac(\La_0)$-bimodule whose elements are the same as in $C_{\Ac_1-\Ac_0}(\La_1,\La_0)$ but the grading of a mixed chord $\gamma_{10}$ is modified by taking the negative and adding $n$, i.e. $|\gamma_{10}|_{\RFC(\La_0,\La_1)}=n-\CZ(\gamma_{10})$.
	\item $C_-(\La_0,\La_1)=C_{\Ac_1-\Ac_0}(\La_0,\La_1)$.
\end{itemize}

The differential on $\RFC_{\Ac_1-\Ac_0}(\La_0,\La_1)$ is given by a lower triangular $2\times2$-matrix
$$\mfm_1=\begin{pmatrix}
	\bs{\D}_1^{++}&0\\
	\ba_1^{-+}&\ba_1^{--}
\end{pmatrix}$$
for which we describe the components on generators of $\RFC_{\Ac_1-\Ac_0}(\La_0,\La_1)$. The map $\bs{\D}_1^{++}:C_+(\La_0,\La_1)\to C_+(\La_0,\La_1)$ on generators is given by the Legendrian contact homology differential of $\La_0\cup\La_1$ restricted to mixed chords from $\La_0$ to $\La_1$, namely for a chord $\gamma_{10}\in\Rc(\La_1,\La_0)$ we have:
\begin{alignat*}{1}
	\bs{\D}_1^{++}(\gamma_{10})=\sum\limits_{\beta_{10},\bs{\delta}_0,\bs{\delta}_1}\#\cM^0_{\La_{01}}(\beta_{10};\bd_0,\gamma_{10},\bd_1)\cdot\bd_1\beta_{10}\bd_0
\end{alignat*}
where the sum is over all mixed chords $\beta_{01}\in\Rc(\La_0,\La_1)$ and all words of pure Reeb chords $\bd_i$ of $\La_i$ for $i=0,1$. Then, we extend it to bimodule elements $\bs{v}_1\gamma_{10}\bs{v}_0$ by:
\begin{alignat*}{1}
	\bs{\D}_1^{++}(\bs{v}_1\gamma_{10}\bs{v}_0)=\partial_{\Ac(\La_1)}(\bs{v}_1)\gamma_{10}\bs{v}_0+\bs{v}_1\bs{\D}_1^{++}(\gamma_{10})\bs{v}_0+\bs{v}_1\gamma_{10}\partial_{\Ac(\La_0)}(\bs{v}_0)
\end{alignat*}
and then by linearity.
The component $\ba_1^{--}:C_-(\La_0,\La_1)\to C_-(\La_0,\La_1)$ is given by the restriction to mixed chords from $\La_1$ to $\La_0$ of the Legendrian contact cohomology differential:
\begin{alignat*}{1}
	\ba_1^{--}(\gamma_{01})=\sum\limits_{\beta_{01},\bs{\delta}_0,\bs{\delta}_1}\#\cM^0_{\La_{01}}(\beta_{01};\bd_0,\gamma_{01},\bd_1)\cdot\bd_1\beta_{01}\bd_0
\end{alignat*}
We extend it to bimodule elements $\bs{v}_1\gamma_{01}\bs{v}_0$ by:
\begin{alignat*}{1}
	\ba_1^{--}(\bs{v}_1\gamma_{01}\bs{v}_0)=\partial_{\Ac(\La_1)}(\bs{v}_1)\gamma_{01}\bs{v}_0+\bs{v}_1 \ba_1^{--}(\gamma_{01})\bs{v}_0+\bs{v}_1\gamma_{01}\partial_{\Ac(\La_0)}(\bs{v}_0)
\end{alignat*}
Finally, the component $\ba_1^{-+}:C_+(\La_0,\La_1)\to C_-(\La_0,\La_1)$ is the \textit{banana map} defined on generators by:
\begin{alignat*}{1}
	\ba_1^{-+}(\gamma_{10})=\sum\limits_{\beta_{01},\bs{\delta}_0,\bs{\delta}_1}\#\cM^0_{\La_{01}}(\beta_{01};\bd_0,\gamma_{10},\bd_1)\cdot\bd_1\beta_{01}\bd_0
\end{alignat*}
and we extend it to bimodule elements $\bs{v}_1\gamma_{10}\bs{v}_0$ by:
\begin{alignat*}{1}
	\ba_1^{-+}(\bs{v}_1\gamma_{10}\bs{v}_0)=\bs{v}_1 \ba_1^{-+}(\gamma_{10})\bs{v}_0
\end{alignat*}
\begin{prop} We have:
	\begin{enumerate}
		\item $\mfm_1$ is a degree $1$ map, and
		\item $\RFC_{\Ac_1-\Ac_0}(\La_0,\La_1)$ is a DG $\Ac(\La_1)-\Ac(\La_0)$-bimodule; i.e. $\mfm_1^2=0$.
	\end{enumerate}
\end{prop}
\begin{proof}
	(1) is obtained by the dimension formula of the moduli spaces, see Section \ref{sec:moduli}.
	(2) is obtained by considering the algebraic contributions of pseudo-holomorphic buildings arising in the boundary of $1$ dimensional moduli spaces of the following type:
	\begin{itemize}
	\item $\cM^1_{\La_{01}}(\beta_{10};\bd_0,\gamma_{10},\bd_1)$: strips with a positive asymptotic at $\gamma_{10}$, a negative asymptotic at $\beta_{10}$ and negative pure asymptotics at the words $\bs{\delta}_0$ and $\bs{\delta}_1$,
	\item $\cM^1_{\La_{01}}(\beta_{01};\bd_0,\gamma_{01},\bd_1)$: strips with a positive asymptotic at $\beta_{10}$ and negative asymptotic at $\gamma_{10}$,
	\item $\cM^1_{\La_{01}}(\beta_{01};\bd_0,\gamma_{10},\bd_1)$: bananas with positive asymptotics at $\beta_{10}$ and  $\gamma_{10}$.
	\end{itemize}
\end{proof}
Equivalently we have that $(C_+(\La_0,\La_1),\bs{\D}_1^{++})$ and $(C_-(\La_0,\La_1),\ba_1^{--})$ are DG $\Ac(\La_1)$-$\Ac(\La_0)$-bimodules, and $\RFC_{\Ac_1-\Ac_0}(\La_0,\La_1)$ is the cone of the DG-bimodule map $\ba_1^{-+}:C_+(\La_0,\La_1)[1]\to C_-(\La_0,\La_1)$.

\begin{nota}
	We suppress the exponents ``$++$'', ``$-+$'' and ``$--$'' indicating the output and input of each component of the differential, and denote $\bs{\D}_1:=\bs{\D}_1^{++}$ and $\ba_1:=\ba_1^{-+}+\ba_1^{--}$. We will write explicitly, when needed, if we consider some restriction of the map $\ba_1$.
\end{nota}

\begin{rem}
	In case the C-E DGAs $\Ac(\La_0)$ and $\Ac(\La_1)$ admit augmentations (see \cite{Che}) $\ep_0$ and $\ep_1$ respectively over $\Z_2$, we can use them to turn the bimodule coefficients into elements of $\Z_2$ and thus get a $\Z_2$-module $\RFC(\La_0,\La_1)$. This latter module corresponds to the one defined in \cite{L2} in the case where the pair of Lagrangian cobordisms is a pair of trivial cylinders. Moreover, if $\La_1$ is a copy of $\La_0$, this complex is isomorphic to the complex of the $2$-copy described in \cite{EESa}.
\end{rem}
We recall now a sufficient condition for the Rabinowitz complex to be acyclic.
\begin{defi}
	A pair of Legendrians $(\La_0,\La_1)$ in $Y=P\times\R$ is said to be \textit{horizontally displaceable} if there is a Hamiltonian isotopy $\varphi_t$ of $P$ such that $\Pi_P(\La_0)\cap\varphi_1(\Pi_P(\La_1))=\emptyset$, where $\Pi_P:Y\to P$ is the projection.
\end{defi}
\begin{teo}[\cite{L2}]\label{teo:acyclicity}
	If $(\La_0,\La_1)$ is horizontally displaceable, then $\RFC_{\Ac_1-\Ac_0}(\La_0,\La_1)$ is acyclic.
\end{teo}
\begin{rem}
	In \cite{L2}, Theorem \ref{teo:acyclicity} is proved for the $\Z_2$-module $\RFC(\La_0,\La_1)$ but the proof extends directly to the bimodule case.
\end{rem}

Assume that $\La_1$ is a perturbation of $\La_0$ by a small negative Morse function, see Section \ref{sec:2copy} below. If the perturbation is sufficiently small then by invariance of the C-E algebra \cite{EES2} (and the fact that the almost complex structure on $\R\times Y$ is the cylindrical lift of an admissible almost complex structure on $P$, see \cite{DR}), the DGAs $\Ac(\La_0)$ and $\Ac(\La_1)$ are canonically isomorphic in the sense that there is a canonical identification of Reeb chords of $\La_0$ with Reeb chords of $\La_1$ such that the differentials coincide under this identification. In this case, we denote the Rabinowitz bimodule $(\RFC_{\Ac-\Ac}(\La_0,\La_1),\mfm_1)$ with $\Ac:=\Ac(\La_0)=\Ac(\La_1)$. It is a DG $\Ac$-bimodule.

\section{The Calabi-Yau structure}\label{sec:CYDGA}

The goal of this section is to prove the following theorem:
\begin{teo}\label{teo:CY}
	The C-E algebra $\Ac(\La)$ of an horizontally displaceable Legendrian sphere $\La\subset Y$ is a $(n+1)$-Calabi-Yau differential graded algebra.
\end{teo}
In order to prove the theorem, we want to find a quasi-isomorphism $\Ac\to\RHom_{\Ac-\Ac}(\Ac,\Ac\otimes\Ac)$ of DG $\Ac$-bimodules. We will start by describing DG-bimodules $\widehat{C}_+(\La_0,\La_1)$ and $\widecheck{C}_-(\La_0,\La_1)$ which are respectively a quotient bimodule and a sub-bimodule of $\RFC_{\Ac_1-\Ac_0}(\La_0,\La_1)$ for a $2$-copy $\La_0\cup\La_1$ of $\La$. We show then that these bimodules are quasi-isomorphic to $\Ac$ and $\RHom_{\Ac-\Ac}(\Ac,\Ac\otimes\Ac)$ respectively (with some degree shifts). Finally, we define a DG-bimodule morphism $\CY:\widehat{C}_+(\La_0,\La_1)\to \widecheck{C}_-(\La_0,\La_1)$ and show that the cone of this morphism is quasi-isomorphic to the DG-bimodule $\RFC_{\Ac-\Ac}(\La_0,\La_1)$. By acyclicity of $\RFC_{\Ac-\Ac}(\La_0,\La_1)$, $\CY$ is a quasi-isomorphism and thus we get the sought quasi-isomorphism $\Ac\to\RHom_{\Ac-\Ac}(\Ac,\Ac\otimes\Ac)$.

\subsection{The $2$-copy $\La_0\cup\La_1$}\label{sec:2copy}
In all this Section \ref{sec:CYDGA} we will assume that $\La\subset Y$ is an $n$-dimensional Legendrian sphere. Let $f\colon\La\to\R$ be a $\Cc^1$-small negative Morse function with exactly one maximum and one minimum, such that the norm of $f$ is much smaller than the length of the shortest Reeb chord of $\La$. Let $\La_1$ denote the $1$-jet of $f$ in a standard neighborhood of $\La$ (identified with a neighborhood of the $0$-section in $J^1(\La)$) and denote $\La_0:=\La$. Then we say that $\La_0\cup\La_1$ is a $2$-copy of $\La$.

Each Reeb chord $\gamma$ of $\La$ gives rise to two mixed Reeb chords of $\La_0\cup\La_1$: $\gamma_{01}\in\Rc(\La_0,\La_1)$ and $\gamma_{10}\in\Rc(\La_1,\La_0)$. The choice of capping paths in order to define the Conley-Zehnder index of mixed chords is made in such a way that $\CZ(\gamma_0)=\CZ(\gamma_{01})=\CZ(\gamma_{10})$.
The Legendrian $\La_0\cup\La_1$ admits two additional mixed Reeb chords from $\La_1$ to $\La_0$ corresponding to the critical points of $f$ and called \textit{Morse chords}. We denote $y_{01}$ the Reeb chord corresponding to the maximum of $f$ and by $x_{01}$ the one corresponding to the minimum. Note that $\ell(y_{01})<\ell(x_{01})$ because the function $f$ is negative, and so we will call $y_{01}$ the \textit{minimum Morse Reeb chord} and $x_{01}$ the \textit{maximum Morse Reeb chord}. Note that $\CZ(y_{01})=0$ while $\CZ(x_{01})=n$. 
Finally, we denote $C_{\Ac-\Ac}^{l}(\La_0,\La_1)$ the $\Ac(\La_1)-\Ac(\La_0)$-bimodule (or $\Ac$-bimodule for short, as the algebras $\Ac(\La_0)$ and $\Ac(\La_1)$ are canonically identified) generated by \textit{long chords} from $\La_1$ to $\La_0$, i.e. chords from $\La_1$ to $\La_0$ which are not Morse chords.

\subsection{The bimodules $\widehat{C}_+(\La_0,\La_1)$ and $\widecheck{C}_-(\La_0,\La_1)$.}
The DG $\Ac$-bimodule $(\widehat{C}_+(\La_0,\La_1),\widehat{\mfm}_1)$ is defined as follows. The underlying graded bimodule is generated by the positive action generators of the Rabinowitz bimodule and the maximum Morse chord:
\begin{alignat*}{1}
	\widehat{C}_+(\La_0,\La_1)=C_+(\La_0,\La_1)[1]\oplus\langle x_{01}\rangle_{\Ac-\Ac}[1]
\end{alignat*}
where $\langle x_{01}\rangle_{\Ac-\Ac}$ denotes the $\Ac$-sub-bimodule of $C_{\Ac-\Ac}(\La_0,\La_1)$ generated by $x_{01}$. So the chord $x_{01}$ in $\langle x_{01}\rangle_{\Ac-\Ac}[1]$ has degree $n+1=\CZ(x_{01})+1$. The differential $\widehat\mfm_1$ is given on generators by
\begin{alignat*}{1}
	&\widehat{\mfm}_1(\gamma_{10})=\bs{\D}_1(\gamma_{10})+\gamma_1 x_{01}+x_{01}\gamma_0\\
	&\widehat{\mfm}_1(x_{01})=0
\end{alignat*}
where $\gamma_i$ denotes the pure chord of $\La_i$ corresponding to $\gamma_{10}$.
\begin{prop}\label{prop:widehatDelta}
	The bimodule $(\widehat{C}_+(\La_0,\La_1),\widehat{\mfm}_1)$ is a semifree DG-bimodule.
\end{prop}
We state here a lemma that we will use repeatedly in several proofs.
\begin{lem}\cite{EESa}\label{lem:bananas} For a $2$-copy $\La_0\cup\La_1$ of $\La_0$ we have:
	\begin{enumerate}
		\item For every Reeb chord $\gamma_{01}\in C^{l}(\La_0,\La_1)$, there are exactly two rigid pseudo-holomorphic strips with positive asymptotic at $\gamma_{01}$ and negative asymptotic at the minimum Reeb chord $y_{01}$. Moreover, each of these strips has exactly one pure negative chord asymptotic which is the chord $\gamma_0$ for one strip and $\gamma_1$ for the other (where $\gamma_i$ denotes the pure Reeb chord of $\La_i$ corresponding to $\gamma_{01}$).
		\item For every Reeb chord $\gamma_{10}\in \Rc(\La_1,\La_0)$, there are exactly two rigid pseudo-holomorphic discs with boundary on $\R\times(\La_0\cup\La_1)$ which are bananas with positive asymptotics at $\gamma_{10}$ and at the maximum Reeb chord $x_{01}$. Moreover, each of these bananas has exactly one pure negative chord asymptotic which is the chord $\gamma_0$ for one banana and $\gamma_1$ for the other.
		\item The count of rigid pseudo-holomorphic strips with boundary on $\R\times(\La_0\cup\La_1)$ admitting a positive puncture at the maximum $x_{01}$ and a negative puncture at a chord $\beta_{01}$, vanishes.
	\end{enumerate}
\end{lem}
\begin{proof} 
	By \cite[Theorem 3.6]{EESa}, rigid pseudo-holomorphic strips with boundary on $\R\times(\La_0\cup\La_1)$, a positive asymptotic at $\gamma_{01}$ and a negative asymptotic at the minimum Reeb chord $y_{01}$ (which corresponds to the maximum of the function $f$) correspond to rigid \textit{generalized discs} which consist of a pseudo-holomorphic disc with boundary on $\R\times\La_0$ with a positive asymptotic at $\gamma_0$ and a negative gradient flow line of $f$ from the maximum critical point to a point on the boundary of the disc. By rigidity, the pseudo-holomorphic disc must be a constant strip $\R\times\gamma_0$ and then there are two ways the flow line can be attach to it (either on $\R\times\{\text{starting point of }\gamma_0\}$, or on $\R\times\{\text{ending point of }\gamma_0\}$), giving after identification te two strips in (1).
	
	For (2), any rigid banana with positive asymptotics at $\gamma_{10}$ and at the maximum Reeb chord $x_{01}$ corresponds to a rigid generalized discs consisting of a pseudo-holomorphic disc with boundary on $\R\times\La_0$ with a positive asymptotic at $\gamma_0$ and a negative gradient flow line of $f$ flowing from a point on the boundary of the disc to the minimum critical point (remember $x_{01}$ corresponds to the minimum of $f$). By rigidity again the disc must be constant and the two possible ways the flow line can be attached to it provides after identification the two bananas in (2).
	
	For (3), by action reasons the chord $\beta_{01}$ must be a Morse chord and the strip has no negative pure Reeb chord asymptotics. Thus, \cite[Theorem 3.6]{EESa} tells us that such discs are in bijective correspondence with negative gradient flow lines of $f$ from the critical point corresponding to $\beta_{01}$ to the minimum. But in our case $f$ has only two critical points so there are either no flow line (for degree reasons in case $n\geq2$), or exactly two flow lines (from the maximum to the minimum).
\end{proof}

\begin{proof}[Proof of Proposition \ref{prop:widehatDelta}]
	We compute $(\widehat{\mfm}_1)^2(\gamma_{10})=\widehat{\mfm}_1\big(\bs{\D}_1(\gamma_{10})\big)+\partial_{\Ac}(\gamma_1)x_{01}+x_{01}\partial_{\Ac}(\gamma_0)$. For any term $\bd_1\beta_{10}\bd_0$ appearing in $\bs{\D}_1(\gamma_{10})$, we have
	\begin{alignat*}{1}
		\widehat{\mfm}_1(\bd_1\beta_{10}\bd_0)&=\partial_\Ac(\bd_1)\beta_{10}\bd_0+\bd_1\big(\bs{\D}_1(\beta_{10})+\beta_1 x_{01}+x_{01}\beta_0\big)\bd_0+\bd_1\beta_{10}\partial_\Ac(\bd_{0})\\
		&=\partial_\Ac(\bd_1)\beta_{10}\bd_0+\bd_1\bs{\D}_1(\beta_{10})\bd_0+\bd_1\beta_{10}\partial_\Ac(\bd_{0})+\bd_1(\beta_1 x_{01}+x_{01}\beta_0)\bd_0\\
		&=\bs{\D}_1(\bd_1\beta_{10}\bd_0)+\bd_1(\beta_1 x_{01}+x_{01}\beta_0)\bd_0
	\end{alignat*}
So we have
\begin{alignat*}{1}
	(\widehat{\mfm}_1)^2(\gamma_{10})=(\bs{\D}_1)^2(\gamma_{10})+\partial_{\Ac}(\gamma_1)x_{01}+x_{01}\partial_{\Ac}(\gamma_0)+\sum\limits_{\bd_1\beta_{10}\bd_0\in\bs{\D}_1(\gamma_{10})}\bd_1(\beta_1 x_{01}+x_{01}\beta_0)\bd_0
\end{alignat*}
We know that $(\bs{\D}_1)^2(\gamma_{10})=0$ and it remains to understand why the last three terms vanish.

The boundary of the compactification of $1$-dimensional (after dividing by translation) moduli spaces of bananas with positive asymptotics at $\gamma_{10}$ and $x_{01}$, and pure negative Reeb chord asymptotics consists of broken discs which are $2$-level buildings connected by a Reeb chord. This Reeb chord is either a mixed chord or a pure chord. In the first case, using Lemma \ref{lem:bananas} (3), the only (non-vanishing) possibility is that the upper level of the building is a disc contributing to $\bs{\D}_1(\gamma_{10})$ with a positive asymptotic at $\gamma_{10}$, a negative asymptotic at a chord $\beta_{10}$ and negative pure Reeb chord asymptotics; and the lower level is a rigid banana with positive asymptotic at $\beta_{10}$ and $x_{01}$. By Lemma \ref{lem:bananas} (2) there are two such bananas and so we get two buildings contributing to
$\bd_1(\beta_1 x_{01}+x_{01}\beta_0)\bd_0$, for any $\bd_1\beta_{10}\bd_0\in\bs{\D}_1(\gamma_{10})$.
In the second case, the top level of the building is a rigid banana with positive asymptotics at $\gamma_{10}$ and $x_{01}$. By Lemma \ref{lem:bananas} (2) the pure connecting Reeb chord is either $\gamma_0$ or $\gamma_1$. The lower level of the building is thus a disk contributing to $\partial_\Ac(\gamma_i)$, for $i=0,1$ depending on which one is the connecting pure Reeb chord. This implies that $\partial_{\Ac}(\gamma_1)x_{01}+x_{01}\partial_{\Ac}(\gamma_0)=0$.

The semifree property of $\widehat{C}_+(\La_0,\La_1)$ is deduced directly from the action filtration on chords. Indeed, denote $\gamma_{10}^j$ for $1\leq j\leq k$ the generators of $\widehat{C}_+(\La_0,\La_1)$ which are long chords. Up to relabeling, there are real numbers $\ell_0,\ell_1,\dots,\ell_k$ such that 
$$\mathfrak{a}(x_{01})<\ell_0<0<\mathfrak{a}(\gamma_{10}^1)<\ell_1<\dots<\ell_{k-1}<\mathfrak{a}(\gamma_{10}^k)<\ell_k$$
The $\Z_2$-vector spaces $F_j\widehat{C}_+=\{\gamma\in C_+(\La_0,\La_1)[1]\oplus\langle x_{01}\rangle_{\Z_2}[1]: \mathfrak{a}(\gamma)<\ell_j\}$, for $0\leq j\leq k$, produce a filtration
$$\Ac\otimes F_0\widehat{C}_+\otimes\Ac\subset \Ac\otimes F_1\widehat{C}_+\otimes\Ac\subset\dots\subset \Ac\otimes F_k\widehat{C}_+\otimes\Ac=\widehat{C}_+(\La_0,\La_1)$$
giving that $\widehat{C}_+(\La_0,\La_1)$ is semifree and of finite rank.
\end{proof}

Let us now describe the $\Ac$-bimodule $(\widecheck{C}_-(\La_0,\La_1),\widecheck{\mfm}_1)$. The underlying graded bimodule is:
\begin{alignat*}{1}
	\widecheck{C}_-(\La_0,\La_1)=C_{\Ac-\Ac}^{l}(\La_0,\La_1)\oplus\langle y_{01}\rangle_{\Ac-\Ac}
\end{alignat*}
and as a $\Ac$-sub-bimodule of $C_{\Ac-\Ac}(\La_0,\La_1)$ the elements are graded with the Conley-Zehnder index.
The differential $\widecheck{\mfm}_1$ on generators is given by the restriction to $\widecheck{C}_-(\La_0,\La_1)$ of the map $\ba_1^{--}$ which was defined in Section \ref{sec:RFCDG}. Note that by Lemma \ref{lem:bananas} (1) we have
\begin{alignat*}{1}
	&\widecheck{\mfm}_1(y_{01})=\ba_1^{--}(y_{01})=\sum\limits_{\gamma\in\Rc(\La)}\gamma_1\gamma_{01}+\gamma_{01}\gamma_0
\end{alignat*}
\begin{prop}
	The bimodule $(\widecheck{C}_-(\La_0,\La_1),\widecheck{\mfm}_1)$ is a semifree DG-bimodule.
\end{prop}
\begin{proof}
	Consider the boundary of the compactification of moduli spaces of strips with a positive and a negative mixed asymptotics at a generator of $\widecheck{C}_-(\La_0,\La_1)$. This boundary consists of two level buildings connected by either a mixed chord or a pure chord. If the connecting mixed chord is $x_{01}$, then by Lemma \ref{lem:bananas} (3) the building will either not exist or arise in pair so its contribution vanishes algebraically. The other cases are exactly those contributing to $(\widecheck{\mfm}_1)^2$.

The argument showing that $\widecheck{C}_-(\La_0,\La_1)$ is semifree is analogous to what was done in the proof of Proposition \ref{prop:widehatDelta}. 
\end{proof}

\subsection{Alternative cone description of the Rabinowitz complex of a sphere.}
We will define a DG-bimodule map 
\begin{alignat}{1}
	\CY:\widehat{C}_+(\La_0,\La_1)\to\widecheck{C}_-(\La_0,\La_1)\label{CYmap}
\end{alignat}
and prove that the cone of $\CY$ is quasi-isomorphic to the DG-bimodule $\RFC_{\Ac-\Ac}(\La_0,\La_1)$.
For $\gamma_{10},x_{01}$ generators of $\widehat{C}_+(\La_0,\La_1)$, we set:
\begin{alignat*}{1}
	&\CY(\gamma_{10})=\sum\limits_{\substack{\beta_{01}\in\Rc^{l}(\La_0,\La_1)\cup\{y_{01}\}\\\bs{\delta}_0,\bs{\delta}_1}}\#\cM^0_{\La_{01}}(\beta_{01};\bd_0,\gamma_{10},\bd_1)\cdot\bd_1\beta_{01}\bd_0\\
	&\CY(x_{01})=\sum\limits_{\substack{\beta_{01}\in\Rc^{l}(\La_0,\La_1)\\\bs{\delta}_0,\bs{\delta}_1}}\#\cM^0_{\La_{01}}(\beta_{01};\bd_0,x_{01},\bd_1)\cdot\bd_1\beta_{01}\bd_0
\end{alignat*}
See Figure \ref{fig:CYmap}.
Note that in the last equation $\beta_{01}$ can never be the minimum Morse chord $y_{01}$ for energy reasons. Observe also that as ungraded maps the map $\CY$ is given on generators of $\widehat{C}_+(\La_0,\La_1)$ by 
the component of the map $\ba_1$ taking values in $\widecheck{C}_-(\La_0,\La_1)$.
\begin{lem}\label{lem:CY1}
	The map $\CY$ is a DG-bimodule map.
\end{lem}
\begin{proof}
We compute first
\begin{alignat*}{1}
	\CY\circ\widehat{\mfm}_1(x_{01})+\widecheck{\mfm}_1\circ\CY(x_{01})=\widecheck{\mfm}_1\circ\CY(x_{01})=(\ba_1^{--})^2(x_{01})=0
\end{alignat*}
where the second to last equality holds simply by definition. Then, we have
\begin{alignat}{2}
	\CY\circ\widehat{\mfm}_1(\gamma_{10})&+\widecheck{\mfm}_1\circ\CY(\gamma_{10})=\CY\big(\bs{\D}_1(\gamma_{10})+\gamma_1x_{01}+x_{01}\gamma_0\big)+\widecheck{\mfm}_1\big(\CY(\gamma_{10})\big)\nonumber\\
	&=\CY\big(\bs{\D}_1(\gamma_{10})\big)+\gamma_1\CY(x_{01})+\CY(x_{01})\gamma_0+\widecheck{\mfm}_1\big(\CY(\gamma_{10})\big)\nonumber\\
	&=\CY\Big(\sum\#\cM^0(\beta_{10};\bs{\delta}_0,\gamma_{10},\bs{\delta}_1)\bs{\delta}_1\beta_{10}\bs{\delta}_0\Big)+\gamma_1\CY(x_{01})+\CY(x_{01})\gamma_0\nonumber\\
	&+\widehat{\mfm}_1\Big(\sum\#\cM^0(\xi_{01};\bs{\sigma}_0,\gamma_{10},\bs{\sigma}_1)\bs{\sigma}_1\xi_{01}\bs{\sigma}_0\Big)\nonumber\\
	&=\sum\#\cM^0(\nu_{01},\bs{\delta}_0',\beta_{10},\bs{\delta}_1')\#\cM^0(\beta_{10};\bs{\delta}_0,\gamma_{10},\bs{\delta}_1)\bs{\delta}_1\bs{\delta}_1'\nu_{01}\bs{\delta}_0'\bs{\delta}_0\label{ba1}\\
	&+\gamma_1\CY(x_{01})+\CY(x_{01})\gamma_0\label{ba2}\\
	&+\sum\#\cM^0(\xi_{01};\bs{\sigma}_0,\gamma_{10},\bs{\sigma}_1)\Big(\partial_\Ac(\bs{\sigma}_1)\xi_{01}\bs{\sigma}_0+\bs{\sigma}_1\xi_{01}\partial_\Ac(\bs{\sigma}_0)\Big)\label{ba3}\\
	&+\sum\#\cM^0(\nu_{01};\bs{\sigma}_0',\xi_{01},\bs{\sigma}_1')\#\cM^0(\xi_{01};\bs{\sigma}_0,\gamma_{10},\bs{\sigma}_1)\bs{\sigma}_1\bs{\sigma}_1'\nu_{01}\bs{\sigma}_0'\bs{\sigma}_0\label{ba4}
\end{alignat}
The terms of this last equation correspond actually to the algebraic contributions of the pseudo-holomorphic buildings appearing in the boundary of the compactification of moduli spaces of bananas with two positive asymptotics, one at $\gamma_{10}$ and another at a generator of $\widecheck{C}_-(\La_0,\La_1)$, and some negative pure Reeb chord asymptotics, see Figure \ref{fig:index2banana}.
\end{proof}
\begin{figure}[ht]  
	\labellist
	\footnotesize
	\pinlabel $\gamma_{10}$ at 39 90
	\pinlabel out at 92 90
	\pinlabel $x_{01}$ at 165 90
	\pinlabel out at 165 -10
	\endlabellist
	\normalsize
	\begin{center}\includegraphics[width=6cm]{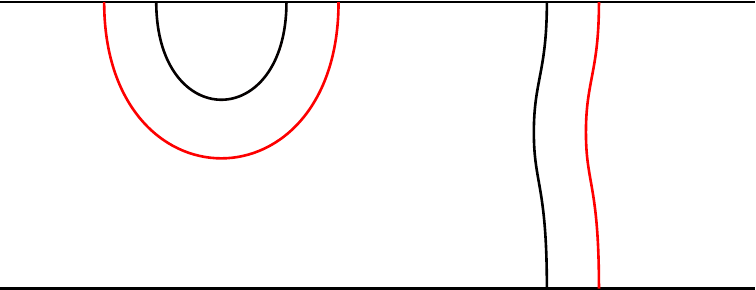}\end{center}
	\caption{Pseudo-holomorphic discs contributing to the map $\CY$.}
	\label{fig:CYmap}
\end{figure}
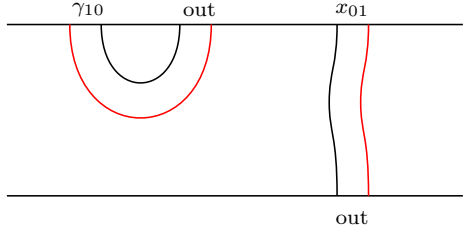
\begin{figure}[ht]  
	\labellist
	\footnotesize
	\pinlabel $\gamma_{10}$ at 38 55
	\pinlabel out at 92 138
	\pinlabel $\gamma_{10}$ at 158 138
	\pinlabel out at 212 138
	\pinlabel $\gamma_{10}$ at 275 138
	\pinlabel out at 332 55
	\endlabellist
	\normalsize
	\begin{center}\includegraphics[width=8cm]{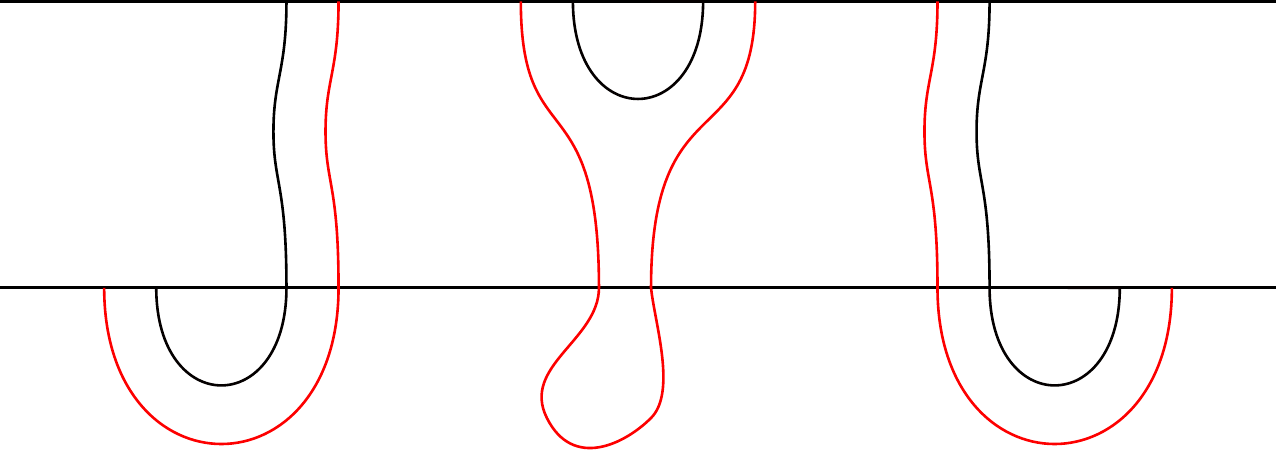}\end{center}
	\caption{Types of pseudo-holomorphic buildings in the boundary of moduli spaces of bananas with a positive asymptotic at $\gamma_{10}$. When the connecting Reeb chords between the two components of the leftmost building is a long chord or $y_{01}$, then it contributes algebraically to \eqref{ba4}; and if it is the chord $x_{01}$, then it contributes algebraically to \eqref{ba2}. The middle building schematizes the contributions of type \eqref{ba3} and the rightmost building contributes to \eqref{ba1}.}
	\label{fig:index2banana}
\end{figure}

The cone of $\CY$, $\Cone(\CY)=\widehat{C}_+(\La_0,\La_1)[-1]\oplus\widecheck{C}_-(\La_0,\La_1)$, is a DG $\Ac$-bimodule,
and we prove that it 
is quasi-isomorphic to the Rabinowitz bimodule $\RFC_{\Ac-\Ac}(\La_0,\La_1)$. This is almost trivial, indeed consider the bimodule map:
\begin{alignat*}{1}
	\nu:\Cone(\CY)\to\RFC_{\Ac-\Ac}(\La_0,\La_1)
\end{alignat*}
sending each mixed chord of the $2$-copy $\La_0\cup\La_1$ to itself.
\begin{prop}
	$\nu$ is a quasi-isomorphism of DG-bimodules.
\end{prop}
\begin{proof}
	We have
\begin{alignat*}{1}
	\nu\circ\mfm_1^{\Cone(\CY)}(\gamma_{10})+\mfm_1\circ\nu(\gamma_{10})&=\nu\big(\widehat{\mfm}_1(\gamma_{10})+\CY(\gamma_{10})\big)+\mfm_1(\gamma_{10})\\
	&=\nu\big(\bs{\D}_1(\gamma_{10})+\gamma_1 x_{01}+x_{01}\gamma_0+\CY(\gamma_{10})\big)+\bs{\D}_1(\gamma_{10})+\ba_1(\gamma_{10})\\
	&=\bs{\D}_1(\gamma_{10})+\gamma_1 x_{01}+x_{01}\gamma_0+\CY(\gamma_{10})+\bs{\D}_1(\gamma_{10})+\ba_1(\gamma_{10})
\end{alignat*}
Observe that $\gamma_1 x_{01}+x_{01}\gamma_0+\CY(\gamma_{10})=\ba_1(\gamma_{10})$ because $\CY(\gamma_{10})$ is defined by a count of bananas with positive asymptotics at $\gamma_{10}$ and at a chord in $\Rc^{l}(\La_0,\La_1)$, and the two other terms correspond to bananas with positive asymptotics at $\gamma_{10}$ and $x_{10}$. So the sum above vanishes. Then the relation $\nu\circ\mfm_1^{\Cone(\CY)}+\mfm_1\circ\nu=0$
when the input is $x_{01}$ or any generator of $\widecheck{C}_-(\La_0,\La_1)$ follows directly from the definition of the maps.
\end{proof}

\subsection{The Calabi-Yau isomorphism}\label{subsec:CYiso}
In this subsection, we assume that $\La_0$ is horizontally displaceable, which implies that the complex $\RFC_{\Ac-\Ac}(\La_0,\La_1)$ is acyclic. Hence it implies that the complex $\Cone(\CY)$ defined in the previous section is also acyclic and thus that the degree $0$ map $\CY:\widehat{C}_+(\La_0,\La_1)\to\widecheck{C}_-(\La_0,\La_1)$ is a quasi-isomorphism of DG-bimodules. We will show that this quasi-isomorphism is a Calabi-Yau isomorphism, by showing that there are quasi-isomorphisms of DG-bimodules $\Ac\simeq\widehat{C}_+(\La_0,\La_1)[-n-1]$ and $\RHom_{\Ac-\Ac}(\Ac,\Ac\otimes\Ac)\simeq\widecheck{C}_-(\La_0,\La_1)$.

We define a bimodule map $F:\widehat{C}_+(\La_0,\La_1)[-n-1]\to\Ac$ by
\begin{alignat*}{1}
	&F(\bs{v}_1\gamma_{10}\bs{v}_0)=0\\
	&F(\bs{v}_1x_{01}\bs{v}_0)=\bs{v}_1\bs{v}_0
\end{alignat*}
for $\bs{v}_i$ words of Reeb chords of $\La_i$ which on the right hand side are canonically identified with words of Reeb chords of $\La_0$.
\begin{prop}\label{prop:F}
	The map $F$ is a quasi-isomorphism of DG-bimodules.
\end{prop}
\begin{proof}
	The fact that $F$ is a degree $0$ chain map follows directly from the definition. It is thus a DG-bimodule morphism and we check that it is a quasi-isomorphism by proving that its cone is acyclic. 
	For a word of Reeb chords $\bs{w}_1=w^1\cdots w^k$ of $\La_1$, write $$\overline{\bs{w}}_1=\sum\limits_{j=1}^kw^1\cdots w^{j-1}w^j_{10}w^{j+1}\cdots w^k\in C_{\Ac-\Ac}(\La_1,\La_0)$$
	Consider the DG-bimodule $\Cone(F)$, whose differential we denote $\partial^{\Cone(F)}$. We define a degree $-1$ $\Z_2$-linear map $h:\Cone(F)\to\Cone(F)$ by:
	\begin{alignat*}{1}
		&h(\bs{w}_1\gamma_{10}\bs{v}_0)=0,\\
		&h(\bs{w}_1 x_{01}\bs{v}_0)=\overline{\bs{w}}_1\bs{v}_0,\\
		&h(a^1\cdots a^k)=x_{01}a^1\cdots a^k
	\end{alignat*}
and show that $h$ defines a homotopy between the identity map and the zero map on $\Cone(F)$, i.e. that
\begin{alignat*}{1}
	\partial^{\Cone(F)}\circ h&+h\circ\partial^{\Cone(F)}=\id_{\Cone(F)}
\end{alignat*}
We check this relation for the three different types of elements in $\Cone(F)$. First:
\begin{alignat*}{1}
	\big(\partial^{\Cone(F)}\circ h&+h\circ\partial^{\Cone(F)}\big)(\bs{w}_1\gamma_{10}\bs{v}_0)=0+h\circ\widehat{\mfm}_1(\bs{w}_1\gamma_{10}\bs{v}_0)+h\circ F(\bs{w}_1\gamma_{10}\bs{v}_0)\\
	&=h\Big(\partial_\Ac(\bs{w}_1)\gamma_{10}\bs{v}_0+\bs{w}_1\gamma_{10}\partial_\Ac(\bs{v}_0)+\bs{w}_1\bs{\D}_1(\gamma_{10})\bs{v}_0+\bs{w}_1\big(\gamma x_{01}+x_{01}\gamma\big)\bs{v}_0\Big)+0\\
	&=h\big(\bs{w}_1\gamma x_{01}\bs{v}_0+\bs{w}_1x_{01}\gamma\bs{v}_0\big)\\
	&=\overline{\bs{w}}_1\gamma\bs{v}_0+\bs{w}_1\gamma_{10}\bs{v}_0+\overline{\bs{w}}_1\gamma\bs{v}_0\\
	&=\bs{w}_1\gamma_{10}\bs{v}_0.
\end{alignat*}
Then,
\begin{alignat*}{1}
	\big(\partial^{\Cone(F)}\circ h+h\circ\partial^{\Cone(F)}\big)(\bs{w}_1x_{01}\bs{v}_0)&=\partial^{\Cone(F)}(\overline{\bs{w}}_1\bs{v}_0)+h\big(\partial_\Ac(\bs{w}_1)x_{01}\bs{v}_0+\bs{w}_1x_{01}\partial_\Ac(\bs{v}_0)+\bs{w}_1\bs{v}_0\big)\\
	&=\widehat{\mfm}_1(\overline{\bs{w}}_1)\bs{v}_0+\overline{\bs{w}}_1\partial_\Ac(\bs{v}_0)+\overline{\partial_\Ac(\bs{w}_1)}\bs{v}_0+\overline{\bs{w}}_1\partial_\Ac(\bs{v}_0)+x_{01}\bs{w}_1\bs{v}_0\\
	&=\widehat{\mfm}_1(\overline{\bs{w}}_1)\bs{v}_0+\overline{\partial_\Ac(\bs{w}_1)}\bs{v}_0+x_{01}\bs{w}_1\bs{v}_0
\end{alignat*}
Observe that $\overline{\partial_\Ac(\bs{w}_1)}=\bs{\D}_1(\overline{\bs{w}}_1)$ which is one component of $\widehat{\mfm}_1(\overline{\bs{w}}_1)$. Assuming $\bs{w}_1=w^1\dots w^k$ we thus get
\begin{alignat*}{1}
	\big(\partial^{\Cone(F)}\circ h+h\circ\partial^{\Cone(F)}\big)(\bs{w}_1x_{01}\bs{v}_0)&=\sum\limits_{j=1}^k(w^1\cdots w^{j-1}(w^jx_{01}+x_{01}w^j)w^{j+1}\cdots w^k)\bs{v}_0+x_{01}\bs{w}_1\bs{v}_0\\
	&=\bs{w}_1x_{01}\bs{v}_0.
\end{alignat*}
Finally, for a word $a^1\dots a^k$ in $\Ac$, we compute
\begin{alignat*}{1}
	\big(\partial^{\Cone(F)}\circ h+h\circ\partial^{\Cone(F)}\big)(a^1\cdots a^k)&=\partial^{\Cone(F)}(x_{01}a^1\cdots a^k)+h\circ\partial_\Ac(a^1\cdots a^k)\\
	&=x_{01}\partial_\Ac(a^1\cdots a^k)+F(x_{01}a^1\cdots a^k)+x_{01}\partial_\Ac(a^1\cdots a^k)\\
	&=a^1\cdots a^k.
\end{alignat*}
\end{proof}
Observe that Proposition \ref{prop:F} implies that $\widehat{C}_+(\La_0,\La_1)[-n-1]$ is a semifree resolution of the diagonal DG-bimodule $\Ac$. It means in particular that there is a quasi-isomorphism of DG-bimodules $\RHom_{\Ac-\Ac}(\Ac,\Ac\otimes\Ac)\simeq\Hom_{\Ac-\Ac}\big(\widehat{C}_+(\La_0,\La_1)[-n-1],\Ac\otimes\Ac\big)$. This semifree resolution is also of finite rank so $\Ac$ is homologically smooth.

\begin{rem}
	Observe that such a finite rank semifree resolution is explicitly given in \cite[Section 3.6]{Keller:deformed} for a DG-category associated to a quiver (with some specific properties).
\end{rem}

We consider now a bimodule map $G:\widecheck{C}_-(\La_0,\La_1)\to\Hom_{\Ac-\Ac}\big(\widehat{C}_+(\La_0,\La_1)[-n-1],\Ac\otimes\Ac\big)$ defined on generators by $G(\gamma_{01})=\phi_{\gamma}$, where $\phi_{\gamma}:\widehat{C}_+(\La_0,\La_1)[-n]\to\Ac\otimes\Ac$ is the bimodule map defined on generators by:
\begin{alignat*}{1}
&\phi_\gamma(\beta_{10})=\left\{
	\begin{array}{ll}
		1\otimes 1&\mbox{ if }\beta=\gamma,\\
		0&\mbox{ otherwise.}
	\end{array}\right.\\
&\phi_\gamma(x_{01})=0.
\end{alignat*}
and similarly $G(y_{01})=\phi_{y}$ which is the bimodule map defined by $\phi_y(\beta_{10})=0$ and $\phi_y(x_{01})=1\otimes 1$.

\begin{prop}\label{prop:Gqi}
	The map $G$ is a quasi-isomorphism of DG-bimodules.
\end{prop}
\begin{proof}
Recall that the differential on $\Hom_{\Ac-\Ac}\big(\widehat{C}_+(\La_0,\La_1)[-n-1],\Ac\otimes\Ac\big)$ is given by $D(\phi)=\phi\circ\widehat{\mfm}_1+\partial_{\Ac\otimes\Ac}\circ\phi$. We want to prove that
\begin{alignat*}{1}
	G\circ\widecheck{\mfm}_1+D\circ G=0
\end{alignat*}
Let $\gamma_{01}$ be a generator in $\widecheck{C}_-(\La_0,\La_1)$. We have that $G\circ\widecheck{\mfm}_1(\gamma_{01}):\widehat{C}_+(\La_0,\La_1)[-n-1]\to\Ac\otimes\Ac$ is the bimodule map for which each pseudo-holomorphic disc with a positive asymptotic at a chord $\beta_{01}$, a negative asymptotic at $\gamma_{01}$ and negative asymptotics at words of pure Reeb chords $\bs{v}_0$ and $\bs{v}_1$ contributes $\bs{v}_1\otimes\bs{v}_0$ when applied to $\beta_{10}\in\widehat{C}_+(\La_0,\La_1)$. On the other side, we have $D\circ G(\gamma_{01})=G(\gamma_{01})\circ\widehat{\mfm}_1+\partial_{\Ac\otimes\Ac}\circ G(\gamma_{01})$. The exact same pseudo-holomorphic discs as above contribute $\bs{v}_1\otimes\bs{v}_0$ to $G(\gamma_{01})\circ\widehat{\mfm}_1(\beta_{10})$. Note that these discs can exist only if $\beta\neq\gamma$, for action reasons. In this case, we also have $\partial_{\Ac\otimes\Ac}\circ G(\gamma_{01})(\beta_{10})=0$. If $\beta=\gamma$, then $G\circ\widecheck{\mfm}_1(\gamma_{01})(\beta_{10})=G(\gamma_{01})\circ\widehat{\mfm}_1(\beta_{10})=0$ and $\partial_{\Ac\otimes\Ac}\circ G(\gamma_{01})(\beta_{10})=\partial_{\Ac\otimes\Ac}(1\otimes1)=0$. Observe that for $x_{01}$ we directly have that
\begin{alignat*}{1}
	G\circ\widecheck{\mfm}_1(\gamma_{01})(x_{01})+D\circ G(\gamma_{01})(x_{01})=0
\end{alignat*}
We now prove that $G\circ\widecheck{\mfm}_1+D\circ G$ vanishes when applied to $y_{01}\in\widecheck{C}_-(\La_0,\La_1)$.
Recall that $\widecheck{\mfm}_1(y_{01})=\sum\limits_{\gamma\in\Rc(\La)}\gamma\gamma_{01}+\gamma_{01}\gamma$, so we have
\begin{alignat*}{1}
	&G\circ\widecheck{\mfm}_1(y_{01})(\gamma_{10})=\gamma\otimes1+1\otimes\gamma\\
	&G\circ\widecheck{\mfm}_1(y_{01})(x_{01})=0
\end{alignat*}
On the other side, for a chord $\gamma_{10}$, we have
\begin{alignat*}{1}
	D\circ G(y_{01})(\gamma_{10})&=G(y_{01})\circ\widehat{\mfm}_1(\gamma_{10})+\partial_{\Ac\otimes\Ac}\circ G(y_{01})(\gamma_{10})\\
	&=G(y_{01})\circ\widehat{\mfm}_1(\gamma_{10})+0\\
	&=G(y_{01})\big(\bs{\D}_1(\gamma_{10})+\gamma x_{01}+x_{01}\gamma\big)\\
	&=G(y_{01})(\gamma x_{01}+x_{01}\gamma)=\gamma\otimes1+1\otimes\gamma
\end{alignat*}
And finally,
\begin{alignat*}{1}
	D\circ G(y_{01})(x_{01})&=G(y_{01})\circ\widehat{\mfm}_1(\gamma_{10})+\partial_{\Ac\otimes\Ac}\circ G(y_{01})(x_{01})\\
	&=0+\partial_{\Ac\otimes\Ac}(1\otimes1)=0
\end{alignat*}
Observe that $G$ admits an inverse $G^{-1}:\Hom_{\Ac-\Ac}\big(\widehat{C}_+(\La_0,\La_1)[-n-1],\Ac\otimes\Ac\big)\to\widecheck{C}_-(\La_0,\La_1)$ which maps a bimodule morphism $\phi\in\Hom_{\Ac-\Ac}\big(\widehat{C}_+(\La_0,\La_1)[-n-1],\Ac\otimes\Ac\big)$ given by
\begin{alignat*}{1}
	&\phi(\gamma_{10})=\sum\limits_{j=1}^{k_\gamma}\bs{w}_\gamma^j\otimes\bs{v}_\gamma^j\\
	&\phi(x_{01})=\sum\limits_{j=1}^{k_x}\bs{w}_x^j\otimes\bs{v}_x^j
\end{alignat*}
to $G^{-1}(\phi)=\sum\limits_{\gamma\in\Rc(\La)}\sum\limits_{j=1}^{k_\gamma}\bs{w}_\gamma^j\gamma_{01}\bs{v}_\gamma^j+\sum\limits_{j=1}^{k_x}\bs{w}_x^jy_{01}\bs{v}_x^j$.
\end{proof}

\begin{proof}[Proof of Theorem \ref{teo:CY}]
From the proof of Proposition \ref{prop:F} we obtained that $\Ac$ is homologically smooth with $\widehat{C}_+(\La_0,\La_1)[-n-1]$ a finite dimensional semifree resolution. Consider the shifted Calabi-Yau map:
\begin{alignat*}{1}
	\mcCY:=\CY[-n-1]:\widehat{C}_+(\La_0,\La_1)[-n-1]\to\widecheck{C}_-(\La_0,\La_1)[-n-1]
\end{alignat*}
If $\La_0$ is horizontally displaceable then $\CY$ is a quasi-isomorphism. Moreover, $\widehat{C}_+(\La_0,\La_1)[-n-1]$ is quasi-isomorphic to $\Ac$ by the map $F$ and the target $\widecheck{C}_-(\La_0,\La_1)[-n-1]$ is quasi-isomorphic to  $\Hom_{\Ac-\Ac}\big(\widehat{C}_+(\La_0,\La_1)[-n-1],\Ac\otimes\Ac\big)[-n-1]$ by the map $G[-n-1]$. Remember also that we have $\Hom_{\Ac-\Ac}\big(\widehat{C}_+(\La_0,\La_1)[-n-1],\Ac\otimes\Ac\big)[-n-1]\cong\RHom_{\Ac-\Ac}\big(\Ac,\Ac\otimes\Ac\big)[-n-1]$, so we get a quasi-isomorphism $\Ac\to\Ac^![-n-1]$.
Finally, we need to check that $\mcCY=\mcCY^![-n-1]$.
Note that 
\begin{alignat*}{1}
	\mcCY^!:\Hom_{\Ac-\Ac}\big(\widecheck{C}_-(\La_0,\La_1)[-n-1],\Ac\otimes\Ac\big)\to\Hom_{\Ac-\Ac}\big(\widehat{C}_+(\La_0,\La_1)[-n-1],\Ac\otimes\Ac\big)
\end{alignat*}
is by definition given by $\mcCY^!(\phi)=\phi\circ\mcCY$.
Observe that the target of $\mcCY^!$ is quasi-isomorphic to $\widecheck{C}_-(\La_0,\La_1)$ via the map $G^{-1}$, while its domain is quasi-isomorphic to $\widehat{C}_+(\La_0,\La_1)$ via the DG-bimodule map:
\begin{alignat*}{1}
	H:\widehat{C}_+(\La_0,\La_1)\to\Hom_{\Ac-\Ac}\big(\widecheck{C}_-(\La_0,\La_1)[-n-1],\Ac\otimes\Ac\big)
\end{alignat*}
defined on generators by $H(\gamma_{10})=\phi_{\gamma}$ with $\phi_{\gamma}(\gamma_{01})=1\otimes1$ and $\phi_{\gamma}$ vanishes otherwise; and $H(x_{01})=\phi_{x}$ where $\phi_{x}(y_{01})=1\otimes1$ and vanishes otherwise.
A similar argument as in the proof of Proposition \ref{prop:Gqi} shows that $H$ is a quasi-isomorphism. Thus we get a DG-bimodule map
\begin{alignat*}{1}
	G^{-1}\circ\mcCY^!\circ H:\widehat{C}_+(\La_0,\La_1)\to\widecheck{C}_-(\La_0,\La_1)
\end{alignat*}
We compute this map on generators:
\begin{alignat*}{1}
	G^{-1}\circ\mcCY^!\circ H(\gamma_{10})=G^{-1}\circ\mcCY^!(\phi_{\gamma})=G^{-1}\circ\phi_{\gamma}\circ\mcCY
\end{alignat*}
And $\phi_{\gamma}\circ\mcCY\in\Hom_{\Ac-\Ac}(\widehat{C}_+(\La_0,\La_1)[-n-1],\Ac\otimes\Ac)$ is given by a count of bananas or strips with a positive asymptotic at $\gamma_{01}$, i.e.
\begin{alignat*}{1}
	\phi_{\gamma}\circ\mcCY(\beta_{10})&=\phi_{\gamma}\Big(\sum_{\substack{\xi_{01}\in\Rc^{l}(\La_0,\La_1)\cup\{y_{01}\}\\\bd_0,\bd_1}}\#\cM^0(\xi_{01};\bd_0,\beta_{10},\bd_1)\bd_1\xi_{01}\bd_0\Big)\\
	&=\sum_{\bd_0,\bd_1}\#\cM^0(\gamma_{01};\bd_0,\beta_{10},\bd_1)\bd_1\otimes\bd_0
\end{alignat*}
and
\begin{alignat*}{1}
	\phi_{\gamma}\circ\mcCY(x_{01})=\sum_{\bd_0,\bd_1}\#\cM^0(\gamma_{01};\bd_0,x_{01},\bd_1)\bd_1\otimes\bd_0
\end{alignat*}
So we get
\begin{alignat*}{1}
	G^{-1}\circ\mcCY^!\circ H(\gamma_{10})=\sum_{\beta_{10},\bd_0,\bd_1}\#\cM^0(\gamma_{01};\bd_0,\beta_{10},\bd_1)\bd_1\beta_{01}\bd_0+\sum_{\bd_0,\bd_1}\#\cM^0(\gamma_{01};\bd_0,x_{01},\bd_1)\bd_1y_{01}\bd_0
\end{alignat*}
while
\begin{alignat*}{1}
	\mcCY(\gamma_{10})=\sum_{\substack{\beta_{01}\in\Rc^{l}(\La_0,\La_1)\\\bd_0,\bd_1}}\#\cM^0(\beta_{01};\bd_0,\gamma_{10},\bd_1)\bd_1\beta_{01}\bd_0+\sum_{\bd_0,\bd_1}\#\cM^0(y_{01};\bd_0,\gamma_{10},\bd_1)\bd_1y_{01}\bd_0
\end{alignat*}
Similarly we can compute 
\begin{alignat*}{1}
	G^{-1}\circ\mcCY^!\circ H(x_{01})=\sum_{\beta_{10},\bd_0,\bd_1}\#\cM^0(y_{01};\bd_0,\beta_{10},\bd_1)\bd_1\beta_{01}\bd_0
\end{alignat*}
while
\begin{alignat*}{1}
	\mcCY(x_{01})=\sum\#\cM^0(\beta_{01};\bd_0,x_{01},\bd_1)\bd_1\beta_{01}\bd_0
\end{alignat*}
So in order to get $\mcCY=\big(G^{-1}\circ\mcCY^!\circ H\big)[-n-1]:=G^{-1}[-n-1]\circ\mcCY[-n-1]\circ H[-n-1]$ we need to check that:
\begin{enumerate}
	\item for any $\gamma_{10}\in\Rc(\La_1,\La_0)$, $\beta_{01}\in\Rc^{l}(\La_0,\La_1)$, and their corresponding $\gamma_{01}\in\Rc^{l}(\La_0,\La_1)$ and $\beta_{10}\in\Rc(\La_1,\La_0)$, we have:
	\begin{alignat*}{1}
		\sum_{\bd_0,\bd_1}\#\cM^0(\beta_{01};\bd_0,\gamma_{10},\bd_1)=\sum_{\bd_0,\bd_1}\#\cM^0(\gamma_{01};\bd_0,\beta_{10},\bd_1)
	\end{alignat*}
	\item for any $\beta_{01}\in\Rc^{l}(\La_0,\La_1)$ and its corresponding $\beta_{10}\in\Rc(\La_1,\La_0)$, we have:
	\begin{alignat*}{1}
		\sum_{\bd_0,\bd_1}\#\cM^0(\beta_{01};\bd_0,x_{01},\bd_1)=\sum_{\bd_0,\bd_1}\#\cM^0(y_{01};\bd_0,\beta_{10},\bd_1)
	\end{alignat*}
\end{enumerate}
The point (1) follows from \cite[Theorem 3.6]{EESa}. Indeed the count of rigid bananas with boundary on $\R\times(\La_0\cup\La_1)$ and positive asymptotics at $\beta_{10}$ and $\gamma_{01}$ is in bijective correspondence with the count of rigid strips with boundary on $\R\times(\La_0\cup\La^A_1)$, where $\La_1^A$ is a translation of $\La_1$ in the positive Reeb direction by $A>>0$ such that the only mixed Reeb chords are from $\La_0$ to $\La_1^A$, with a positive asymptotic at $\beta_{10}^A$ and a negative asymptotic at $\gamma_{10}^A$, where $\beta_{10}^A$ and $\gamma_{10}^A$ are the mixed chords of $\La_0\cup\La_1^A$ corresponding to $\beta_{10}$ and $\gamma_{01}$ respectively. By \cite[Theorem 3.6]{EESa}, this count of rigid strips with boundary on $\R\times(\La_0\cup\La_1^A)$ corresponds to the count of strips with boundary on $\R\times\La$ with a positive asymptotic at $\beta$ (the pure chord corresponding to $\beta_{10}$) and a negative asymptotic at $\gamma$. Similarly, the count of rigid bananas with boundary on $\R\times(\La_0\cup\La_1)$ and positive asymptotics at $\beta_{01}$ and $\gamma_{10}$ is in bijective correspondence with the count of rigid strips with boundary on $\R\times(\La_0\cup\La_1^{-A})$ with a positive asymptotic at $\beta_{01}^{-A}$ and negative asymptotic at $\gamma_{01}^{-A}$. This last count is also in bijective correspondence with the count of rigid strips with boundary on $\R\times\La$, positively asymptotic to $\beta$ and negatively asymptotic to $\gamma$.

For (2), by \cite[Theorem 3.6]{EESa} a rigid strip with boundary on $\R\times(\La_0\cup\La_1)$, a positive asymptotic at $\beta_{01}$ and a negative asymptotic at $x_{01}$ corresponds to a generalized disc consisting of a disc with boundary on $\R\times\La$ having a positive asymptotic at $\beta$ together with a negative gradient flow line of $f$ from the minimum (remember that the maximum Reeb chord $x_{01}\in\Rc(\La_0,\La_1)$ corresponds to the minimum critical point $x$ of $f$) to a point on the boundary of the disc. But there is no non constant negative gradient flow line flowing from the minimum so the boundary of the disc has to pass by $x$ (more precisely the boundary crosses $\R\times\{x\}$ as we are in the cylindrical setting). On the other side, a rigid banana with boundary on $\R\times(\La_0\cup\La_1)$ and positively asymptotic to $\beta_{10}$ and $y_{01}$ corresponds to a rigid generalized disc consisting of a disc with boundary on $\R\times\La$ positively asymptotic to $\beta$ together with a negative gradient flow line of $f$ from a point on the boundary of the disc to the maximum critical point $y$. Again, this flow line must be constant and the boundary of the disc must pass through $\R\times\{y\}$. Now we check that the count of these two types of discs is the same.

If $\dim\La=1$: assume $\cM^0(\beta_{01};\bd_0,x_{01},\bd_1)$ is not empty, i.e. assume that there is a pseudo-holomorphic disc with boundary on $\R\times\La$ passing through $\R\times\{x\}$, positively asymptotic to a chord $\beta$ and negatively asymptotic to the words $\bd_0,\bd_1$. Note that the boundary of the disc is transverse to $\R\times\{x\}$. In particular, it passes also through all $\R\times\{\text{pt}\}$ for every points $\text{pt}$ sufficiently close to $x$ on $\La$. If the function $f$ is chosen so that its critical points are sufficiently close to each other, we get the equality in (2).

If $\dim\La\geq2$, we use the results in \cite{DR:surgery}. Denote $\cM^{\{*\}}_\La(\beta;\bd_0,\bd_1)$ the moduli space of pseudo-holomorphic discs with boundary on $\R\times\La$, positively asymptotic to $\beta$, negatively asymptotic to the words $\bd_0$ and $\bd_1$, and having a marked point $*$ on the boundary of the disc in the domain which is situated between the puncture mapped to the last Reeb chord of the word $\bd_0$ and the puncture mapped to the first Reeb chord of the word $\bd_1$. Note that there is still an action of $\R$ on this moduli space.
There is a smooth evaluation map
\begin{alignat*}{1}
	ev:\cM^{\{*\}}_\La(\beta;\bd_0,\bd_1)\to\La
\end{alignat*}
given by $ev(u)=u(*)$. Note that the evaluation map takes values in $\La$ (instead of the general $\R\times\La$) as we are in the cylindrical setting. 
By a generalization of \cite[Chapter 3]{McDS}, every point of $\La$ is a regular value of the evaluation map, and so are in particular the minimum and maximum Morse critical points $x$ and $y$. Now, using the transversality results in \cite[Section 8]{DR:surgery}, we have that the evaluation map is proper and thus $\#ev^{-1}(x)=\#ev^{-1}(y)$. Note finally that the $0$ dimensional moduli spaces $\cM_\La^0(\beta_{01};\bd_0,x_{01},\bd_1)$  and $\cM_\La^0(y_{01};\bd_0,\beta_{10},\bd_1)$ are respectively identified with $ev^{-1}(x)$ and $ev^{-1}(y)$. 

\end{proof}

\section{Products and higher order structure maps}\label{sec:Product}
In this section we define chain complexes $\widecheck{C}_-^{cyc}(\La_0,\La_1)$ and $\widehat{C}_+^{cyc}(\La_0,\La_1)$ over $\Z_2$, obtained from $\widecheck{C}_-(\La_0,\La_1)$ and $\widehat{C}_+(\La_0,\La_1)$ by the bimodule tensor product with the diagonal bimodule $\Ac$. We show that these complexes admit product structures and that the Calabi-Yau morphism \eqref{CYmap} induced on them preserves the products in homology.
\begin{nota}
	In the following we will consider Reeb chords with boundary on a $3$-copy (and even more) of a Legendrian. Previously we denoted $\gamma_{01}$ for a Reeb chord from $\La_1$ to $\La_0$ and $\gamma_{10}$ for the corresponding Reeb chords from $\La_0$ to $\La_1$. Unless specified, this ``correspondence'' doesn't apply anymore in this section, i.e. we will denote Reeb chords of the $3$-copy by $\gamma_{ij}$ with $1\leq i\neq j\leq2$, but $\gamma_{ij}$ is not necessarily the chord from $\La_j$ to $\La_i$ which corresponds to $\gamma_{ji}$, unless specified. But we will still use $x_{ij}$ and $y_{ij}$ to denote maximum, respectively minimum, Morse Reeb chords between different $2$-copies.
\end{nota}

\subsection{Chain complexes for Hochschild homology and cohomology}
	
We start by describing the chain complex $(\widehat{C}_+^{cyc}(\La_0,\La_1),\widehat{\mfm}_1)$, where by abuse of notation we also denote $\widehat{\mfm}_1$ the differential but this should not create any confusion. The vector space $\widehat{C}_+^{cyc}(\La_0,\La_1)$ is infinite dimensional, generated over $\Z_2$ by elements of the form $\gamma_{10}\bs{a}$ and $x_{01}\bs{a}$ where $\bs{a}=a_1\dots a_k$ denotes a word of Reeb chords of $\La$. The differential is given by
\begin{alignat*}{1}
	\widehat{\mfm}_1(\gamma_{10}\bs{a})=&\sum\limits_{\beta_{10},\bs{\delta}_0,\bs{\delta}_1}\#\cM^0_{\La_{01}}(\beta_{10};\bd_0,\gamma_{10},\bd_1)\cdot\beta_{10}\bd_0\bs{a}\bd_1+x_{01}\gamma\bs{a}+x_{01}\bs{a}\gamma\\
	&+\sum_{j=1}^k\gamma_{10}a_1\dots a_{j-1}\partial_\Ac(a_j)a_{j+1}\dots a_k
\end{alignat*}
where $\gamma$ is the pure Reeb chord corresponding to $\gamma_{10}$, and $\widehat{\mfm}_1(x_{01}\bs{a})=0$. See Figure \ref{fig:Hochschild}.
Note that this complex computes the Hochschild homology of $\Ac$, by definition. Indeed, denote $\Ac^e=\Ac\otimes\Ac^{op}$ and observe that the $\Ac$-bimodule $\widehat{C}_+(\La_0,\La_1)$ can be viewed as a right DG $\Ac^e$-module with module structure given by $\widehat{c}\cdot(a,b)=b\widehat{c}a$, for any $\widehat{c}\in\widehat{C}_+(\La_0,\La_1)$,
while the algebra $\Ac$ can be viewed as a DG left $\Ac^e$-module with $(a,b)\cdot a_1\dots a_k=aa_1\dots a_kb$.
By definition, $\widehat{C}_+^{cyc}(\La_0,\La_1)$ is equal to the tensor product $\widehat{C}_+(\La_0,\La_1)\otimes_{\Ac^e}\Ac$ of the left and right DG $\Ac^e$-modules. Remember moreover that $\widehat{C}_+(\La_0,\La_1)$ is a semifree resolution of $\Ac$, so the homology of $\widehat{C}_+(\La_0,\La_1)\otimes_{\Ac^e}\Ac$ is isomorphic to the Hochschild homology of $\Ac$.
Similarly, the complex $\big(\widecheck{C}_-^{cyc}(\La_0,\La_1),\widecheck{\mfm}_1\big)$ is generated over $\Z_2$ by elements $\gamma_{01}\bs{a}$ and $y_{01}\bs{a}$, and 
\begin{alignat*}{1}
	\widecheck{\mfm}_1(\gamma_{01}a_1\dots a_k)=&\sum\limits_{\beta_{01},\bs{\delta}_0,\bs{\delta}_1}\#\cM^0_{\La_{01}}(\beta_{01};\bd_0,\gamma_{01},\bd_1)\cdot\beta_{01}\bd_0a_1\dots a_k\bd_1\\
	&+\sum_{j=1}^k\gamma_{01}a_1\dots a_{j-1}\partial_\Ac(a_j)a_{j+1}\dots a_k
\end{alignat*}
and $\widehat{\mfm}_1(y_{01}\bs{a})=\sum\limits_{\gamma\in\Rc(\La)}\gamma_{01}\gamma\bs{a}+\gamma_{01}\bs{a}\gamma$, where $\gamma_{10}$ is the mixed chord corresponding to $\gamma$. Again, we can check that
\begin{alignat*}{1}
	\widecheck{C}_-^{cyc}(\La_0,\La_1)\simeq\widecheck{C}_-(\La_0,\La_1)\otimes_{\Ac^e}\Ac\simeq\RHom_{\Ac^e}(\Ac,\Ac\otimes\Ac)\otimes_{\Ac^e}\Ac\simeq\RHom_{\Ac^e}(\Ac,\Ac)
\end{alignat*}
and thus the complex $(\widecheck{C}_-^{cyc}(\La_0,\La_1),\widecheck{\mfm}_1)$ computes the Hochschild cohomology of $\Ac$. By Theorem \ref{teo:CY} we thus get an isomorphism between Hochschild homology and cohomology of the C-E algebra of an horizontally displaceable Legendrian sphere in $\R\times P$.
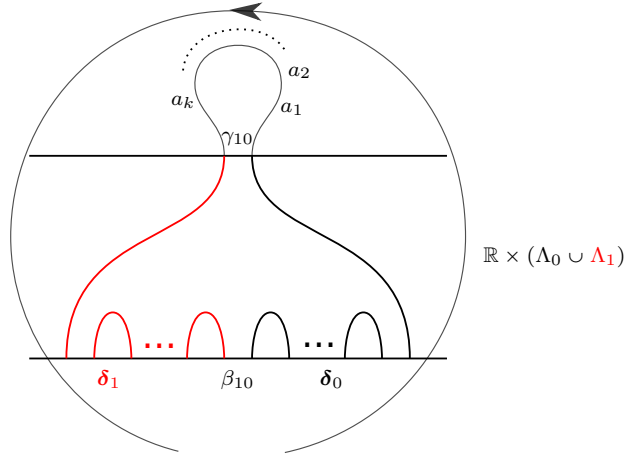
\begin{figure}[ht]  
	\labellist
	\footnotesize
	\pinlabel $\R\times(\La_0\cup\textcolor{red}{\La_1})$ at 220 80
	\pinlabel $\gamma_{10}$ at 92 128
	\pinlabel $\beta_{10}$ at 92 30
	\pinlabel $\bs{\delta}_0$ at 130 30
	\pinlabel $\textcolor{red}{\bs{\delta}_1}$ at 40 30
	\pinlabel $a_1$ at 114 140
	\pinlabel $a_2$ at 117 155
	\pinlabel $a_k$ at 70 143
	\endlabellist
	\normalsize
	\begin{center}\includegraphics[width=6cm]{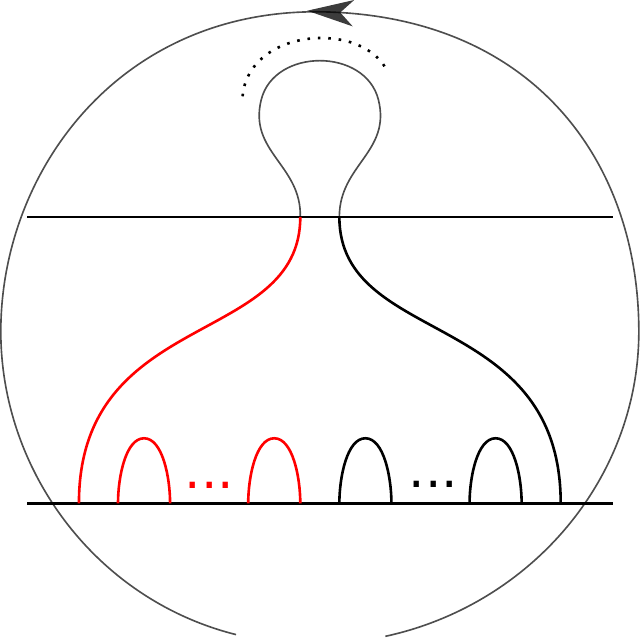}\end{center}
	\caption{Pseudo-holomorphic disc contributing to $\widehat{\mfm}(\gamma_{10}a_1\dots a_k)$. Observe that the ``bubble'' at the top is not a pseudo-holomorphic disc but a way to write the cyclic word $\gamma_{10}a_1\dots a_k$. The contribution $\beta_{10}\bs{\delta}_0a_1\dots a_k\bs{\delta}_1$ of the disc is given by the output mixed chord $\beta_{10}$ followed by a word of Reeb chords as they appear along the boundary of the disc when following it counterclockwise from $\beta_{10}$.}
	\label{fig:Hochschild}
\end{figure}

\subsection{Product structures}\label{sec:product}
In \cite{L2} the author defined a product structure and more generally $A_\infty$-structure maps on the complex $\RFC(\La_0,\La_1)$ in the case when $\La_0,\La_1$ admit augmentations of their C-E algebras.
When $\La_1$ is a small negative push-off of $\La_0$, this product extends naturally to the chain complex $\RFC^{cyc}(\La_0,\La_1)$ (cone of the banana map induced on the cyclic model $\widehat{C}^{cyc}(\La_0,\La_1)$ and with values in $\widecheck{C}^{cyc}(\La_0,\La_1)$).
This product is defined by counting the same type of pseudo-holomorphic discs as the one counted to get a product on  $\RFC(\La_0,\La_1)$ but keeping the negative pure Reeb chords asymptotics as coefficients instead of turning them into elements of $\Z_2$ with augmentations. It doesn't seem possible however to define a product directly on the $\Ac$-bimodule $\RFC(\La_0,\La_1)$, because there is no good way to deal with the coefficients (see \cite{CDRGG:noncomm} for constructions of $A_\infty$ structures with coefficients in a non-commutative algebra).
Note that the pseudo-holomorphic discs we count have boundary on the cylinder over a $3$-copy $\La_0\cup\La_1\cup\La_2$ of a Legendrian $\La$. This $3$-copy is given by $\La_0=\La$ and then $\La_1$ and $\La_2$ are small perturbed negative push-offs of $\La$ ($\La_2$ is a slightly more negative push-off than $\La_1$) such that $\La_0\cup\La_1$, $\La_0\cup\La_2$ and $\La_1\cup\La_2$ are $2$-copies as described in Section \ref{sec:2copy}.
Observe that $\widecheck{C}_-^{cyc}(\La_0,\La_1)$ is a subcomplex of $\RFC^{cyc}(\La_0,\La_1)$, and that the restriction of the product in $\RFC^{cyc}(\La_0,\La_1)$ to $\widecheck{C}_-^{cyc}$ takes values in $\widecheck{C}_-^{cyc}$, see \cite{L2}. So we get a well-defined product on $\widecheck{C}_-^{cyc}$. More precisely, given $\La_0\cup\La_1\cup\La_2$ a $3$-copy of $\La$ there is a degree $0$ map $\widecheck{\mfm}_2:\widecheck{C}_-^{cyc}(\La_1,\La_2)\otimes\widecheck{C}_-^{cyc}(\La_0,\La_1)\to\widecheck{C}_-^{cyc}(\La_0,\La_2)$ defined by
\begin{alignat*}{1}
\widecheck{\mfm}_2(\gamma_{12}\bs{a}_1,\gamma_{01}\bs{a_0})=\sum\limits_{\substack{\gamma_{02}\\\bd_0,\bd_1,\bd_2}}\#\cM^0_{\La_{012}}(\gamma_{02};\bd_0,\gamma_{01},\bd_1,\gamma_{12},\bd_2)\cdot\gamma_{02}\bd_0\bs{a}_0\bd_1\bs{a}_1\bd_2
\end{alignat*}
where $\gamma_{ij}$ can also be the minimum Morse Reeb chord. And this map $\widecheck{\mfm}_2$ satisfies the Leibniz rule $\widecheck{\mfm}_1\circ\widecheck{\mfm}_2+\widecheck{\mfm}_2(\id\otimes\widecheck{\mfm}_1)+\widecheck{\mfm}_2(\widecheck{\mfm}_1\otimes\id)=0$.
This product is the standard ``two negative inputs one positive output'' product, see Figure \ref{fig:standard_prod}. As for the differential, the word of pure Reeb chords in the output element is obtained by following the boundary of the pseudo-holomorphic disc counterclockwise from the mixed output Reeb chord.
\begin{figure}[ht]  
	\labellist
	\footnotesize
	\pinlabel $\R\times(\La_0\cup\textcolor{red}{\La_1}\cup\textcolor{blue}{\La_2})$ at 140 50
	\pinlabel out at 60 90
	\pinlabel in at 35 -7
	\pinlabel in at 85 -7
	\endlabellist
	\normalsize
	\begin{center}\includegraphics[width=3cm]{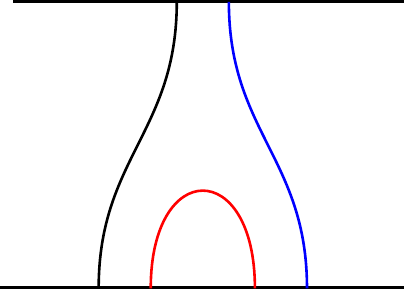}\end{center}
	\caption{Pseudo-holomorphic building with boundary on $\R\times(\La_0\cup\La_1\cup\La_2)$ contributing to the product on $\widecheck{C}_-^{cyc}(\La_0,\La_1)$.}
	\label{fig:standard_prod}
\end{figure}
Because of the coefficients in the C-E algebra, verifying the Leibniz rule for the product on $\widecheck{C}_-^{cyc}(\La_0,\La_1)$ involves slightly more terms than for the Leibniz rule for the product on the $Z_2$ vector space $C(\La_0,\La_1)$ as done in \cite{L2}, so we detailed it now. We want to prove
\begin{alignat*}{1}
	\widecheck{\mfm}_1\circ\widecheck{\mfm}_2(\gamma_{12}\bs{a}_1,\gamma_{01}\bs{a}_0)+\widecheck{\mfm}_2\big(\gamma_{12}\bs{a}_1,\widecheck{\mfm}_1(\gamma_{01}\bs{a}_0)\big)+\widecheck{\mfm}_2\big(\widecheck{\mfm}_1(\gamma_{12}\bs{a}_1),\gamma_{01}\bs{a}_0\big)=0
\end{alignat*}
The left-hand side can be rewritten
\begin{alignat*}{1}
	&\sum\limits_{\substack{\gamma_{02}\\\bd_0,\bd_1,\bd_2}}\#\cM^0_{\La_{012}}(\gamma_{02};\bd_0,\gamma_{01},\bd_1,\gamma_{12},\bd_2)\cdot\widecheck{\mfm}_1\big(\gamma_{02}\bd_0\bs{a}_0\bd_1\bs{a}_1\bd_2\big)\\
	&\hspace{1cm}+\widecheck{\mfm}_2\big(\gamma_{12}\bs{a}_1,\widecheck{\mfm}_1(\gamma_{01})\bs{a}_0+\gamma_{01}\partial_\Ac(\bs{a}_0)\big)+\widecheck{\mfm}_2\big(\widecheck{\mfm}_1(\gamma_{12})\bs{a}_1+\gamma_{01}\partial_\Ac(\bs{a}_0),\gamma_{01}\bs{a}_0\big)\\
	&=\sum\#\cM^0_{\La_{012}}(\gamma_{02};\bd_0,\gamma_{01},\bd_1,\gamma_{12},\bd_2)\cdot\Big(\widecheck{\mfm}_1(\gamma_{02})\bd_0\bs{a}_0\bd_1\bs{a}_1\bd_2+\gamma_{02}\partial_\Ac(\bd_0\bs{a}_0\bd_1\bs{a}_1\bd_2)\Big)\\
	&\hspace{1cm}+\widecheck{\mfm}_2\big(\gamma_{12}\bs{a}_1,\widecheck{\mfm}_1(\gamma_{01})\bs{a}_0\big)+\widecheck{\mfm}_2\big(\widecheck{\mfm}_1(\gamma_{12})\bs{a}_1,\gamma_{01}\bs{a}_0\big)\\
	&\hspace{1cm}+\widecheck{\mfm}_2\big(\gamma_{12}\bs{a}_1,\gamma_{01}\partial_\Ac(\bs{a}_0)\big)+\widecheck{\mfm}_2\big(\gamma_{01}\partial_\Ac(\bs{a}_0),\gamma_{01}\bs{a}_0\big)
\end{alignat*}
We separate the terms having the boundary of a $\bs{a}_i$ from the other terms to get
\begin{alignat*}{1}
	&=\sum\#\cM^0_{\La_{012}}(\gamma_{02};\bd_0,\gamma_{01},\bd_1,\gamma_{12},\bd_2)\cdot\Big(\widecheck{\mfm}_1(\gamma_{02})\bd_0\bs{a}_0\bd_1\bs{a}_1\bd_2+\gamma_{02}\partial_\Ac(\bd_0)\bs{a}_0\bd_1\bs{a}_1\bd_2\\
	&\hspace{7cm}+\bd_0\bs{a}_0\partial_\Ac(\bd_1)\bs{a}_1\bd_2+\bd_0\bs{a}_0\bd_1\bs{a}_1\partial_\Ac(\bd_2)\Big)\\
	&+\widecheck{\mfm}_2\big(\gamma_{12}\bs{a}_1,\widecheck{\mfm}_1(\gamma_{01})\bs{a}_0\big)+\widecheck{\mfm}_2\big(\widecheck{\mfm}_1(\gamma_{12})\bs{a}_1,\gamma_{01}\bs{a}_0\big)\\
	&+\sum\#\cM^0_{\La_{012}}(\gamma_{02};\bd_0,\gamma_{01},\bd_1,\gamma_{12},\bd_2)\cdot\Big(\gamma_{02}\bd_0\partial_\Ac(\bs{a}_0)\bd_1\bs{a}_1\bd_2+\bd_0\bs{a}_0\bd_1\partial_\Ac(\bs{a}_1)\bd_2\Big)\\
	&+\widecheck{\mfm}_2\big(\gamma_{12}\bs{a}_1,\gamma_{01}\partial_\Ac(\bs{a}_0)\big)+\widecheck{\mfm}_2\big(\gamma_{01}\partial_\Ac(\bs{a}_0),\gamma_{01}\bs{a}_0\big)
\end{alignat*}
The two first lines vanish as the algebraic contributions of pseudo-holomorphic buildings in the boundary of moduli spaces of type $\cM_{\La_{012}}^1(\gamma_{02};\bd_0,\gamma_{01},\bd_1,\gamma_{12},\bd_2)$. The two last lines vanish also, because by definition of the product we have
\begin{alignat*}{1}
	\sum\limits_{\substack{\gamma_{02}\\\bd_0,\bd_1,\bd_2}}\#\cM^0_{\La_{012}}(\gamma_{02};\bd_0,\gamma_{01},\bd_1,\gamma_{12},\bd_2)\cdot\gamma_{02}\bd_0\partial_\Ac(\bs{a}_0)\bd_1\bs{a}_1\bd_2=\widecheck{\mfm}_2\big(\gamma_{12}\bs{a}_1,\gamma_{01}\partial_\Ac(\bs{a}_0)\big)
\end{alignat*}
and 
\begin{alignat*}{1}
	\sum\limits_{\substack{\gamma_{02}\\\bd_0,\bd_1,\bd_2}}\#\cM^0_{\La_{012}}(\gamma_{02};\bd_0,\gamma_{01},\bd_1,\gamma_{12},\bd_2)\cdot\gamma_{02}\bd_0\bs{a}_0\bd_1\partial_\Ac(\bs{a}_1)\bd_2=\widecheck{\mfm}_2\big(\gamma_{12}\partial_\Ac(\bs{a}_1),\gamma_{01}\bs{a}_0\big)
\end{alignat*}
and so the Leibniz rule is satisfied.
Moreover, it follows from \cite[Theorem 5.5]{EESa} that the product $\widecheck{\mfm}_2$ is unital with the unit given by the minimum Morse Reeb chord, i.e. we have $\widecheck{\mfm}_1(y_{01})=\widecheck{\mfm}_1(y_{12})=0$ and $\widecheck{\mfm}_2(\gamma_{12}\bs{a}_1,y_{01})=\gamma_{02}\bs{a}_1$ and $\widecheck{\mfm}_2(y_{12},\gamma_{01}\bs{a}_0)=\gamma_{02}\bs{a}_0$, for all $\gamma_{01}\in\Rc(\La_0,\La_1)$, $\gamma_{12}\in\Rc(\La_1,\La_2)$ and $\gamma_{02}\in\Rc(\La_0,\La_2)$ corresponding to the \textit{same} chord (either a long chord or a Morse chord).
\begin{nota}
	In the following, we will write only $\gamma_{01}$ instead of the more general $\gamma_{01}\bs{a}$ for an element in $\widecheck{C}_-^{cyc}(\La_0,\La_1)$. This is just in order to reduce a bit the length of formulas. We will also do the same for elements in $\widehat{C}^{cyc}_+(\La_0,\La_1)$. In particular we'll define a product $\widehat{\mfm}_2$ only for pairs of inputs $(\gamma_{21},\gamma_{10})$, where the inputs can also be maximum Morse Reeb chords, but keeping in mind that the rule to define more generally $\widehat{\mfm}_2(\gamma_{21}\bs{a}_1,\gamma_{10}\bs{a}_0)$ is the same as for the product $\widecheck{\mfm}_2$. Namely, the output will contain the words $\bs{a}_0$ and $\bs{a}_1$ in a larger word of pure Reeb chords obtained by following the boundary of pseudo-holomorphic discs involved in the definition of $\widehat{\mfm}_2(\gamma_{21},\gamma_{10})$ counter clockwise.
\end{nota}

So let us now construct this product on $\widehat{C}^{cyc}_+(\La_0,\La_1)$. From a $3$-copy of $\La$ we define a map
\begin{alignat*}{1}
	\widehat{\mfm}_2:\widehat{C}^{cyc}_+(\La_1,\La_2)\otimes\widehat{C}^{cyc}_+(\La_0,\La_1)\to\widehat{C}^{cyc}_+(\La_0,\La_2)
\end{alignat*}
by counting $2$-levels pseudo-holomorphic buildings as shown on Figure \ref{fig:secondary_prod}. More precisely, for generators $\gamma_{10},x_{01}$ of $\widehat{C}^{cyc}_+(\La_0,\La_1)$ and generators $\gamma_{21},x_{12}$ of $\widehat{C}^{cyc}_+(\La_1,\La_2)$ we have:
\small
\begin{alignat*}{1}
	&\widehat{\mfm}_2(\gamma_{21},\gamma_{10})=\\
	&\sum_{\substack{\beta_{12}\in\Rc^{l}(\La_1,\La_2)\\\cup\{y_{12}\}}}\sum\limits_{\substack{\bd_0,\bd_1,\bd_1'\\ \bd_2,\bd_2'}}\#\cM^0_{\La_{012}}(x_{02};\bd_0,\gamma_{10},\bd_1,\beta_{12},\bd_2')\cdot\#\cM^0_{\La_{12}}(\beta_{12};\bd_1',\gamma_{21},\bd_2)\cdot x_{02}\bd_0\bd_1\bd_1'\bd_2\bd_2'\\
	&+\sum_{\gamma_{20}}\sum_{\substack{\beta_{12}\in\Rc^{l}(\La_1,\La_2)\\\cup\{y_{12}\}}}\sum\limits_{\substack{\bd_0,\bd_1,\bd_1'\\ \bd_2,\bd_2'}}\#\cM^0_{\La_{012}}(\gamma_{20};\bd_0,\gamma_{10},\bd_1,\beta_{12},\bd_2')\#\cM^0_{\La_{12}}(\beta_{12};\bd_1',\gamma_{21},\bd_2)\cdot \gamma_{20}\bd_0\bd_1\bd_1'\bd_2\bd_2'
\end{alignat*}
\begin{alignat*}{1}
	&\widehat{\mfm}_2(\gamma_{21},x_{01})=\sum_{\beta_{12}}\sum\limits_{\substack{\bd_0,\bd_1,\bd_1'\\ \bd_2,\bd_2'}}\#\cM^0_{\La_{012}}(x_{02};\bd_0,x_{01},\bd_1,\beta_{12},\bd_2')\#\cM^0_{\La_{12}}(\beta_{12};\bd_1',\gamma_{21},\bd_2)\cdot x_{02}\bd_0\bd_1\bd_1'\bd_2\bd_2'\\
	&\widehat{\mfm}_2(x_{12},\gamma_{10})=\sum_{\beta_{12}}\sum\limits_{\substack{\bd_0,\bd_1,\bd_1'\\ \bd_2,\bd_2'}}\#\cM^0_{\La_{012}}(x_{02};\bd_0,\gamma_{10},\bd_1,\beta_{12},\bd_2')\#\cM^0_{\La_{12}}(\beta_{12};\bd_1',x_{12},\bd_2)\cdot x_{02}\bd_0\bd_1\bd_1'\bd_2\bd_2'\\
	&\hspace{15mm}+\sum_{\gamma_{20},\beta_{12}}\sum\limits_{\substack{\bd_0,\bd_1,\bd_1'\\ \bd_2,\bd_2'}}\#\cM^0_{\La_{012}}(\gamma_{20};\bd_0,\gamma_{10},\bd_1,\beta_{12},\bd_2')\#\cM^0_{\La_{12}}(\beta_{12};\bd_1',x_{12},\bd_2)\cdot \gamma_{20}\bd_0\bd_1\bd_1'\bd_2\bd_2'\\
	&\widehat{\mfm}_2(x_{12},x_{01})=0
\end{alignat*}
\normalsize
Observe that in the definition of $\widehat{\mfm}_2(\gamma_{21},x_{01})$ and $\widehat{\mfm}_2(x_{12},\gamma_{10})$, the ``connecting'' chord $\beta_{12}$ will automatically be in $\Rc^{l}(\La_1,\La_2)\cup\{y_{12}\}$. In the first case, it has to be a Morse chord for action reasons, and can not be $x_{12}$ for degree reason. In the second case it can not be $x_{12}$ also for degree reasons.

\begin{figure}[ht]  
	\labellist
	\footnotesize
	\pinlabel $x_{02}^{\text{out}}$ at 85 175
	\pinlabel in at 40 175
	\pinlabel in at 100 90
	\pinlabel out at 230 75
	\pinlabel in at 205 175
	\pinlabel in at 145 90
	\pinlabel $x_{02}^{\text{out}}$ at 295 175
	\pinlabel $x_{01}^{\text{in}}$ at 270 70
	\pinlabel in at 355 90
	\pinlabel $x_{02}^{\text{out}}$ at 450 175
	\pinlabel $x_{12}^{\text{in}}$ at 425 -10
	\pinlabel in at 400 175
	\pinlabel in at 520 175
	\pinlabel out at 541 75
	\pinlabel $x_{12}^{\text{in}}$ at 495 -10
	\pinlabel \textit{not} at 120 20
	\pinlabel $x_{12}$ at 120 8
	\endlabellist
	\normalsize
	\begin{center}\includegraphics[width=12cm]{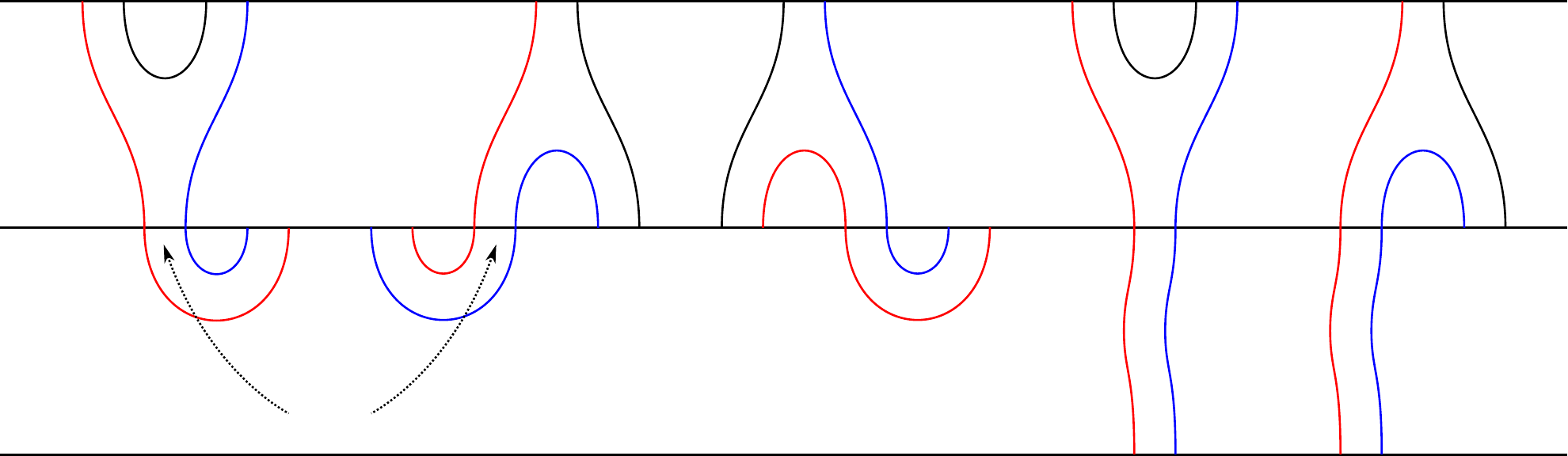}\end{center}
	\caption{Pseudo-holomorphic buildings contributing to the product $\widehat{\mfm}_2$ on $\widehat{C}^{cyc}_+(\La_0,\La_1)$.}
	\label{fig:secondary_prod}
\end{figure}
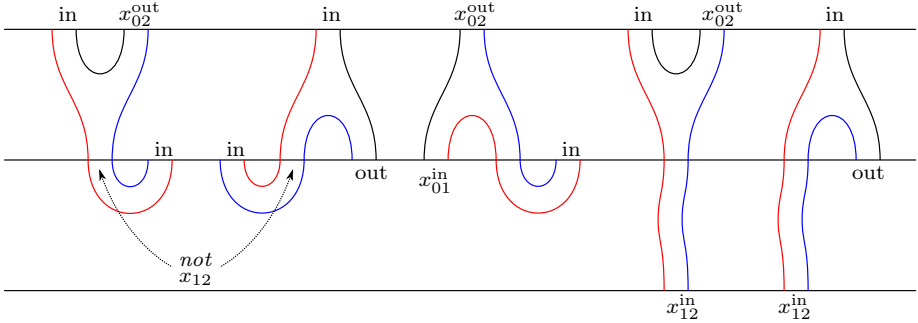 
By \cite[Proposition 2]{L2} we can check that this map $\widehat{\mfm}_2$ is of degree $0$. Then we have:
\begin{prop}
	The map $\widehat{\mfm}_2$ satisfies $\widehat{\mfm}_1\circ\widehat{\mfm}_2+\widehat{\mfm}_2(\id\otimes\widehat{\mfm}_1)+\widehat{\mfm}_2(\widehat{\mfm}_1\otimes\id)=0$, i.e. $\widehat{\mfm}_2$ descends to a well-defined map on homology.
\end{prop}

\begin{proof}
We prove the proposition for each type of pair of inputs. For a pair of inputs $(\gamma_{21},\gamma_{10})$, we obtain the Leibniz rule by considering the algebraic contributions of pseudo-holomorphic buildings in the boundary of the compactification of the following products of moduli spaces:
\begin{alignat}{1}
	&\cM^1(x_{02};\bd_0,\gamma_{10},\bd_1,\beta_{12},\bd_2')\times\cM^0(\beta_{12};\bd_1',\gamma_{21},\bd_2)\label{leibniz11}\\
	&\cM^0(x_{02};\bd_0,\gamma_{10},\bd_1,\beta_{12},\bd_2')\times\cM^1(\beta_{12};\bd_1',\gamma_{21},\bd_2)\label{leibniz12}\\
	&\cM^1(\gamma_{20};\bd_0,\gamma_{10},\bd_1,\beta_{12},\bd_2')\times\cM^0(\beta_{12};\bd_1',\gamma_{21},\bd_2)\label{leibniz21}\\
	&\cM^0(\gamma_{20};\bd_0,\gamma_{10},\bd_1,\beta_{12},\bd_2')\times\cM^1(\beta_{12};\bd_1',\gamma_{21},\bd_2)\label{leibniz22}
\end{alignat}
See Figure \ref{fig:sec_prod_leibniz1} for buildings in the boundary of the compactification of \eqref{leibniz11} and \eqref{leibniz12}, and Figure \ref{fig:sec_prod_leibniz2} for those in the boundary of the compactification of \eqref{leibniz21} and \eqref{leibniz22}. In each case (and in all the rest of the proof), we omit to draw the buildings which appear twice, i.e. both in the boundary of \eqref{leibniz11} and \eqref{leibniz12} or both in the boundary of \eqref{leibniz21} and \eqref{leibniz22}, because they cancel each other over $\Z_2$.
\begin{figure}[ht] 
	\labellist
	\footnotesize
	\pinlabel in at 72 45
	\pinlabel in at 110 45
	\pinlabel in at 330 45
	\pinlabel in at 490 45
	\pinlabel in at 13 127
	\pinlabel $x_{02}^{out}$ at 230 130
	\pinlabel $x_{02}^{out}$ at 315 130
	\pinlabel in at 367 130
	\pinlabel in at 580 130
	\pinlabel in at 675 130
	\pinlabel $x_{02}^{out}$ at 60 215
	\pinlabel in at 170 212
	\pinlabel in at 275 212
	\pinlabel $x_{02}^{out}$ at 435 215
	\pinlabel in at 518 212
	\pinlabel $x_{02}^{out}$ at 565 215
	\pinlabel in at 617 212
	\pinlabel $x_{02}^{out}$ at 662 215
	\pinlabel \textbf{A} at 30 170
	\pinlabel \textbf{B} at 135 170
	\pinlabel \textbf{C} at 290 170
	\pinlabel \textbf{D} at 390 170
	\pinlabel \textbf{E} at 500 170
	\pinlabel \textbf{F} at 675 170
	\pinlabel $x_{12}$ at 610 0
	\endlabellist
	\normalsize
	\begin{center}\includegraphics[width=14cm]{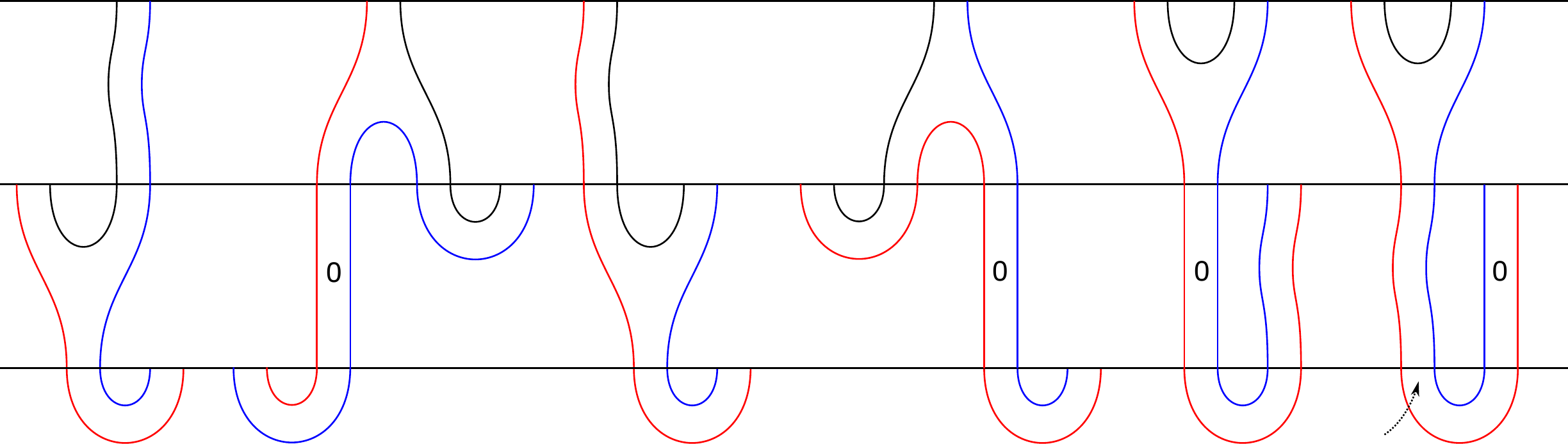}\end{center}
	\caption{Pseudo-holomorphic buildings in the boundary of \eqref{leibniz11} and \eqref{leibniz12}.}
	\label{fig:sec_prod_leibniz1}
\end{figure} 
\begin{figure}[ht]  
\labellist
\footnotesize
\pinlabel in at 38 125
\pinlabel in at 97 210
\pinlabel $\gamma_{20}^{out}$ at 125 20
\pinlabel in at 172 45
\pinlabel in at 235 210
\pinlabel $\gamma_{20}^{out}$ at 263 20
\pinlabel in at 315 130
\pinlabel in at 375 210
\pinlabel $\gamma_{20}^{out}$ at 400 105
\pinlabel \textbf{G} at 60 170
\pinlabel \textbf{H} at 200 170
\pinlabel \textbf{I} at 330 170
\endlabellist
\normalsize
\begin{center}\includegraphics[width=9cm]{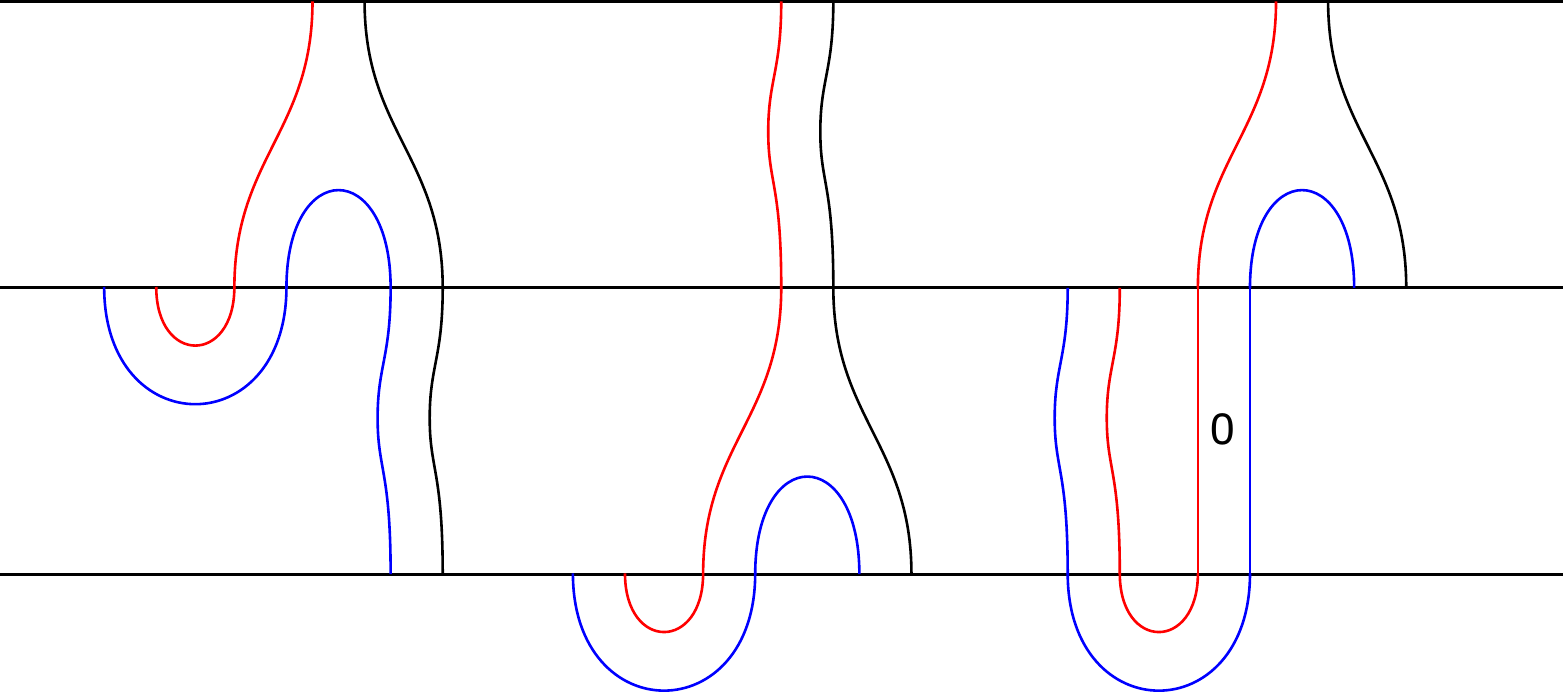}\end{center}
	\caption{Pseudo-holomorphic buildings in the boundary of \eqref{leibniz21} and \eqref{leibniz22}.}
	\label{fig:sec_prod_leibniz2}
\end{figure}

According to Lemma \ref{lem:bananas} (3), the building \textbf{A} either never appears or appears twice, so we can ignore it.

Consider now the building \textbf{B}. Removing the banana containing the output $x_{02}$ puncture gives a building contributing to the component of $\widehat{\mfm}_2(\gamma_{21},\gamma_{10})$ taking values in the generators of $\widehat{C}_+^{cyc}(\La_0,\La_2)$ of the form $\beta_{20}\bd_0\bd_1\bd_2$. Adding the banana with positive asymptotic at $x_{02}$ and using Lemma \ref{lem:bananas} (2), we get that the building contributes $x_{02}\beta\bd_0\bd_1\bd_2+x_{02}\bd_0\bd_1\bd_2\beta$, where $\beta$ is the pure Reeb chord of $\La$ corresponding to $\beta_{01}$. Thus, this type \textbf{B} of buildings, together with \textbf{G}, contribute to $\widehat{\mfm}_1\circ\widehat{\mfm}_2(\gamma_{21},\gamma_{10})$.

The buildings of type \textbf{C} and \textbf{D}, together with \textbf{H} contribute to $\widehat{\mfm}_2(\gamma_{21},\widehat{\mfm}_1(\gamma_{10}))$. Indeed the only slightly subtle thing here is about building \textbf{D}. By assumption, the mixed connecting chord from $\La_2$ to $\La_1$ is not the maximum Morse Reeb chord $x_{12}$. But then, if such a building exists  this connecting chord must be the minimum Morse chord $y_{12}$ for action reasons, because the component of the building with three mixed asymptotics has a positive asymptotic at the Morse chord $x_{02}$. For degree reasons, the connecting mixed chord from $\La_1$ to $\La_0$ must then be $x_{01}$ and the banana having it as a positive asymptotic contributes to the component of $\widehat{\mfm}_1(\gamma_{01})$ taking values in $\langle x_{01}\rangle_{\Ac-\Ac}^{cyc}$ (this component actually vanishes for elements $\gamma_{10}\bs{a}$ where $\bs{a}=1$, but it doesn't in the general case).

Finally, let us consider the building \textbf{F}. Usually, these types of buildings cancel by pairs but because of the assumption about the connecting chord, the buildings of type \textbf{F} arise by degenerating the banana but will never arise by degenerating the top level with three mixed asymptotic. Observe then that the buildings \textbf{E}, \textbf{F} and \textbf{I} contribute to $\widehat{\mfm}_2(\widehat{\mfm}_1(\gamma_{21}),\gamma_{10})$, and we have thus proved the Leibniz rule for the pair of inputs $(\gamma_{21},\gamma_{10})$.

For a pair of inputs $(\gamma_{21},x_{01})$, the Leibniz rule restricts to $\widehat{\mfm}_2(\widehat{\mfm}_1(\gamma_{21}),x_{01})=0$
because $\widehat{\mfm}_1$ vanishes on the maximum Reeb chord. Let us consider the pseudo-holomorphic buildings in the boundary of the compactification of the following products of moduli spaces:
\begin{alignat}{1}
	&\cM^1(x_{02};\bd_0,x_{01},\bd_1,\beta_{12},\bd_2')\times\cM^0(\beta_{12};\bd_1',\gamma_{21},\bd_2)\label{leibniz31}\\
	&\cM^0(x_{02};\bd_0,x_{01},\bd_1,\beta_{12},\bd_2')\times\cM^1(\beta_{12};\bd_1',\gamma_{21},\bd_2)\label{leibniz32}
\end{alignat}
See Figure \ref{fig:sec_prod_leibniz3} for a schematic picture of these buildings. As before, the first building can be ignored. The second building never appears either, for action reasons.
The third building finally contributes to $\widehat{\mfm}_2(\widehat{\mfm}_1(\gamma_{21}),x_{01})$.
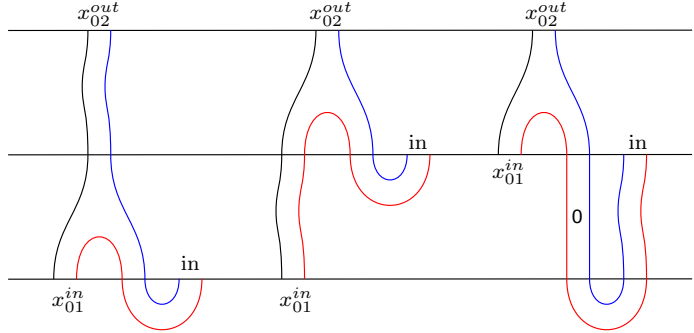
\begin{figure}[ht]  
\labellist
\footnotesize
	\pinlabel $x_{01}^{in}$ at 40 20
	\pinlabel $x_{02}^{out}$ at 60 210
	\pinlabel in at 120 45
	\pinlabel $x_{02}^{out}$ at 213 210
	\pinlabel $x_{01}^{in}$ at 190 20
	\pinlabel in at 270 125
	\pinlabel $x_{01}^{in}$ at 330 105
	\pinlabel $x_{02}^{out}$ at 355 210
	\pinlabel in at 415 125
\endlabellist
\normalsize
\begin{center}\includegraphics[width=9cm]{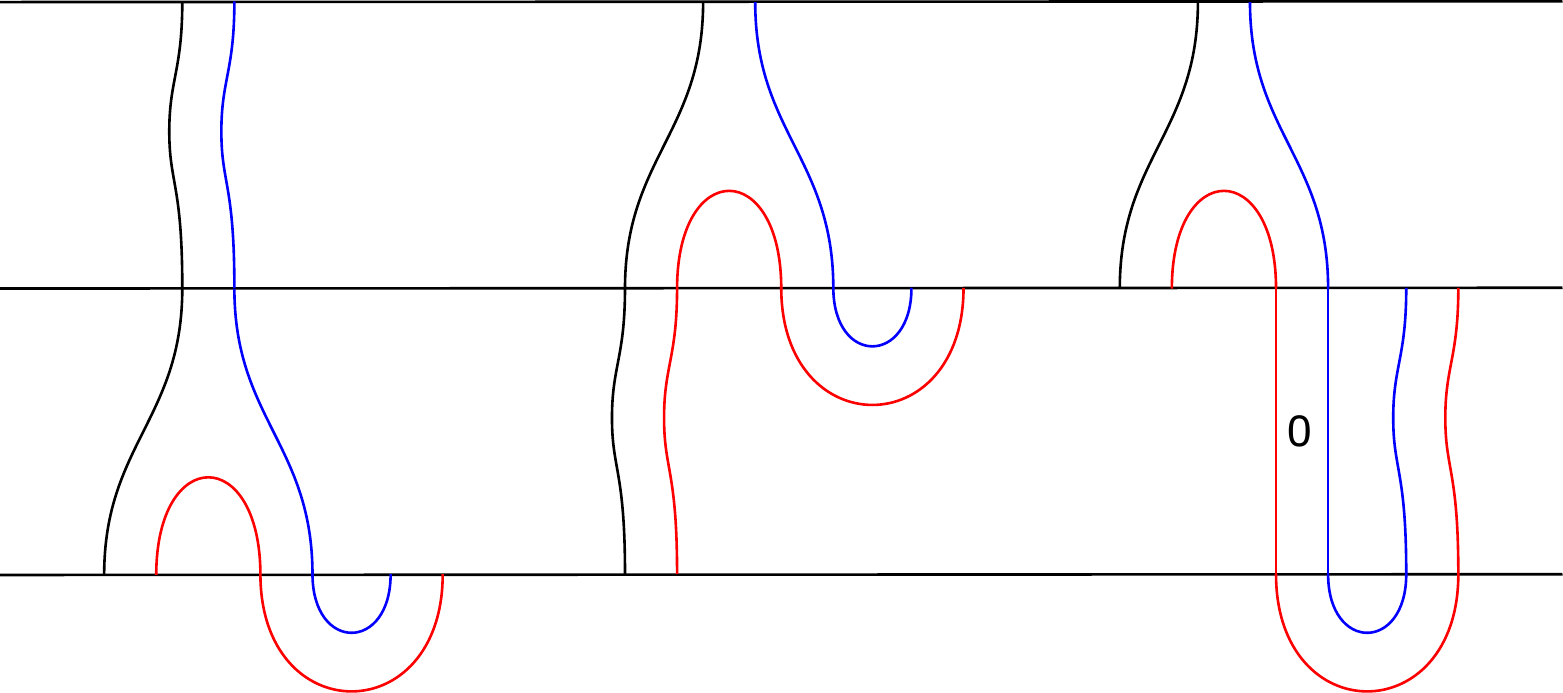}\end{center}
\caption{Pseudo-holomorphic buildings in the boundary of \eqref{leibniz31} and \eqref{leibniz32}.}
\label{fig:sec_prod_leibniz3}
\end{figure}
For the pair of inputs $(x_{12},\gamma_{10})$, the Leibniz rule restricts to $\widehat{\mfm}_1\circ\widehat{\mfm}_2(x_{12},\gamma_{10})+\widehat{\mfm}_2(x_{12},\widehat{\mfm}_1(\gamma_{10}))=0$. We consider the boundary of the compactification of
\begin{alignat}{1}
	&\cM^1(x_{02};\bd_0,\gamma_{10},\bd_1,\beta_{12},\bd_2')\times\cM^0(\beta_{12};\bd_1',x_{12},\bd_2)\label{sec_prod21}\\
	&\cM^0(x_{02};\bd_0,\gamma_{10},\bd_1,\beta_{12},\bd_2')\times\cM^1(\beta_{12};\bd_1',x_{12},\bd_2)\label{sec_prod22}
\end{alignat}
and of 
\begin{alignat}{1}
	&\cM^1(\gamma_{20};\bd_0,\gamma_{10},\bd_1,\beta_{12},\bd_2')\times\cM^0(\beta_{12};\bd_1',x_{12},\bd_2)\label{sec_prod23}\\
	&\cM^0(\gamma_{20};\bd_0,\gamma_{10},\bd_1,\beta_{12},\bd_2')\times\cM^1(\beta_{12};\bd_1',x_{12},\bd_2)\label{sec_prod24}
\end{alignat}
whose different components are schematized in Figure \ref{fig:sec_prod2_leibniz}
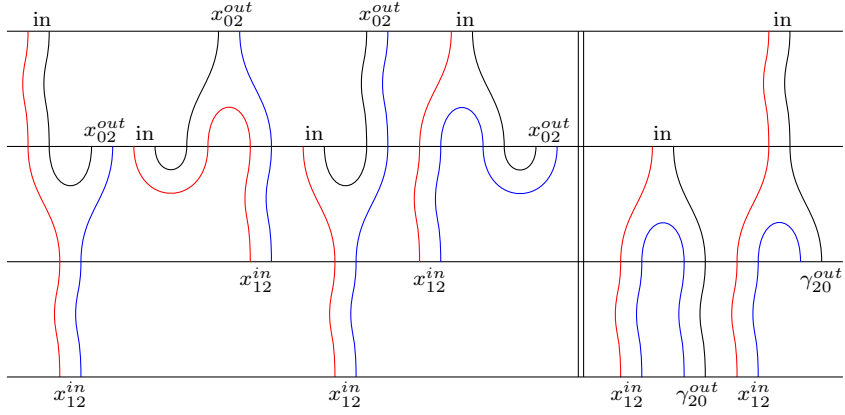
\begin{figure}[ht]  
	\labellist
	\footnotesize
		\pinlabel in at 25 257
		\pinlabel $x_{12}^{in}$ at 45 -12
		\pinlabel $x_{02}^{out}$ at 70 177
		\pinlabel in at 98 175
		\pinlabel $x_{12}^{in}$ at 177 70
		\pinlabel $x_{02}^{out}$ at 160 260
		\pinlabel in at 218 175
		\pinlabel $x_{12}^{in}$ at 240 -12
		\pinlabel $x_{02}^{out}$ at 265 260	
		\pinlabel in at 325 257
		\pinlabel $x_{12}^{in}$ at 300 70
		\pinlabel $x_{02}^{out}$ at 385 177
		\pinlabel in at 465 175
		\pinlabel $x_{12}^{in}$ at 440 -12
		\pinlabel $\gamma_{20}^{out}$ at 490 -12
		\pinlabel in at 550 257
		\pinlabel $x_{12}^{in}$ at 530 -12
		\pinlabel $\gamma_{20}^{out}$ at 578 70
		\endlabellist
	\normalsize
	\begin{center}\includegraphics[width=11cm]{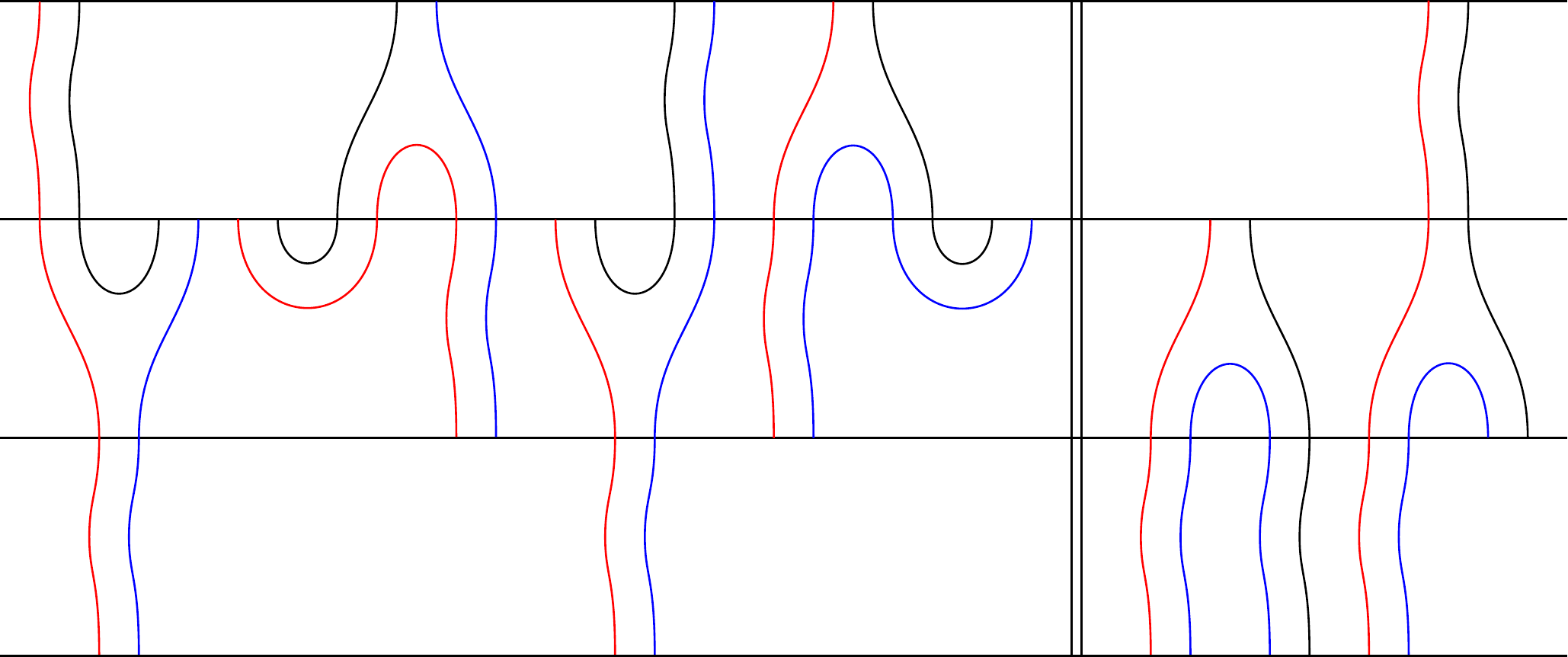}\end{center}
	\caption{Left: pseudo-holomorphic buildings in the boundary of \eqref{sec_prod21} and \eqref{sec_prod22}; right: pseudo-holomorphic buildings in the boundary of \eqref{sec_prod23} and \eqref{sec_prod24}.}
	\label{fig:sec_prod2_leibniz}
\end{figure}
By arguments similar as before, the algebraic contributions of the second and third buildings vanish. The first and sixth contribute to $\widehat{\mfm}_2(x_{12},\widehat{\mfm}_1(\gamma_{10}))$, and the fourth and fifth to $\widehat{\mfm}_1\circ\widehat{\mfm}_2(x_{12},\gamma_{10})$.
Finally, for a pair $(x_{12},x_{01})$ all the terms of the Leibinz rule vanish by definition.
\end{proof}

\begin{rem}
	Observe that the product on $\widehat{C}^{cyc}_+(\La_0,\La_1)$ is defined by a count of $2$-levels pseudo-holomorphic buildings, which arise in the boundary of the compactification of $1$-dimensional moduli spaces. In particular, there is no canonical choice for the buildings we choose to count to define the product, but there is a choice up to homotopy. In particular, we could define a map
	\begin{alignat*}{1}
		\widehat{\bs{d}}_2:\widehat{C}^{cyc}_+(\La_1,\La_2)\otimes\widehat{C}^{cyc}_+(\La_0,\La_1)\to\widehat{C}^{cyc}_+(\La_0,\La_2)
	\end{alignat*}	
by $\widehat{\bs{d}}_2(a_1,a_0)=\widehat{\mfm}_2(a_1,a_0)$ for $(a_1,a_0)$ being of type $(\gamma_{21},\gamma_{10})$, $(\gamma_{21},x_{01})$ or $(x_{12},x_{01})$, and then set
\begin{alignat*}{1}
	\widehat{\bs{d}}_2(x_{12},\gamma_{01})=\sum_{\beta_{01}}\sum\limits_{\substack{\bd_0,\bd_0'\\ \bd_1,\bd_1',\bd_2}}\#\cM^0_{\La_{012}}(x_{02};\bd_0,\beta_{01},\bd_1',x_{12},\bd_2)\#\cM^0_{\La_{01}}(\beta_{01};\bd_0',\gamma_{10},\bd_1)\cdot x_{02}\bd_0\bd_0'\bd_1\bd_1'\bd_2
\end{alignat*}
See Figure \ref{fig:sec_prod3}. The maps $\widehat{\bs{d}}_2$ and $\widehat{\mfm}_2$ are homotopic via a (degree $-1$) homotopy $h:\widehat{C}^{cyc}_+(\La_1,\La_2)\otimes\widehat{C}^{cyc}_+(\La_0,\La_1)\to\widehat{C}^{cyc}_+(\La_0,\La_2)$
defined by
\begin{alignat*}{1}
	h(x_{12}\bs{a}_1,\gamma_{01}\bs{a}_0)&=\sum\limits_{\bd_0,\bd_1,\bd_2}\#\cM^0(x_{02};\bd_0,\gamma_{10},\bd_1,x_{12},\bd_2)\cdot x_{02}\bd_0\bs{a}_0\bd_1\bs{a}_1\bd_2\\
	&+\sum_{\gamma_{20}}\sum\limits_{\bd_0,\bd_1 \bd_2}\#\cM^0(\gamma_{20};\bd_0,\gamma_{10},\bd_1,x_{12},\bd_2)\cdot \gamma_{20}\bd_0\bs{a}_0\bd_1\bs{a}_1\bd_2
\end{alignat*}
and $h$ vanishes for other pairs of generators. Indeed, by studying the boundary of the compactification of $1$-dimensional moduli spaces as the one used to define $h$, one can check that
\begin{alignat*}{1}
	\widehat{\mfm}_2+\widehat{d}_2=h(\id\otimes\widehat{\mfm}_1)+h(\widehat{\mfm}_1\otimes\id)+\widehat{\mfm}_1\circ h
\end{alignat*}
Although the pseudo-holomorphic buildings used to define $\widehat{\bs{d}}_2$ have some ``symmetry'' (i.e. the buildings contributing to $\widehat{\bs{d}}_2(\gamma_{21},x_{01})$ and $\widehat{\bs{d}}_2(x_{12},\gamma_{10})$ are symmetric to each other), we chose other buildings to define the product $\widehat{\mfm}_2$. The main reason for this choice is that it is easier to then find formulas to generalize this product to a family of maps $\{\widehat{\mfm}_d\}$ satisfying the $A_\infty$-equations, see Section \ref{sec:A-infty}.
\begin{figure}[ht]  
	\labellist
	\footnotesize
		\pinlabel in at 45 41
		\pinlabel $x_{12}^{in}$ at 130 23
		\pinlabel $x_{02}^{out}$ at 107 127
	\endlabellist
	\normalsize
	\begin{center}\includegraphics[width=4cm]{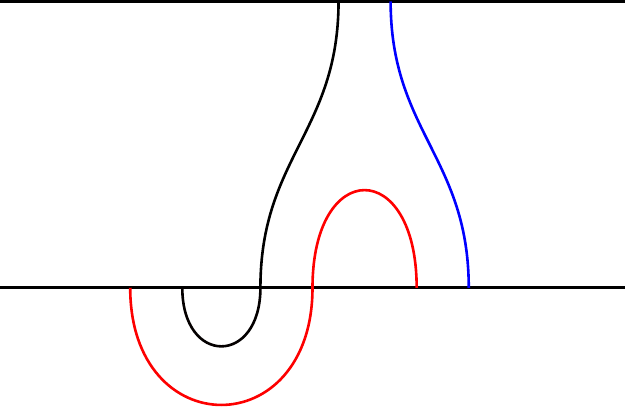}\end{center}
	\caption{Pseudo-holomorphic buildings contributing to $\widehat{\bs{d}}_2(x_{12},\gamma_{01})$.}
	\label{fig:sec_prod3}
\end{figure}
\end{rem}

\begin{rem}
	The pseudo-holomorphic buildings contributing to the component of the product $\widehat{\mfm}_2$ with long chords as inputs and output can also be used to build a product structure on the linearized Legendrian contact \textit{homology} complex; product which will be unital in case the Legendrian submanifold is horizontally displaceable.
	
	It is well known that there is a (non unital) product structure (even an $A_\infty$-structure) on the linearized Legendrian contact cohomology, defined first in \cite{CKESW} and generalized to the bilinearized case in \cite{BCh}. 
	This product can be computed directly from the C-E algebra, and equivalently by a count of pseudo-holomorphic discs with boundary on a $3$-copy $\La_0\cup\La_1\cup\La_2$ of $\La$ negatively asymptotic to a chord from $\La_0$ to $\La_1$ and to a chord from $\La_1$ to $\La_2$ (the inputs), and positively asymptotic to a chord from $\La_0$ to $\La_2$ (the output), and potentially having pure negative asymptotics which are augmented. In this case, the mixed chord considered are never Morse chords.
	
	In \cite{NRSSZ}, for knots in $\R^3$ the authors consider the complex generated by mixed chords from $\La_1$ to $\La_0$ in a $2$-copy of $\La$, which they denote $\Hom_+(\ep_0,\ep_1)$. They define then a product (as well as an $A_\infty$-structure) on this complex by a count of similar curves as above (two negative mixed inputs and one positive mixed output) but the main difference is that in this case the Morse chords can be inputs and outputs. The homology of the complex $\Hom_+(\ep_0,\ep_1)$ is isomorphic to the Legendrian contact homology of $\La$. This is implied by the acyclicity of the complex of the $2$-copy which holds for knots in $\R^3$, see \cite[Proposition 5.4]{NRSSZ}. In particular, the product on $\Hom_+$ doesn't give canonically a product on Legendrian contact homology.
	
	By \cite[Proposition 2.7]{CDGG1}, the (bi)linearization of the DG-bimodule $(C_+(\La_0,\La_1),\bs{\D}_1)$ by augmentations $\ep_0,\ep_1$ is canonically isomorphic to the bilinearized Legendrian contact homology complex $LCC_*^{\ep_0,\ep_1}(\La)$.
	We claim that the only-long-chords-asymptotics component of the product $\widehat{\mfm}_2$, when linearized by augmentations, gives a product on the Legendrian contact homology complex, the proof of this being schematized in Figure \ref{fig:sec_prod_leibniz2}. In the presence of a filling of $\La$ and under the hypothesis of horizontal displaceability, it is possible to prove that this new product on Legendrian contact homology is isomorphic to the product on the $\Hom_+$ complex, because both are isomorphic to the product on the singular cohomology of the filling (through the Ekholm-Seidel isomorphism \cite{E1,DR}). This will be investigated more precisely in a forthcoming paper with Georgios Dimitroglou-Rizell, where we describe a \textit{relative Calabi-Yau structure} carried by the C-E algebra. This relative structure allows in particular to show the isomorphism between the products on $LCH_*$ and $\Hom_+$ without the presence of a filling.
\end{rem}

\subsection{Products under the Calabi-Yau morphism}
Let us consider the Calabi-Yau map $\CY$ induced on the $\Z_2$-modules, we denote it $\CY_1:\widehat{C}_+^{cyc}(\La_0,\La_1)\to\widecheck{C}_-^{cyc}(\La_0,\La_1)$.
\begin{teo}\label{teo:CY2}
	The map $\CY_1$ preserves the product structures in homology, i.e. the relation
	$$\CY_1\circ\,\widehat{\mfm}_2+\widecheck{\mfm}_2\big(\CY_1,\CY_1\big)=0$$ is satisfied in homology.
\end{teo}

\begin{proof}
In this proof we will use the notations $\widehat{\mfm}_2^+$ and $\widehat{\mfm}_2^x$ to denote the components of $\widehat{\mfm}_2$ with values in $C_+^{cyc}(\La_0,\La_1)$ and $\langle x_{01}\rangle_{\Ac-\Ac}^{cyc}$ respectively.

Given a $3$-copy $\La_0\cup\La_1\cup\La_2$ of $\La$, we define a (degree $-1$) map
\begin{alignat*}{1}
	\CY_2:\widehat{C}_+^{cyc}(\La_1,\La_2)\otimes\widehat{C}_+^{cyc}(\La_0,\La_1)\to\widecheck{C}_-^{cyc}(\La_0,\La_2)
\end{alignat*}
by a count of pseudo-holomorphic buildings as shown in Figure \ref{fig:CY2}. Similarly as for the product $\widehat{\mfm}_2$ we require that the connecting chord from $\La_2$ to $\La_1$ in the buildings is not the maximum Morse Reeb chord $x_{12}$. Observe however that this is automatically satisfied for action reasons for the buildings of types \textbf{C} and \textbf{D}.

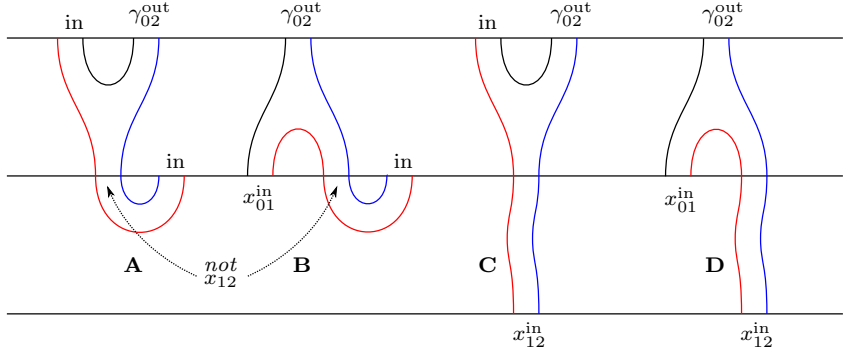
\begin{figure}[ht]  
	\labellist
	\footnotesize
		\pinlabel $\gamma_{02}^\text{out}$ at 85 180
	\pinlabel in at 40 175
	\pinlabel in at 100 92
	\pinlabel \textbf{A} at 75 30
		\pinlabel $\gamma_{02}^\text{out}$ at 175 180
	\pinlabel $x_{01}^\text{in}$ at 150 70
	\pinlabel in at 235 92
	\pinlabel \textbf{B} at 175 30
		\pinlabel $\gamma_{02}^\text{out}$ at 335 180
	\pinlabel in at 285 175
	\pinlabel $x_{12}^\text{in}$ at 310 -12
	\pinlabel \textbf{C} at 285 30
		\pinlabel $\gamma_{02}^\text{out}$ at 425 180
	\pinlabel $x_{12}^\text{in}$ at 445 -12
	\pinlabel $x_{01}^\text{in}$ at 400 69
	\pinlabel \textbf{D} at 420 30
	\pinlabel \textit{not} at 127 30
	\pinlabel $x_{12}$ at 127 20
	\endlabellist
	\normalsize
	\begin{center}\includegraphics[width=11cm]{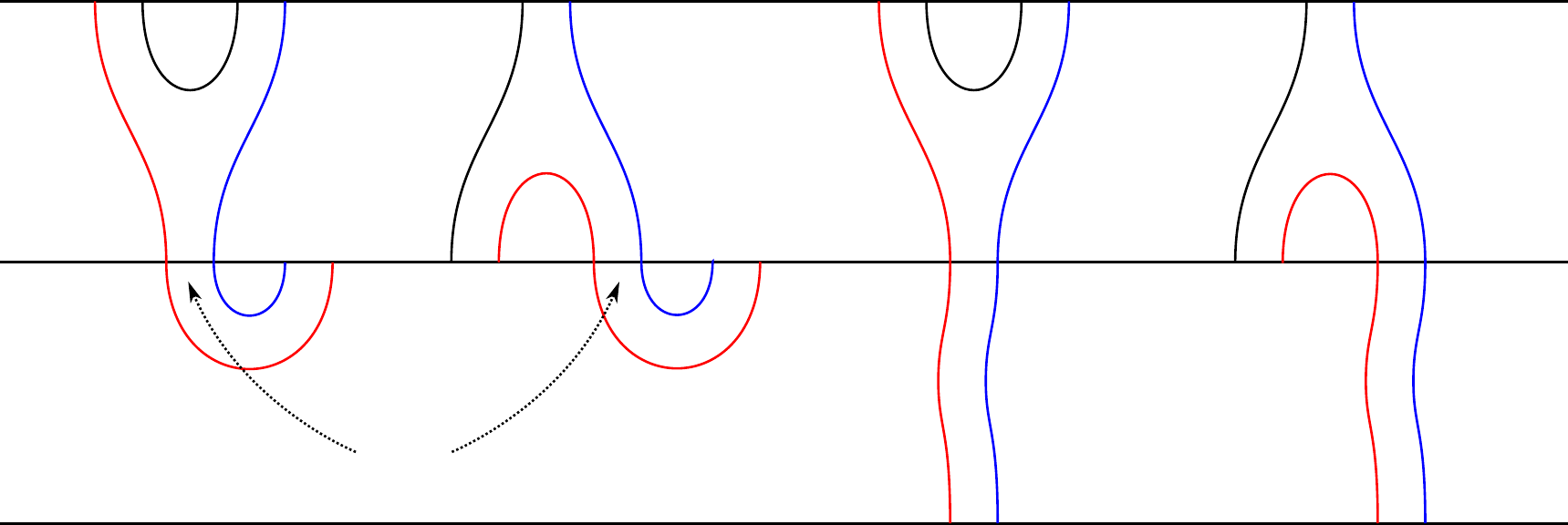}\end{center}
	\caption{Pseudo-holomorphic buildings contributing to the map $\CY_2$.}
	\label{fig:CY2}
\end{figure}
By considering buildings in the boundary of the compactification of one-dimensional products of moduli spaces of types \textbf{A}, \textbf{B}, \textbf{C} and \textbf{D} in Figure \ref{fig:CY2}, one proves that the following relation is satisfied:
\begin{alignat}{1}
	\widecheck{\mfm}_1\circ\CY_2&+\widecheck{\mfm}_2\big(\CY_1\otimes\CY_1\big)+\CY_1\circ\widehat{\mfm}_2+\CY_2(\id\otimes\,\widehat{\mfm}_1)+\CY_2(\widehat{\mfm}_1\otimes\id)=0\label{rel:functor2}
\end{alignat}
which shows that $\CY_1$ preserves the products in homology.
In Figure \ref{fig:CY2dege1} we depicted the different types of buildings in the boundary of the compactification of $1$-dimensional products of moduli spaces of type \textbf{A}. The buildings of type \textbf{A1} together with those of type \textbf{A3} where the connecting chord from $\La_1$ to $\La_0$ is $x_{01}$, contribute to $\CY_2(\gamma_{21},\widehat{\mfm}_1(\gamma_{10}))$. Those of type \textbf{A2} contribute to $\CY_1\circ\widehat{\mfm}_2^x(\gamma_{21},\gamma_{10})$ when the connecting chord from $\La_2$ to $\La_0$ is $x_{02}$, and to $\widecheck{\mfm}_1\circ\CY_2(\gamma_{21},\gamma_{10})$ otherwise. The buildings of type \textbf{A3} when the connecting chord from $\La_1$ to $\La_0$ is not $x_{01}$ contribute to $\widecheck{\mfm}_2(\CY_1(\gamma_{21}),\CY_1(\gamma_{10}))$. Those of type \textbf{A4} contribute to $\CY_1\circ\widehat{\mfm}_2^+(\gamma_{21},\gamma_{10})$ and finally those of types \textbf{A5} nad \textbf{A6} contribute to $\CY_2(\widehat{\mfm}_1(\gamma_{21}),\gamma_{10})$. This gives the relation \eqref{rel:functor2} for the pair of inputs $(\gamma_{21},\gamma_{10})$. 
\begin{figure}[ht]  
	\labellist
	\footnotesize
	\pinlabel \textbf{A1} at 50 150
	\pinlabel \textbf{A2} at 120 150
	\pinlabel \textbf{A3} at 260 60
	\pinlabel \textbf{A4} at 410 60
	\pinlabel \textbf{A5} at 560 145
	\pinlabel \textbf{A6} at 660 145
	\pinlabel $x_{12}$ at 593 -5
	\endlabellist
	\normalsize
	\begin{center}\includegraphics[width=12cm]{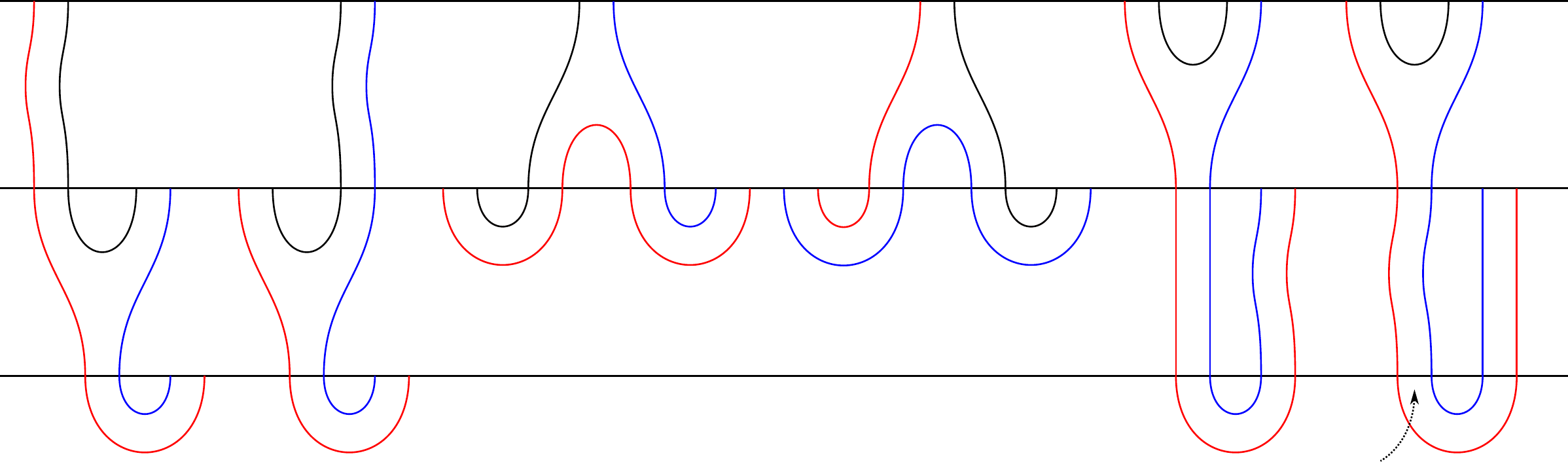}\end{center}
	\caption{Pseudo-holomorphic buildings in the boundary of $1$-dimensional products of moduli spaces of type \textbf{A}.}
	\label{fig:CY2dege1}
\end{figure}
In Figure \ref{fig:CY2dege2} we consider the broken discs in the boundary of $1$-dimensional products of moduli spaces of type \textbf{B}. The buildings \textbf{B1} contribute to $\CY_1\circ\widehat{\mfm}_2^x(\gamma_{21},x_{01})$ when the connecting chord from $\La_2$ to $\La_0$ is the maximum Morse chord $x_{02}$, and to $\widecheck{\mfm}_1\circ\CY_2(\gamma_{21},x_{01})$ otherwise. Those of type \textbf{B2} contribute to $\widecheck{\mfm}_2(\CY_1(\gamma_{21}),\CY_1(x_{01}))$ (none of the connecting chord is a maximum).The buildings of types \textbf{B3} and \textbf{B4} finally contribute to $\CY_2(\widehat{\mfm}_1(\gamma_{21}),x_{01})$. The sum of these contributions gives the relation \ref{rel:functor2} for the pair of inputs $(\gamma_{21},x_{01})$ (observe that some terms in the relation vanish by definition). Similarly this relation can be checked for pairs of inputs $(x_{12},\gamma_{10})$ and $(x_{12},x_{01})$ by considering broken discs in the boundary of $1$-dimensional products of moduli spaces of type \textbf{C} and \textbf{D} respectively, see Figure \ref{fig:CY2dege34}.

\begin{figure}[ht]  
	\labellist
	\footnotesize
	\pinlabel \textbf{B1} at 10 80
	\pinlabel \textbf{B2} at 180 80
	\pinlabel \textbf{B3} at 290 80
	\pinlabel \textbf{B4} at 410 80
	\pinlabel $x_{12}$ at 407 2
	\endlabellist
	\normalsize
	\begin{center}\includegraphics[width=9cm]{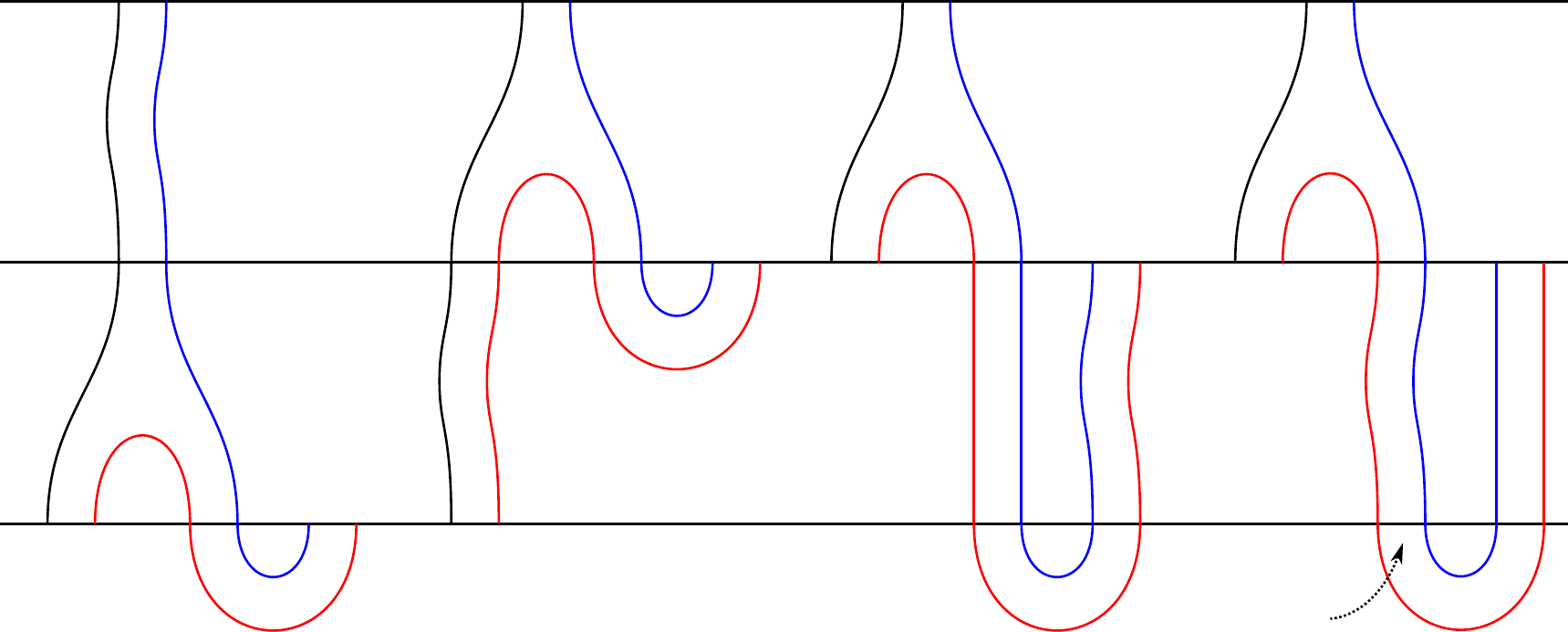}\end{center}
	\caption{Pseudo-holomorphic buildings in the boundary of $1$-dimensional products of moduli spaces of type \textbf{B}.}
	\label{fig:CY2dege2}
\end{figure}
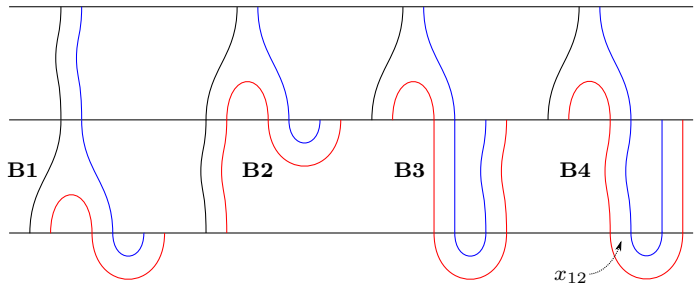
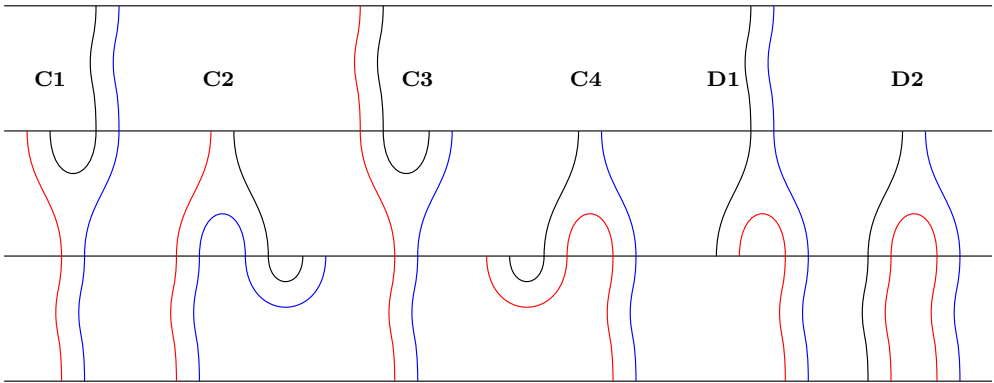
\begin{figure}[ht]  
	\labellist
	\footnotesize
	\pinlabel \textbf{C1} at 30 200
	\pinlabel \textbf{C2} at 140 200
	\pinlabel \textbf{C3} at 270 200
	\pinlabel \textbf{C4} at 380 200
	\pinlabel \textbf{D1} at 470 200
	\pinlabel \textbf{D2} at 590 200
	\endlabellist
	\normalsize
	\begin{center}\includegraphics[width=13cm]{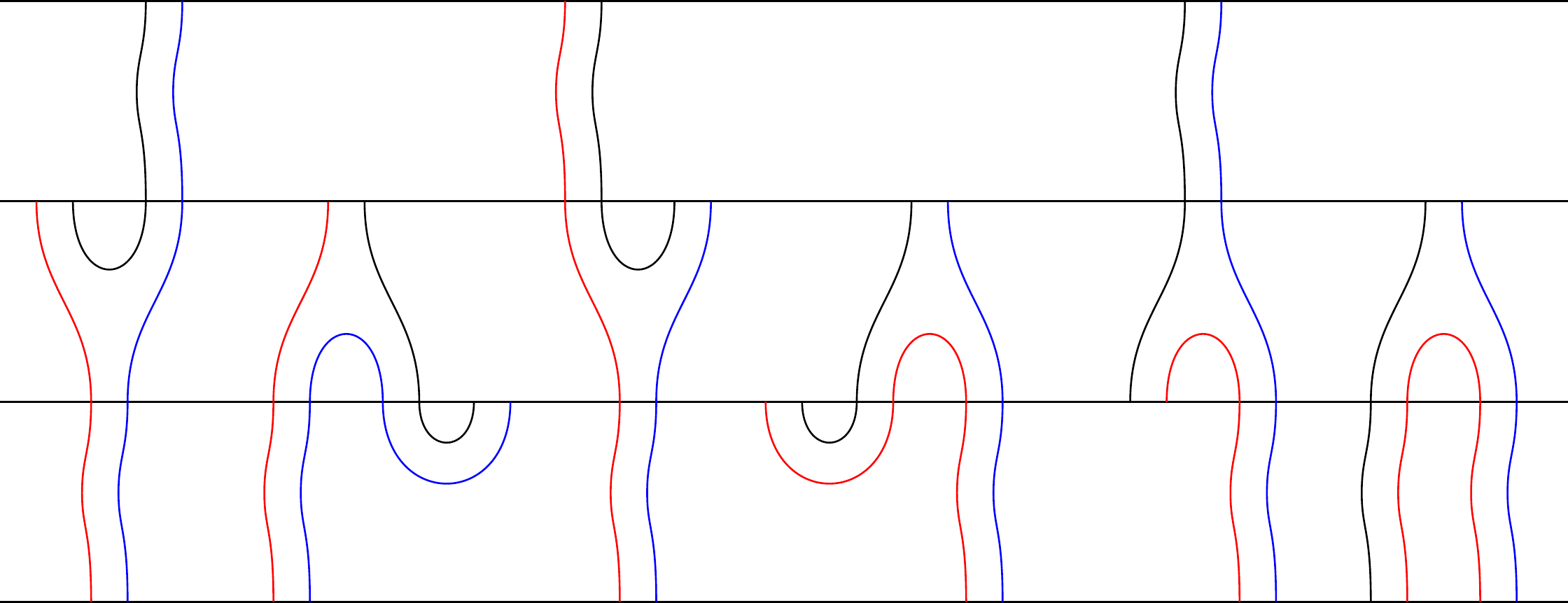}\end{center}
	\caption{Pseudo-holomorphic buildings in the boundary of $1$-dimensional products of moduli spaces of type \textbf{C} and \textbf{D}.}
	\label{fig:CY2dege34}
\end{figure}
\end{proof}

From Theorem \ref{teo:CY2}, we deduce that the product $\widehat{\mfm}_2$ has a unit in homology represented by any cycle in $\widehat{C}_+^{cyc}(\La_0,\La_1)$ which is sent to the minimum Morse Reeb chord $y_{01}$ by the map $\CY_1$.

\subsection{$A_\infty$-structure}\label{sec:A-infty}

We can go further and define for each $d\geq3$ maps $\widecheck{\mfm}_d$ and $\widehat{\mfm}_d$ of degree $2-d$ by counting pseudo-holomorphic discs with boundary on a $(d+1)$-copy of $\La$ (defined in an analogous way as the $2$- and $3$-copies). 
Given a $(d+1)$-copy $\La_0\cup\dots\cup\La_d$ of $\La$, the maps
\begin{alignat*}{1}
	\widecheck{\mfm}_d:\widecheck{C}^{cyc}_-(\La_{d-1},\La_d)\otimes\dots\otimes\widecheck{C}^{cyc}_-(\La_0,\La_1)\to\widecheck{C}^{cyc}_-(\La_0,\La_d)
\end{alignat*}
are given by a count of pseudo-holomorphic discs having mixed negative asymptotics corresponding to inputs and one positive asymptotic which is the output. The map $\widecheck{\mfm}_d$ has degree $2-d$. The fact that these maps satisfy the $A_\infty$-equations goes back to \cite{CKESW,BCh}. In our case we have to take extra care of the coefficients in the C-E algebra but it works exactly the same as in the case $d=2$ we treated in Section \ref{sec:product}.

Now let's define the maps $\widehat{\mfm}_d$.
First, we extend the maps $\ba_1$ and $\bs{\D}_1$ to higher order maps
\begin{alignat*}{1}
	\ba_d,\bs{\D}_d:\RFC^{cyc}(\La_{d-1},\La_d)\otimes\dots\otimes\RFC^{cyc}(\La_0,\La_1)\to\RFC^{cyc}(\La_0,\La_d)
\end{alignat*}
for $d\geq1$ as follows. These maps have degree $2-d$ and were considered in \cite[Section 8]{L2}, but we recall here the definitions. For a $d$-tuple of elements $(c_{d-1}\bs{a}_{d-1},\dots,c_0\bs{a}_0)\in\RFC^{cyc}(\La_{d-1},\La_d)\otimes\dots\otimes\RFC^{cyc}(\La_0,\La_1)$ where $c_j$ are mixed chords and $\bs{a}_j$ words of pure Reeb chords, set
\begin{alignat*}{1}
	&\ba_d(c_{d-1}\bs{a}_{d-1},\dots,c_0\bs{a}_0)=\sum\limits_{\substack{\gamma_{0d}\\\bd_0,\dots,\bd_d}}\#\cM^0_{\La_{0\dots d}}(\gamma_{0d};\bd_0,c_0,\bd_2,\dots,\bd_{d-1},c_{d-1},\bd_d)\cdot\gamma_{0d}\bd_0\bs{a}_0\bd_1\dots\bs{a}_{d-1}\bd_d\\
	&\bs{\D}_d(c_{d-1}\bs{a}_{d-1},\dots,c_0\bs{a}_0)=\sum\limits_{\substack{\gamma_{d0}\\\bd_0,\dots,\bd_d}}\#\cM^0_{\La_{0\dots d}}(\gamma_{d0};\bd_0,c_0,\bd_2,\dots,\bd_{d-1},c_{d-1},\bd_d)\cdot\gamma_{d0}\bd_0\bs{a}_0\bd_1\dots\bs{a}_{d-1}\bd_d
\end{alignat*}
Observe that for the map $\ba_d$, the mixed chord in the output is a positive asymptotic of the pseudo-holomorphic discs considered to define it, while for $\bs{\D}_d$ it is a negative asymptotic. Thus, for energy reasons $\bs{\D}_d(c_{d-1}\bs{a}_{d-1},\dots,c_0\bs{a}_0)$ is automatically $0$ if for all $0\leq j\leq d-1$ we have $c_j\in C(\La_{j},\La_{j+1})$.
\begin{nota}
	We will denote $\ba_d^x$ the component of $\ba_d$ which takes values in $\langle x_{0d}\rangle_{\Ac-\Ac}^{cyc}$, and $\ba_d^\vee$ for the component which takes values in $\widecheck{C}_-^{cyc}(\La_0,\La_d)$.
\end{nota}
We make the following observations:
\begin{itemize}
	\item as ungraded maps, the map $\CY_1$ is equal to the restriction to $\widehat{C}_+^{cyc}$ of the map $\ba_1^\vee$, namely it is defined by a count of bananas with two positive asymptotics.
	\item the maps $\widecheck{\mfm}_d$ are equal to the restriction of $\ba_d^\vee$ to $\widecheck{C}_-^{cyc}(\La_{d-1},\La_d)\otimes\dots\otimes\widecheck{C}_-^{cyc}(\La_0,\La_1)$.
\end{itemize}
With these notations
we can also rewrite the differential $\widehat{\mfm}_1$ and the product $\widehat{\mfm}_2$ as follows (we now drop again the full notation for elements, removing the words of pure Reeb chords):
\begin{alignat*}{1}
	&\widehat{\mfm}_1(\gamma_{10})=\bs{\D}_1(\gamma_{10})+\ba_1^x(\gamma_{10})\\
	&\widehat{\mfm}_2(\gamma_{21},\gamma_{10})=\ba_2^x\big(\CY_1(\gamma_{21}),\gamma_{10}\big)+\bs{\D}_2\big(\CY_1(\gamma_{21}),\gamma_{10}\big)\\
	&\widehat{\mfm}_2(\gamma_{21},x_{01})=\ba_2^x\big(\CY_1(\gamma_{21}),x_{01}\big)\\
	&\widehat{\mfm}_2(x_{12},\gamma_{10})=\ba_2^x\big(\CY_1(x_{12}),\gamma_{10}\big)+\bs{\D}_2\big(\CY_1(x_{12}),\gamma_{10}\big)
\end{alignat*}
Note that $\bs{\D}_2\big(\CY_1(\gamma_{21}),x_{01}\big)=0$ for action reasons, as well as $\ba_2^x\big(x_{12},\CY_1(x_{01})\big)=\bs{\D}_2\big(\CY_1(x_{12}),x_{01}\big)=0$. So we can write a compact formula for the product
\begin{alignat*}{1}
	\widehat{\mfm}_2=\big(\ba_2^x+\bs{\D}_2\big)\big(\CY_1\otimes\id\big)
\end{alignat*}
Finally, observe that the map $\CY_2$ can be rewritten $\CY_2=\ba_2^\vee\big(\CY_1\otimes\id\big)$.
\begin{rem}
	Very rigorously, the map $\widehat{\mfm}_1$ has domain and target the $\Z_2$-module $\widehat{C}_+^{cyc}(\La_0,\La_1)$, so it is actually equal to the sum of the \textit{shifted by 1} restrictions of the maps $\bs{\D}_1$ and $\ba_1^x$ (these restrictions have domain $\widehat{C}_+^{cyc}(\La_0,\La_1)[-1]$). So for the formulas for $\widehat{\mfm}_1, \widehat{\mfm}_2$, $\CY_1,\CY_2$ as well as for the maps we define below, the reader should consider these equalities as equalities of ungraded maps.
\end{rem}
We extend now these formulas, i.e. we define maps
\begin{alignat*}{1}
	&\widehat{\mfm}_d:\widehat{C}_+^{cyc}(\La_{d-1},\La_d)\otimes\dots\otimes\widehat{C}_+^{cyc}(\La_0,\La_1)\to\widehat{C}_+^{cyc}(\La_0,\La_d)\\
	&\CY_d:\widehat{C}_+^{cyc}(\La_{d-1},\La_d)\otimes\dots\otimes\widehat{C}_+^{cyc}(\La_0,\La_1)\to\widecheck{C}_-^{cyc}(\La_0,\La_d)
\end{alignat*}
for a $(d+1)$-copy of $\La$, by
\begin{alignat}{1}
	&\widehat{\mfm}_d=\sum\limits_{j=2}^d\sum_{\substack{1\leq i_2,\dots,i_{j}\leq d-1\\i_2+\dots+i_{j}=d-1}}\big(\ba_j^x+\bs{\D}_j\big)\Big(\CY_{i_{j}}\otimes\dots\otimes\CY_{i_2}\otimes\id\Big)\label{def:mfm}\\
	&\CY_d=\sum\limits_{j=2}^d\sum_{\substack{1\leq i_2,\dots,i_{j}\leq d-1\\i_2+\dots+i_{j}=d-1}}\ba_j^\vee\big(\CY_{i_{j}}\otimes\dots\otimes\CY_{i_2}\otimes\id\big)\label{def:CY}
\end{alignat}
Observe that $\widehat{\mfm}_d$ is of degree $2-d$ while $\CY_d$ is of degree $1-d$.
\begin{teo}\label{teo:inf}
	Let $\La_0\cup\dots\cup\La_d$ be a $(d+1)$-copy of $\La_0$. Then for any $1\leq k\leq d$ and any $(k+1)$-tuple of integers $0\leq s_0<\dots<s_k\leq d$ we have:
	\small
	\begin{alignat}{1}
		&\sum_{m=1}^k\sum_{n=0}^{k-m}\widehat{\mfm}_{k-m+1}\big(\id^{\otimes k-m-n}\otimes\widehat{\mfm}_m\otimes\id^{\otimes n}\big)=0\label{prod_inf}\\
		&\sum_{r=1}^k\sum_{\substack{t_1,\dots,t_r\\t_1+\dots+t_r=k}}\widecheck{\mfm}_r\big(\CY_{t_r}\otimes\dots\otimes\CY_{t_1}\big)+\sum_{m=1}^k\sum_{n=0}^{k-m}\CY_{k-m+1}\big(\id^{\otimes k-m-n}\otimes\widehat{\mfm}_m\otimes\id^{\otimes n}\big)=0\label{fun_inf}
	\end{alignat}
\normalsize
where
\begin{itemize}
	\item $\widehat{\mfm}_m$ has domain $\widehat{C}_+^{cyc}(\La_{s_{n+m-1}},\La_{s_{n+m}})\otimes\dots\otimes\widehat{C}_+^{cyc}(\La_{s_n},\La_{s_{n+1}})$,
	\item $\widehat{\mfm}_{k-m+1}$ and $\CY_{k-m+1}$ have domain
	$$\widehat{C}_+^{cyc}(\La_{s_{k-1}},\La_{s_k})\otimes\dots\widehat{C}_+^{cyc}(\La_{s_{n+m}},\La_{s_{n+m+1}})\otimes\widehat{C}_+^{cyc}(\La_{s_{n-1}},\La_{s_n})\otimes\dots\otimes\widehat{C}_+^{cyc}(\La_{s_1},\La_{s_0}),$$
	\item if we denote $\tau_j=\sum_{i=1}^jt_i$, then $\CY_{t_j}$ has domain
	$$\widehat{C}_+^{cyc}(\La_{s_{\tau_j-1}},\La_{s_{\tau_j}})\otimes\dots\otimes\widehat{C}_+^{cyc}(\La_{s_{\tau_{j-1}}},\La_{s_{\tau_{j-1}+1}}),$$
	and $\widecheck{\mfm}_r$ has domain 
	$$\widecheck{C}_-^{cyc}(\La_{s_{\tau_{r-1}}},\La_{s_{\tau_r}})\otimes\dots\otimes\widecheck{C}_-^{cyc}(\La_{s_{\tau_1}},\La_{s_{\tau_2}})\otimes\widecheck{C}_-^{cyc}(\La_{s_{\tau_1}},\La_{s_{0}}),$$
\end{itemize}
\end{teo}
To simplify notations, in the following we will assume that the $(k+1)$-tuple of integers $(s_0,\dots,s_k)$ is $(0,\dots,k)$.
In order to prove Theorem \ref{teo:inf} we will use results proved in \cite{L2} about the maps $\ba_d$ and $\bs{\D}_d$ that we recall now.
\begin{lem}\cite[Lemma 3 and Lemma 5]{L2}\label{lem:baD}
	Let $\La_0\cup\dots\cup\La_d$ be a $(d+1)$-copy of $\La_0$. Then for any $1\leq k\leq d$ we have:
	\begin{enumerate}
		\item $\sum\limits_{m=1}^k\sum\limits_{n=0}^{k-m}\bs{\mathrm{b}}_{k-m+1}\big(\id^{\otimes d-m-n}\otimes(\bs{\mathrm{b}}_m+\bs{\D}_m)\otimes\id^{\otimes n}\big)=0$
		\item $\sum\limits_{m=1}^k\sum\limits_{n=0}^{k-m}\bs{\D}_{k-m+1}\big(\id^{\otimes d-m-n}\otimes(\bs{\mathrm{b}}_m+\bs{\D}_m)\otimes\id^{\otimes n}\big)=0$
		\end{enumerate}
\end{lem}

\begin{proof}[Proof of Theorem \ref{teo:inf}]
In order to avoid any useless complicated notations and use Lemma \ref{lem:baD} as it is, we prove the theorem ignoring the grading of maps. This means that we will write $\ba_1^\vee$ (restricted to the appropriate module) for $\CY_1$ and $\bs{\D}_1+\ba_1^x$ for $\widehat{\mfm}_1$.
We start by proving \eqref{fun_inf} which is a proof by induction. For $k=1,2$ we have already shown the relation in Lemma \ref{lem:CY1} and the proof of Theorem \ref{teo:CY2} respectively. Now let us prove \eqref{fun_inf} for $k\geq3$ inputs, assuming that the relation holds for any number of inputs less or equal to $k-1$. We denote $\text{LHS}\eqref{fun_inf}$ the left hand side of the equation \eqref{fun_inf}. Using the fact that $\widecheck{\mfm}_r=\ba_r^\vee$ and the formula for the Calabi-Yau map, we have:
\begin{alignat*}{1}
	\text{LHS}\eqref{fun_inf}&=\sum_{j=1}^k\sum_{\substack{t_1,\dots,t_j\geq1\\t_1+\dots+t_j=k}}\ba_j^\vee\big(\CY_{t_j}\otimes\dots\otimes\CY_{t_1}\big)+\ba_1^\vee\circ\widehat{\mfm}_k\\
	&+\sum_{m=1}^{k-1}\sum_{n=0}^{k-m}\sum_{j=2}^{k-m+1}\sum_{\substack{i_2,\dots,i_j\geq1\\i_2+\dots+i_j=k-m}}\ba_j^\vee\Big(\CY_{i_j}\otimes\dots\otimes\CY_{i_2}\otimes\id\Big)\big(\id^{\otimes k-m-n}\otimes\widehat{\mfm}_m\otimes\id^{\otimes n}\big)
\end{alignat*}
We separate the case $n=0$ from the others in the second line and get:
\small
\begin{alignat*}{1}
	&\sum_{j=1}^k\sum_{\substack{t_1,\dots,t_j\geq1\\t_1+\dots+t_j=k}}\ba_j^\vee\big(\CY_{t_j}\otimes\dots\otimes\CY_{t_1}\big)+\ba_1^\vee\circ\widehat{\mfm}_k+\sum_{m=1}^{k-1}\sum_{j=2}^{k-m+1}\sum_{\substack{i_2,\dots,i_j\geq1\\i_2+\dots+i_j=k-m}}\ba_j^\vee\Big(\CY_{i_j}\otimes\dots\otimes\CY_{i_2}\otimes\widehat{\mfm}_m\Big)\\
	&+\sum_{m=1}^{k-1}\sum_{n=1}^{k-m}\sum_{j=2}^{k-m+1}\sum_{\substack{i_2,\dots,i_j\geq1\\i_2+\dots+i_j=k-m}}\ba_j^\vee\Big(\CY_{i_j}\otimes\dots\otimes\CY_{i_2}\otimes\id\Big)\big(\id^{\otimes k-m-n}\otimes\widehat{\mfm}_m\otimes\id^{\otimes n}\big)
\end{alignat*}
\normalsize
We now separate the cases $t_1=1$ and $m=1$ from the others in the first line, and apply a change of variables in the second line (note that for any $2\leq s\leq j$, when $1+\sum\limits_{v=2}^{s-1}i_v\leq n\leq 1+\sum\limits_{v=2}^{s}i_v$, the ``inner'' $\widehat{\mfm}_m$ will be an argument of $\CY_{i_s}$, thus instead of summing over $1\leq n\leq k-m$ we can sum over the variables $s$, $n$ with $2\leq s\leq j$ and $0\leq n\leq i_s-1$) to obtain:
\small
\begin{alignat*}{1}
	&\sum_{j=2}^k\sum_{\substack{t_2,\dots,t_j\geq1\\t_2+\dots+t_j=k-1}}\ba_j^\vee\big(\CY_{t_j}\otimes\dots\otimes\CY_{t_2}\otimes\CY_1\big)+\ba_1^\vee\circ\widehat{\mfm}_k+\sum_{j=2}^{k}\sum_{\substack{i_2,\dots,i_j\geq1\\i_2+\dots+i_j=k-1}}\ba_j^\vee\Big(\CY_{i_j}\otimes\dots\otimes\CY_{i_2}\otimes\widehat{\mfm}_1\Big)\\
	&\sum_{j=1}^k\sum_{\substack{t_1,\dots,t_j\geq1\\t_1\geq2\\t_1+\dots+t_j=k}}\ba_j^\vee\big(\CY_{t_j}\otimes\dots\otimes\CY_{t_1}\big)+\sum_{m=2}^{k-1}\sum_{r=2}^{k-m+1}\sum_{\substack{i_2,\dots,i_r\geq1\\i_2+\dots+i_r=k-m}}\ba_r^\vee\Big(\CY_{i_r}\otimes\dots\otimes\CY_{i_2}\otimes\widehat{\mfm}_m\Big)\\
	&+\sum_{j=2}^{k}\sum_{\substack{i_2,\dots,i_j\geq1\\i_2+\dots+i_j=k-1}}\sum_{s=2}^j\sum_{m=1}^{i_s}\sum_{n=0}^{i_s-m}\ba_j^\vee\Big(\CY_{i_j}\otimes\dots\otimes\CY_{i_s-m+1}\big(\id^{\otimes i_s-m-n}\otimes\widehat{\mfm}_m\otimes\id^{\otimes n}\big)\otimes\dots\otimes\CY_{i_2}\otimes\id\Big)
\end{alignat*}
\normalsize
On the first line the summations are on the same sets of parameters, so we combine them together using that $\CY_1+\widehat{\mfm}_1=\ba_1^\vee+\bs{\D}_1+\ba_1^x=\ba_1+\bs{\D}$. Then we rewrite $\CY_{t_1}$ and $\widehat{\mfm}_m$ on the second line, as well as $\widehat{\mfm}_k$ on the first line, using Formulas \eqref{def:mfm} and \eqref{def:CY}. Finally, we use the relation \eqref{fun_inf} in the third line, which is assumed to hold by induction hypothesis. After all these changes we get:
\begin{alignat*}{1}
	&\sum_{j=2}^k\sum_{\substack{t_2,\dots,t_j\geq1\\t_2+\dots+t_j=k-1}}\ba_j^\vee\Big(\CY_{t_j}\otimes\dots\otimes\CY_{t_2}\otimes\big(\ba_1+\bs{\D}_1\big)\Big)\\
	&+\sum_{r=2}^k\sum_{\substack{t_2,\dots,t_r\geq1\\t_2+\dots+t_r=k-1}}\ba_1^\vee\circ\big(\ba_r^x+\bs{\D}_r\big)\Big(\CY_{t_r}\otimes\dots\otimes\CY_{t_2}\otimes\id\Big)\\
	&+\sum_{j=1}^k\sum_{\substack{t_1,\dots,t_j\geq1\\t_1\geq2\\t_1+\dots+t_j=k}}\sum_{r=2}^{t_1}\sum_{\substack{i_2,\dots,i_r\geq1\\i_2+\dots+i_r=t_1-1}}\ba_j^\vee\Big(\CY_{t_j}\otimes\dots\otimes\CY_{t_2}\otimes\ba_r^\vee\big(\CY_{i_r}\otimes\dots\otimes\CY_{i_2}\otimes\id\big)\Big)\\
	&+\sum_{m=2}^{k-1}\sum_{j=2}^{k-m+1}\sum_{\substack{i_2,\dots,i_j\geq1\\i_2+\dots+i_j=k-m}}\sum_{r=2}^m\sum_{\substack{t_2,\dots,t_r\geq1\\t_2+\dots+t_r=m-1}}\ba_j^\vee\Big(\CY_{i_j}\otimes\dots\otimes\CY_{i_2}\otimes\big(\ba_r^x+\bs{\D}_r\big)\big(\CY_{t_r}\otimes\dots\otimes\CY_{t_2}\otimes\id\big)\Big)\\
	&+\sum_{j=2}^{k}\sum_{\substack{i_2,\dots,i_j\geq1\\i_2+\dots+i_j=k-1}}\sum_{s=2}^j\sum_{r=1}^{i_s}\sum_{\substack{t_1,\dots,t_r\geq1\\t_1+\dots+t_r=i_s}}\ba_j^\vee\Big(\CY_{i_j}\otimes\dots\otimes\ba_r^\vee\big(\CY_{t_r}\otimes\dots\otimes\CY_{t_1}\big)\otimes\CY_{i_{s-1}}\otimes\dots\otimes\CY_{i_2}\otimes\id\Big)
\end{alignat*}
We can sum the three middle lines together, and change variables on the resulting sum as well as on the last line to get:
\begin{alignat*}{1}
	&\sum_{j=2}^k\sum_{\substack{t_2,\dots,t_j\geq1\\t_2+\dots+t_j=k-1}}\ba_j^\vee\Big(\CY_{t_j}\otimes\dots\otimes\CY_{t_2}\otimes\big(\ba_1+\bs{\D}_1\big)\Big)\\
	&+\sum_{j=2}^k\sum_{\substack{t_2,\dots,t_j\geq1\\t_2+\dots+t_j=k-1}}\sum_{r=2}^{j}\ba_{j-r+1}^\vee\Big(\CY_{t_j}\otimes\dots\otimes\CY_{t_{r+1}}\otimes\big(\ba_r+\bs{\D}_r\big)\big(\CY_{t_r}\otimes\dots\otimes\CY_{t_2}\otimes\id\big)\Big)\\
	&+\sum_{j=2}^{k}\sum_{\substack{i_2,\dots,i_j\geq1\\i_2+\dots+i_j=k-1}}\sum_{r=1}^{j-1}\sum_{n=1}^{j-r}\ba_{j-r+1}^\vee\Big(\CY_{i_j}\otimes\dots\otimes\ba_r^\vee\big(\CY_{i_{n+r}}\otimes\dots\otimes\CY_{i_{n+1}}\big)\otimes\CY_{i_n}\otimes\dots\otimes\CY_{i_2}\otimes\id\Big)
\end{alignat*}
Finally, note that adding $\ba_r^x+\bs{\D}_r$ to $\ba_r^\vee$ on the last line doesn't change anything because these terms vanish for degree and energy reasons (observe that curves contributing to $\ba_r^x\big(\CY_{i_{n+r}}\otimes\dots\otimes\CY_{i_{n+1}}\big)$ would have a unique positive asymptotic at a maximum Morse chord and negative asymptotic which have to be Morse chords for action reasons. These negative asymptotics are in the image of the $\CY$ map so can only be minimum Morse chords. For index reasons such rigid discs do not exist). Then we observe that the first line is the case $r=1,n=0$ of the third line while the second line is the case $r\geq1, n=0$ of the third line, so we have
\begin{alignat*}{1}
	\text{LHS}\eqref{fun_inf}=\sum_{j=2}^{k}\sum_{\substack{i_2,\dots,i_j\geq1\\i_2+\dots+i_j=k-1}}\sum_{r=1}^{j}\sum_{n=0}^{j-r}\ba_{j-r+1}^\vee\Big(\id^{\otimes j-r-n}\otimes\big(\ba_r+\bs{\D}_r\big)\otimes\id^{\otimes n}\Big)\Big(\CY_{i_j}\otimes\dots\otimes\CY_{i_2}\otimes\id\Big)
\end{alignat*}
which vanishes by Lemma \ref{lem:baD} (1), and thus we have prove the relation \eqref{fun_inf}. In order to prove the relation \eqref{prod_inf}, we will use the relation \eqref{fun_inf}, which simplifies slightly the notations in the computations below. We have:
\begin{alignat*}{1}
	\text{LHS}\eqref{prod_inf}=\widehat{\mfm}_1\circ\widehat{\mfm}_k+\widehat{\mfm}_k\big(\id^{\otimes k-1}\otimes\widehat{\mfm}_1\big)&+\sum_{m=2}^{k-1}\widehat{\mfm}_{k-m+1}\big(\id^{\otimes k-m}\otimes\widehat{\mfm}_m\big)\\
	&+\sum_{m=1}^{k-1}\sum_{n=1}^{k-m}\widehat{\mfm}_{k-m+1}\big(\id^{\otimes k-m-n}\otimes\widehat{\mfm}_m\otimes\id^{\otimes n}\big)
\end{alignat*}
and we use Formula \eqref{def:mfm} to rewrite $\widehat{\mfm}_j$ where it appear, except the ``inner one'' in the last line. We get:
\small
\begin{alignat*}{1}
	&\text{LHS}\eqref{prod_inf}=\sum_{j=2}^k\sum_{\substack{i_2,\dots,i_j\geq1\\i_2+\dots+i_j=k-1}}\big(\ba_1^x+\bs{\D}_1\big)\big(\ba_j^x+\bs{\D}_j\big)\Big(\CY_{i_j}\otimes\dots\otimes\CY_{i_2}\otimes\id\Big)\\
	&+\sum_{j=2}^k\sum_{\substack{i_2,\dots,i_j\geq1\\i_2+\dots+i_j=k-1}}\big(\ba_j^x+\bs{\D}_j\big)\Big(\CY_{i_j}\otimes\dots\otimes\CY_{i_2}\otimes\big(\ba_1^x+\bs{\D}_1\big)\Big)\\
	&+\sum_{m=2}^{k-1}\sum_{j=2}^{k-m+1}\sum_{\substack{i_2,\dots,i_j\geq1\\i_2+\dots+i_j=k-m}}\sum_{r=2}^m\sum_{\substack{t_2,\dots,t_r\geq1\\t_2+\dots+t_r=m-1}}\big(\ba_j^x+\bs{\D}_j\big)\Big(\CY_{i_j}\otimes\dots\otimes\CY_{i_2}\otimes\big(\ba_r^x+\bs{\D}_r\big)\Big(\CY_{t_r}\otimes\dots\otimes\CY_{t_2}\otimes\id\Big)\Big)\\
	&+\sum_{m=1}^{k-1}\sum_{n=1}^{k-m}\sum_{j=2}^{k-m+1}\sum_{\substack{i_2,\dots,i_j\geq1\\i_2+\dots+i_j=k-m}}\big(\ba_j^x+\bs{\D}_j\big)\Big(\CY_{i_j}\otimes\dots\otimes\CY_{i_2}\otimes\id\Big)\big(\id^{\otimes k-m-n}\otimes\widehat{\mfm}_m\otimes\id^{\otimes n}\big)
\end{alignat*}
\normalsize
Observe that the last line can be written:
\small
\begin{alignat*}{1}
	&\sum_{m=1}^{k-1}\sum_{j=2}^{k-m+1}\sum_{\substack{i_2,\dots,i_j\geq1\\i_2+\dots+i_j=k-m}}\sum_{s=2}^j\sum_{n=0}^{i_s}\big(\ba_j^x+\bs{\D}_j\big)\Big(\CY_{i_j}\otimes\dots\otimes\CY_{i_s}\big(\id^{\otimes i_s-1-n}\otimes\widehat{\mfm}_m\otimes\id^{\otimes n}\big)\otimes\dots\otimes\CY_{i_2}\otimes\id\Big)\\
	&=\sum_{j=2}^{k}\sum_{\substack{i_2,\dots,i_j\geq1\\i_2+\dots+i_j=k-1}}\sum_{s=2}^j\sum_{m=1}^{i_s}\sum_{n=0}^{i_s-m}\big(\ba_j^x+\bs{\D}_j\big)\Big(\CY_{i_j}\otimes\dots\otimes\CY_{i_s-m+1}\big(\id^{\otimes i_s-m-n}\otimes\widehat{\mfm}_m\otimes\id^{\otimes n}\big)\otimes\dots\otimes\CY_{i_2}\otimes\id\Big)
\end{alignat*}
\normalsize
where the right hand side of the equality is obtained after a change of variables. On this last line we can now apply the relation \eqref{fun_inf} to get:
\small
\begin{alignat*}{1}
	&\sum_{j=2}^{k}\sum_{\substack{i_2,\dots,i_j\geq1\\i_2+\dots+i_j=k-1}}\sum_{s=2}^j\sum_{r=1}^{i_s}\sum_{\substack{t_1,\dots,t_r\geq1\\t_1+\dots+t_r=i_s}}\big(\ba_j^x+\bs{\D}_j\big)\Big(\CY_{i_j}\otimes\dots\otimes\CY_{i_{s+1}}\otimes\widecheck{\mfm}_r\big(\CY_{t_r}\otimes\dots\otimes\CY_{t_1}\big)\otimes\dots\otimes\CY_{i_2}\otimes\id\Big)
\end{alignat*}
\normalsize
Using that $\widecheck{\mfm}_r=\ba_r^\vee$, we have:
\small
\begin{alignat*}{1}
	&\text{LHS}\eqref{prod_inf}=\sum_{j=2}^k\sum_{\substack{i_2,\dots,i_j\geq1\\i_2+\dots+i_j=k-1}}\big(\ba_1^x+\bs{\D}_1\big)\big(\ba_j^x+\bs{\D}_j\big)\Big(\CY_{i_j}\otimes\dots\otimes\CY_{i_2}\otimes\id\Big)\\
	&+\sum_{j=2}^k\sum_{\substack{i_2,\dots,i_j\geq1\\i_2+\dots+i_j=k-1}}\big(\ba_j^x+\bs{\D}_j\big)\Big(\CY_{i_j}\otimes\dots\otimes\CY_{i_2}\otimes\big(\ba_1^x+\bs{\D}_1\big)\Big)\\
	&+\sum_{m=2}^{k-1}\sum_{j=2}^{k-m+1}\sum_{\substack{i_2,\dots,i_j\geq1\\i_2+\dots+i_j=k-m}}\sum_{r=2}^m\sum_{\substack{t_2,\dots,t_r\geq1\\t_2+\dots+t_r=m-1}}\big(\ba_j^x+\bs{\D}_j\big)\Big(\CY_{i_j}\otimes\dots\otimes\CY_{i_2}\otimes\big(\ba_r^x+\bs{\D}_r\big)\Big(\CY_{t_r}\otimes\dots\otimes\CY_{t_2}\otimes\id\Big)\Big)\\
	&\sum_{j=2}^{k}\sum_{\substack{i_2,\dots,i_j\geq1\\i_2+\dots+i_j=k-1}}\sum_{s=2}^j\sum_{r=1}^{i_s}\sum_{\substack{t_1,\dots,t_r\geq1\\t_1+\dots+t_r=i_s}}\big(\ba_j^x+\bs{\D}_j\big)\Big(\CY_{i_j}\otimes\dots\otimes\CY_{i_{s+1}}\otimes\ba_r^\vee\big(\CY_{t_r}\otimes\dots\otimes\CY_{t_1}\big)\otimes\dots\otimes\CY_{i_2}\otimes\id\Big)
\end{alignat*}
\normalsize
By a change of variables in the third and fourth lines we obtain:
\small
\begin{alignat*}{1}
	&\text{LHS}\eqref{prod_inf}=\sum_{j=2}^k\sum_{\substack{i_2,\dots,i_j\geq1\\i_2+\dots+i_j=k-1}}\big(\ba_1^x+\bs{\D}_1\big)\big(\ba_j^x+\bs{\D}_j\big)\Big(\CY_{i_j}\otimes\dots\otimes\CY_{i_2}\otimes\id\Big)\\
	&+\sum_{j=2}^k\sum_{\substack{i_2,\dots,i_j\geq1\\i_2+\dots+i_j=k-1}}\big(\ba_j^x+\bs{\D}_j\big)\Big(\CY_{i_j}\otimes\dots\otimes\CY_{i_2}\otimes\big(\ba_1^x+\bs{\D}_1\big)\Big)\\
	&+\sum_{j=2}^{k-1}\sum_{\substack{i_2,\dots,i_j\geq1\\i_2+\dots+i_j=k-1}}\sum_{r=2}^{j-1}\big(\ba_{j-r+1}^x+\bs{\D}_{j-r+1}\big)\Big(\CY_{i_j}\otimes\dots\otimes\CY_{i_{r+1}}\otimes\big(\ba_r^x+\bs{\D}_r\big)\Big(\CY_{i_r}\otimes\dots\otimes\CY_{i_2}\otimes\id\Big)\Big)\\
	&\sum_{j=2}^{k}\sum_{\substack{i_2,\dots,i_j\geq1\\i_2+\dots+i_j=k-1}}\sum_{r=1}^{j-1}\sum_{n=1}^{j-r}\big(\ba_{j-r+1}^x+\bs{\D}_{j-r+1}\big)\Big(\CY_{i_j}\otimes\dots\otimes\ba_r^\vee\big(\CY_{i_{n+r}}\otimes\dots\otimes\CY_{i_{n+1}}\big)\otimes\CY_{i_n}\otimes\dots\otimes\CY_{i_2}\otimes\id\Big)
\end{alignat*}
\normalsize
Finally, observe that adding 
\begin{itemize}
	\item $\ba_j^\vee$ to $\ba_j^x+\bs{\D}_j$ on the first line does nothing as it vanishes for energy reasons (when the output is a long chord) or by cancelling pairs of discs (when the output of $\ba_j^\vee$ is $y$ and the dimension of the Legendrian is $1$).
	\item $\ba_1^\vee$ to $\ba_1^x+\bs{\D}_1$ on the second line contributes nothing more when the output of $\ba_1^\vee$ is a long chord (for energy reasons). Let us check that it also vanishes when the output of $\ba_1^\vee$ is the Morse chord $y$. In such a case, energy arguments imply that the terms
	\begin{alignat*}{1}
		\bs{\D}_j\Big(\CY_{i_j}\otimes\dots\otimes\CY_{i_2}\otimes\ba_1^y\Big)
	\end{alignat*}
vanish. Then we consider the terms:
	\begin{alignat*}{1}
		\ba_j^x\Big(\CY_{i_j}\otimes\dots\otimes\CY_{i_2}\otimes\ba_1^y\Big)
	\end{alignat*}
Pseudo-holomorphic buildings contributing to such terms should contain a rigid disc with a unique positive asymptotic to $x$. For energy reasons, the $j$ negative asymptotics must be Morse chords and in particular can only be $y$'s (by definition of the $\CY$ maps). For index reasons such a rigid disc doesn't exist.
	\item $\ba_r^\vee$ to $\ba_r^x+\bs{\D}_r$ on the third line doesn't contribute either for the same reasons as the previous point.
	\item $\ba_r^x+\bs{\D}_r$ to $\ba_r^\vee$ on the fourth line finally does not change anything either. Indeed, $\bs{\D}_r\big(\CY_{i_{n+r}}\otimes\dots\otimes\CY_{i_{n+1}}\big)$ vanishes for energy reasons and then as before the term $\ba_r^x\big(\CY_{i_{n+r}}\otimes\dots\otimes\CY_{i_{n+1}}\big)$  vanishes also because there is an even number of discs, or no disc at all, with positive asymptotic at $x$ and negative asymptotics only at $y$'s chords. 
\end{itemize} 
Moreover, note that the first line is the case $ r=j, n=0$ of the last line, the second line is the case $r=1, n=0$ of the last line, and the third line is the case $2\leq r\leq j-1, n=0$ of the last line. So we have
\small
\begin{alignat*}{1}
	&\text{LHS}\eqref{prod_inf}=\\
	&\sum_{j=2}^{k}\sum_{\substack{i_2,\dots,i_j\geq1\\i_2+\dots+i_j=k-1}}\sum_{r=1}^{j}\sum_{n=0}^{j-r}\big(\ba_{j-r+1}^x+\bs{\D}_{j-r+1}\big)\Big(\id^{\otimes j-r-n}\otimes\big(\ba_r+\bs{\D}_r\big)\otimes\id^{\otimes n}\Big)\Big(\CY_{i_j}\otimes\dots\otimes\CY_{i_2}\otimes\id\Big)
\end{alignat*}
\normalsize
which vanishes by Lemma \ref{lem:baD} (1) and (2).
\end{proof}

\section{Example: the unknot}

The computation done here is a sub-case of the computation done in \cite[Section 5]{BEE:product}. We nevertheless detail it using our notations.
Let $\La$ be the standard $TB=-1$ unknot and consider a $2$-copy and a $3$-copy, see Figure \ref{fig:unknot12}. We have
\begin{alignat*}{1}
	&\widehat{C}^{cyc}_+(\La_0,\La_1)=\langle a_{10},x_{01}\rangle_{\Ac-\Ac}^{cyc}\quad\text{and}\quad\widecheck{C}_-^{cyc}(\La_0,\La_1)=\langle a_{01},y_{01}\rangle_{\Ac-\Ac}^{cyc}
\end{alignat*}
with $|a_{10}|_{\widehat{C}_+}=0$, $|x_{01}|_{\widehat{C}_+}=2$, $|a_{01}|_{\widecheck{C}_-}=2$ and $|y_{01}|_{\widecheck{C}_-}=0$. Denote $\bs{a}^j=a\dots a$ the word consisting of $j$ times the chord $a$ which is the only Reeb chord of $\La$. We have for all $j\geq0$ $\widehat{\mfm}_1(a_{10}\bs{a}^j)=x_{01}a\bs{a}^j+x_{01}\bs{a}^ja=0$, and $\widehat{\mfm}_1(x_{01})=0$ as well. On the other side $\widecheck{\mfm}_1(a_{01}\bs{a}^j)=0$, for example for degree reasons, and $\widecheck{\mfm}_1(y_{01}\bs{a}^j)=a_{01}a\bs{a}^j+a_{01}\bs{a}^ja=0$.
So both differentials $\widehat{\mfm}_1$ and $\widecheck{\mfm}_1$ vanish implying that the homologies of $\big(\widehat{C}^{cyc}_+(\La_0,\La_1),\widehat{\mfm}_1\big)$ and $\big(\widecheck{C}^{cyc}_-(\La_0,\La_1),\widecheck{\mfm}_1\big)$ are infinite dimensional generated by all words $a_{10}\bs{a}^j, x_{01}\bs{a}^j$ and $a_{01}\bs{a}^j,y_{01}\bs{a}^j$ respectively. In this simple case there is a unique way to define the Calabi-Yau map by degree reasons, but one can also easily see on the figure which are the bananas with a positive asymptotic at $a_{10}$ and the strips with a negative asymptotic at $x_{01}$. This gives:
\begin{alignat*}{1}
	\CY_1(a_{10}\bs{a}^j)=y_{01}\bs{a}^j\quad\text{and}\quad\CY_1(x_{01}\bs{a}^j)=a_{01}\bs{a}^j
\end{alignat*}
We can then use the Lagrangian projection of the $3$-copy to compute the product structures on $\widehat{C}_+^{cyc}$ and $\widecheck{C}_-^{cyc}$.
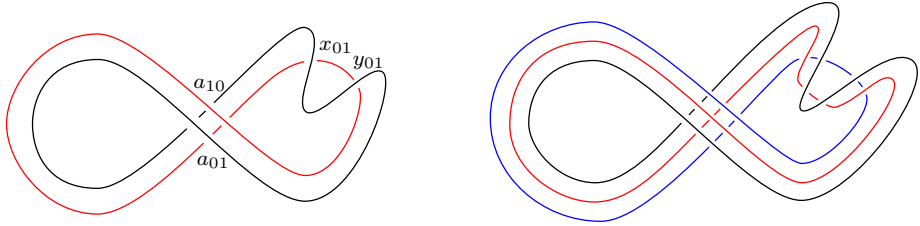
\begin{figure}[ht]  
	\labellist
	\footnotesize
	\pinlabel $x_{01}$ at 191 102
	\pinlabel $y_{01}$ at 210 93
	\pinlabel $a_{01}$ at 120 35
	\pinlabel $a_{10}$ at 118 80
	\endlabellist
	\normalsize
	\begin{center}\includegraphics[width=12cm]{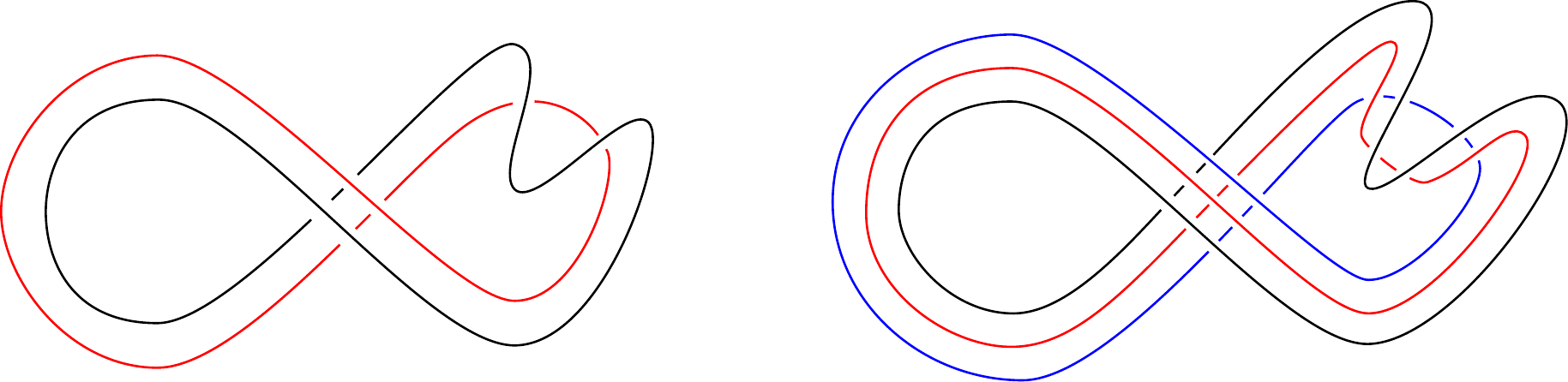}\end{center}
	\caption{Lagrangian projections of the $2$-copy $\La_0\cup\La_1$ on the left, and of the $3$-copy $\La_0\cup\La_1\cup\La_2$ on the right.}
	\label{fig:unknot12}
\end{figure}
The rigid discs asymptotic to generators of $\widecheck{C}_-^{cyc}$ with one positive asymptotic and two negative asymptotics contribute to the product $\widecheck{\mfm}_2$ and one can see that the only one are those giving 
\begin{alignat*}{1}
	\widecheck{\mfm}_2(a_{12}\bs{a}^j,y_{01}\bs{a}^i)=a_{02}\bs{a}^{i+j},\quad\widecheck{\mfm}_2(y_{12}\bs{a}^j,a_{01}\bs{a}^i)=a_{02}\bs{a}^{i+j}\quad\text{and}\quad\widecheck{\mfm}_2(y_{12}\bs{a}^j,y_{01}\bs{a}^i)=y_{02}\bs{a}^{i+j}
\end{alignat*}
which expresses the fact that the minimum Morse Reeb chord acts as a unit, i.e. induces a quasi-isomorphism $\widecheck{C}_-^{cyc}(\La_1,\La_2)\cong\widecheck{C}_-^{cyc}(\La_0,\La_2)$.
For the product $\widehat{\mfm}_2$ we have to find buildings as the one pictured in Figure \ref{fig:secondary_prod}. The only one are those we depicted on Figure \ref{fig:sec_prod_unknot}. They give the following non trivial components of the product:
\begin{alignat*}{1}
	\widehat{\mfm}_2(a_{21}\bs{a}^j,a_{10}\bs{a}^i)=a_{20}\bs{a}^{i+j},\quad\widehat{\mfm}_2(a_{21}\bs{a}^j,x_{01}\bs{a}^i)=x_{02}\bs{a}^{i+j}\quad\text{and}\quad\widehat{\mfm}_2(x_{12}\bs{a}^j,a_{10}\bs{a}^i)=x_{02}\bs{a}^{i+j}
\end{alignat*}
which translates the fact that the mixed $a$ generator acts as a unit for the product $\widehat{\mfm}_2$. The higher order operations $\widecheck{\mfm}_j$ and $\widehat{\mfm}_j$ for $j\geq3$ vanish.
\begin{figure}[ht]  
	\labellist
	\footnotesize
	\pinlabel $a_{01}$ at 115 173
	\pinlabel $a_{10}$ at 55 90
	\pinlabel $a_{20}$ at 140 73
	\pinlabel $y_{12}$ at 110 53
	\pinlabel $x_{01}$ at 217 73
	\pinlabel $x_{02}$ at 240 173
	\pinlabel $y_{12}$ at 245 60
	\pinlabel $a_{21}$ at 300 90
	\pinlabel $a_{10}$ at 375 173
	\pinlabel $x_{02}$ at 425 173
	\pinlabel $a_{12}$ at 430 65
	\pinlabel $x_{12}$ at 398 -10
	\endlabellist
	\normalsize
	\begin{center}\includegraphics[width=10cm]{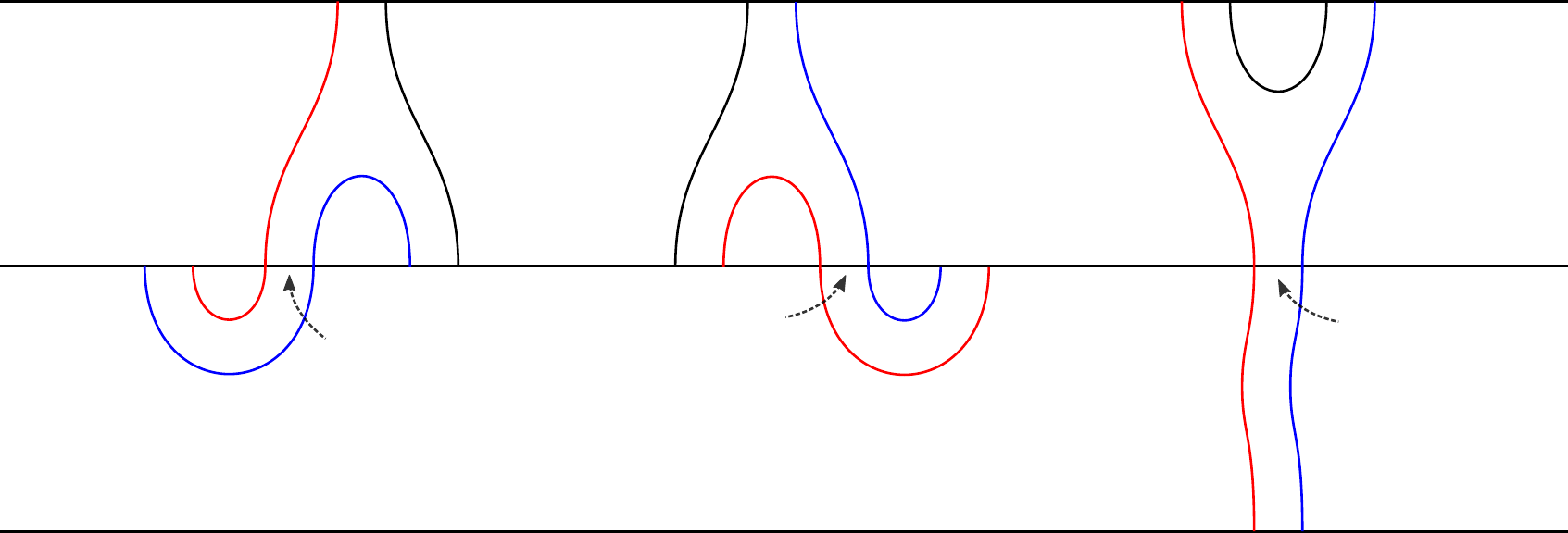}\end{center}
	\caption{Pseudo-holomorphic buildings contributing to the product $\widehat{\mfm}_2$ for the unknot.}
	\label{fig:sec_prod_unknot}
\end{figure}
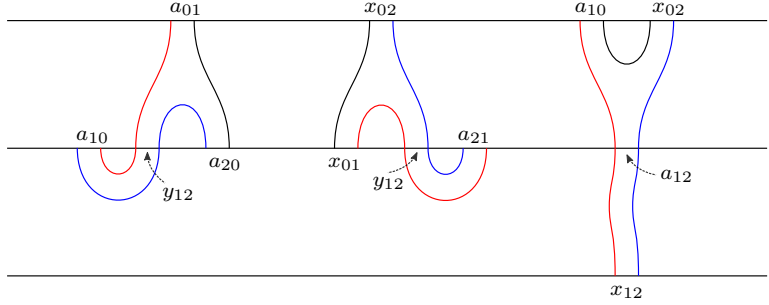

\bibliographystyle{alpha}
\bibliography{ref.bib}
\end{document}